\documentclass[11pt]{article}
\usepackage{amssymb}
\usepackage{graphicx}
\usepackage{setspace}
  \usepackage{paralist}
  \usepackage{longtable}
   \usepackage{multirow}
    \usepackage{rotating}
\usepackage{fancyhdr}
\usepackage[pagewise]{lineno}

\usepackage{amsmath, amsthm, amssymb}
\usepackage{authblk}
\usepackage{graphicx}
\usepackage{float}
\usepackage{hyperref}
\usepackage[margin=1in]{geometry}
\usepackage{comment}
\usepackage{ulem,color}
\allowdisplaybreaks[4]
\numberwithin{equation}{section}
\date{}

\newtheorem{theorem}{Theorem}[section]
\newtheorem{proposition}{Proposition}[section]
\newtheorem{lemma}{Lemma}[section]
\newtheorem{corollary}{Corollary}[section]

\newtheorem{remark}{Remark}[section]

\newcommand{\shortbar}{\mathbin{\text{-}}}

\newcommand{\be}{\begin{equation}}
\newcommand\ee{\end{equation}}
\newcommand\bes{\begin{eqnarray}}
\newcommand\ees{\end{eqnarray}}
\newcommand\bess{\begin{eqnarray*}}
\newcommand\eess{\end{eqnarray*}}

\newcommand\nm{\nonumber}

\title{Wave propagation of a generic non--conservative compressible two--fluid model}

\author{
Zhigang Wu$^{1}$,\ \ Weike Wang$^{2}$,\ \ 
Yinghui Zhang$^{1}$\thanks{Corresponding author (Y.H. Zhang): yinghuizhang@gxnu.edu.cn}}

\begin{document}

\maketitle
\renewcommand{\thefootnote}{\fnsymbol{footnote}}

\footnotetext{1. School of Mathematics and Statistics, Guangxi Normal University, Guilin,
 Guangxi 541004, P.R. China}
\footnotetext{2. School of Mathematical Sciences, CMA-Shanghai and Institute of Natural Science, Shanghai Jiao Tong University, Shanghai 200240, P.R.China}

\pagestyle{myheadings} \markboth{}{Z. WU \& W. Wang \& Y. Zhang:\ \ \  Two-phase Model } \maketitle


\noindent {{\bf  Abstract:}} The generalized Huygens principle for the Cauchy problem of a generic non-conservative compressible two-fluid model in $\mathbb{R}^3$ was established. This work fills a key gap in the theory, as previous results were confined to systems with full conservation laws or ``equivalent" conservative structures from specific compensatory cancellations in Green's function. Indeed, the genuinely non-conservative model studied here falls outside these categories and presents two major analytical challenges. First, its inherent non-conservative structure blocks the direct use of techniques (e.g., variable reformulation) effective for conservative systems.  Second, its Green's function contains a -1-order Riesz operator associated with the fraction densities $\alpha^\pm\rho^\pm$, which generates a so-called Riesz wave-I\!V exhibiting both slower temporal decay and poorer spatial integrability compared to the standard heat kernel, necessitating novel sharp convolution estimates with the Huygens wave. To overcome these difficulties, we develop a framework for precise nonlinear coupling, including interaction of Riesz wave-I\!V and Huygens wave. A pivotal step is extracting enhanced decay rates for the non-conservative pressure terms $\alpha^\pm\nabla P(\rho^\pm)$. By reformulating these terms into a product involving the fraction densities and the specific combination $\bar\rho^-\alpha^+\rho^++\bar\rho^+\alpha^-\rho^-$, and then proving this combination decays faster than the individual densities, we meet the minimal requirements for the crucial convolution estimates. This allows us to close the nonlinear ansatz by constructing essentially new nonlinear estimates.
The success of our analysis stems from the model's special structure, particularly the equal-pressure condition. More broadly, the sharp nonlinear estimates developed herein is applicable to a wide range of non-conservative compressible fluid models.



\section{\leftline {\bf{Introduction.}}}
\setcounter{equation}{0}
\subsection{Background and motivation}
\indent Multi-fluid flows, involving either immiscible phases such as air-water or miscible components forming complex mixtures like emulsions, are ubiquitous in nature and industry. A classic example is interfacial wave motion between separated fluids, traditionally described by a pair of Euler equations coupled through a moving free surface:
\begin{equation}\label{1.1}
\left\{\begin{array}{l}
\partial_{t} \rho_{i}+\operatorname{div}\left(\rho_{i} \mathbf{v}_{i}\right) =0, \quad i=1,2,\\
\partial_{t}\left(\rho_{i} \mathbf{v}_{i}\right)+\operatorname{div}\left(\rho_{i} \mathbf{v}_{i} \otimes \mathbf{v}_{i}\right)+\nabla p_i =-g\rho_{i}
e_3\pm F_D.
\end{array}\right.
\end{equation}
In this formulation, $(\rho_1, \mathbf{v}_1)$ and $(\rho_2, \mathbf{v}_2)$ denote the density and velocity fields of the upper (e.g., air) and lower (e.g., water) fluid, respectively, $p_i$ is the pressure, $-g\rho_i e_3$ represents the gravitational force ($g>0$ is the gravitational acceleration and $e_3$ the vertical unit vector), and $F_D$ accounts for drag. The fluids are separated by a moving interface $z=\eta(x,y,t)$, whose evolution is determined by the kinematic condition
\begin{equation}\partial_t\eta=u_{1,z}-u_{1,x}\partial_x \eta-u_{1, y}\partial_y \eta\label{1.3}\end{equation}
with pressure continuity across the interface.

 However, under conditions such as wave breaking, the interface becomes highly fragmented, leading to a dispersed regime where a two-fluid modeling approach becomes essential. Applying a volume-averaging procedure (see
\cite{Ishii1, Prosperetti} for details) to the microscopic equations yields a macroscopic, interface-free description—the generic compressible two-fluid model. With the inclusion of viscous and capillary effects, this model can be formulated as
\begin{equation}\label{1.4}
\left\{\begin{array}{l}
\partial_{t}\left(\alpha^{\pm} \rho^{\pm}\right)+\operatorname{div}\left(\alpha^{\pm} \rho^{\pm} \mathbf{u}^{\pm}\right)=0, \\
\partial_{t}\left(\alpha^{\pm} \rho^{\pm} \mathbf{u}^{\pm}\right)+\operatorname{div}\left(\alpha^{\pm} \rho^{\pm} \mathbf{u}^{\pm} \otimes \mathbf{u}^{\pm}\right)
+\alpha^{\pm} \nabla P=\operatorname{div}\left(\alpha^{\pm} \tau^{\pm}\right)+\sigma^\pm\alpha^{\pm}
\rho^{\pm}\nabla\Delta(\alpha^{\pm} \rho^{\pm}), \\
P=P^{\pm}\left(\rho^{\pm}\right)=A^{\pm}\left(\rho^{\pm}\right)^{\bar{\gamma}^{\pm}},
\end{array}\right.
\end{equation}
where $\alpha^\pm$, $\rho^\pm$, and $\mathbf{u}^\pm$ denote the volume fractions, densities, and velocities of each phase, sharing a common pressure $P$. Here, $\alpha^++\alpha^-=1$,  the viscous stress tensors are given by 
\begin{equation}\label{1.5}\tau^{\pm} = \mu^{\pm}(\nabla \mathbf{u}^{\pm} + \nabla^{t} \mathbf{u}^{\pm}) + \lambda^{\pm} (\operatorname{div} \mathbf{u}^{\pm}) \mathrm{I},\end{equation}
with viscosity coefficients satisfying $\mu^{\pm} > 0$ and $2\mu^{\pm} + 3\lambda^{\pm} \geq 0$, and $\sigma^\pm>0$ denote the capillary coefficients. 
 For more information about
this model, we refer to \cite{Bear, Brennen1,
Evje3, Evje4, Evje8, Friis1, Raja, Shou, Vasseur, Wen1, Yao2,
Zhang4} and references therein. However, it is well--known that as
far as mathematical analysis of two--fluid model is concerned, there
are many technical challenges. Some of them involve, for example:
\begin{itemize}
\item The corresponding linear system of the model has
multiple eigenvalue, which makes linear analysis of the model become quite difficult
and complicated;

\item Transition to single-phase regions, i.e, regions where the mass
$\alpha^{+} \rho^{+}$ or $\alpha^{-} \rho^{-}$ becomes zero, may
occur when the volume fractions $\alpha^{\pm}$ or the densities
$\rho^{\pm}$ become zero;

\item The system is non-conservative, since the non--conservative terms $\alpha^{\pm} \nabla
P^{\pm}$ are involved in the momentum equations. This brings various
 mathematical difficulties for us to employ methods used
for single phase models to the two--fluid model.

\end{itemize}
\par
Bresch et al. in the seminal work \cite{Bresch1} considered the
generic two-fluid model \eqref{1.4} with the following special
density-dependent viscosities:
\begin{equation}\label{1.7}
\mu^\pm(\rho^\pm)=\mu^\pm \rho^\pm,~~~~~\lambda^\pm(\rho^\pm)=0.
\end{equation}
They obtained the global weak solutions in periodic domain with
$1<\overline{\gamma}^{\pm}< 6$. However, as indicated by themselves,
 their methods rely heavily on the above special
density-dependent viscosities, and particularly cannot handle the
case of constant viscosity coefficients as in \eqref{1.5}. Later,
Bresch--Huang--Li \cite{Bresch2} established the global existence of
weak solutions in one space dimension without capillary effects
(i.e., $\sigma^\pm=0$) when $\overline{\gamma}^{\pm}>1$ by taking
advantage of the one space dimension. Recently, Cui--Wang--Yao--Zhu
\cite{c1} obtained the time--decay rates of classical solutions for
model \eqref{1.4} with the following special density-dependent
viscosities with equal viscosity coefficients, and equal capillary
coefficients:
\begin{equation}\label{1.8}
\mu^\pm(\rho^\pm)=\nu\rho^\pm,~~~~~\lambda^\pm(\rho^\pm)=0,~~~\sigma^+=\sigma^-=\sigma.
\end{equation}
Based on the above special choice for viscosities and capillary
coefficients, they can take a linear combination of model
\eqref{1.4} to reformulate it
 into two $4\times 4$ systems whose linear parts are decoupled with each other and possess
the same dissipation structure as that of the compressible
Navier--Stokes-Korteweg system, and then employ the similar
arguments as in \cite{Bian, Wang-Tan} to prove their main results.
However, since this reformulation played a crucial role in their
analysis, the case of constant viscosities, even if the equal
constant viscosities (i.e.,
$\mu^\pm(\rho^\pm)=\nu,~~\lambda^\pm(\rho^\pm)=\lambda$), cannot be
handled in their settings. 
Very recently, The third author of this article and the coauthors in \cite{Liy} investigated the optimal $L^2$-decay rate for all the derivatives of the solution for this model with general constant viscosities and capillary coefficients based on the 
spectral analysis and energy estimates. More precisely, they verified the following: Assume that $R_{0}^{+}-1,~ R_{0}^{-}-1\in H^{\ell+1}(\mathbb{R}^3)$
 and $\mathbf{u}_{0}^{+},~ \mathbf{u}_{0}^{-}\in H^\ell(\mathbb{R}^3)$ for an integer
$\ell\geq 3$, where $R^\pm=\alpha^\pm\rho^\pm$. Then there exists a constant $\delta_0$ such that if
\begin{equation}\label{1.9}
K_0:=\left\|\left(R_{0}^{+}-1, R_{0}^{-}-1
\right)\right\|_{H^{\ell+1}\cap L^1}+\left\|\left(\mathbf{u}_{0}^{+},
\mathbf{u}_{0}^{-}\right)\right\|_{H^{\ell}\cap L^1} \leq \delta_0,
\end{equation}
 then the Cauchy problem of \eqref{1.4}--\eqref{1.5}  admits a unique
solution $\left(R^{+}, u^{+}, R^{-}, u^{-}\right)$ globally in time, and the solution satisfies the following decay rate for any $t\geq 0,$ and $0\leq k\leq \ell$ that
\begin{equation}\label{1.10}\left\|\nabla^k\left(\rho^{+}-\bar{\rho}^+, \rho^{-}-\bar{\rho}^-\right)(t)\right\|_{H^{\ell-k}}\lesssim
K_0(1+t)^{-\frac{3}{4}-\frac{k}{2}},
\end{equation}
\begin{equation}\label{1.11}\left\|\nabla^{k}\left(\mathbf{u}^{+},
\mathbf{u}^{-}\right)(t)\right\|_{H^{\ell-k}} \lesssim
K_0(1+t)^{-\frac{3}{4}-\frac{k}{2}},
\end{equation}
\begin{equation}\label{1.12}\left\|\nabla^{k}\left[\beta^+(R^+-1)+\beta^-(R^--1)\right](t)\right\|_{H^{\ell-k}} \lesssim
K_0(1+t)^{-\frac{3}{4}-\frac{k}{2}},
\end{equation}
\begin{equation}\label{1.13}\left\|\nabla^k\left(R^{+}-1, R^{-}-1
\right)(t)\right\|_{H^{\ell-k+1}}\lesssim
K_0(1+t)^{-\frac{1}{4}-\frac{k}{2}},
\end{equation}
where $\bar{\rho}^\pm=\rho^\pm(1,1)$ denote equilibrium states of
$\rho^\pm$ respectively, and
$\beta^\pm=\sqrt{\frac{\bar{\rho}^\mp}{{\bar{\rho}}^\pm}}$.

The main purpose of this work is to establish the long-time wave propagation of classical solutions to the generic viscous-capillary two-fluid model \eqref{1.4}–\eqref{1.5}. Specifically, we aim to prove that its solutions obey the so-called {\it generalized Huygens principle}—a property previously established for the one-phase compressible Navier-Stokes equations in Liu-Wang \cite{Lw}. This principle states that the time-asymptotic profile of the solution decomposes into a superposition of a stationary diffusion wave ({D-wave}) of the form $(1+t)^{-\frac{3}{2}}\big(1+\frac{|x|^2}{1+t}\big)^{-\frac{3}{2}}$, and a moving diffusion wave ({H-wave}) of the form $(1+t)^{-2}\big(1+\frac{(|x|-{\rm c}t)^2}{1+t}\big)^{-\frac{3}{2}}$, where ``{\rm c}" is the base sound speed. The {H-wave} originates from the wave-type components in the low-frequency of Green's function. Since the H-wave decays more slowly in $L^p(\mathbb{R}^3)$ for $p<2$ than the {D-wave}, its nonlinear coupling requires “borrowing” one derivative from the nonlinear terms—a step that typically relies on the conservative (divergence) structure of the equations, as seen in earlier works \cite{Ls,Lw} and recent works \cite{Bai,Wu10}. We would like to refer to other related results for the long-time behaviors of the one-phase model in \cite{Guo2,Hoff1,Hoff2, Lihl,L7,Mat1,Mat2,Wang-Tan,Zeng1,Zeng2} and the kinetic equations in \cite{Lihl2,Lihl3,L3,L4,L5,L6}.

Compared with classical one-phase models, the analysis of the two-fluid system faces two basical structural obstacles.
\textbf{First}, the system is inherently {\it non-conservative} due to the pressure terms $\alpha^\pm\nabla P(\rho^\pm)$ in the momentum equations. This prevents the direct use of cancellations and reformulation techniques that are essential in conservative settings.
\textbf{Second}, a detailed spectral analysis of the linearized system reveals that Green's function contains a singular low-frequency term of the type $|\xi|^{-1}e^{-|\xi|^2t}$ in the entries corresponding to the two densities. This term gives rise to a so-called {\it Riesz wave} with pointwise behavior like $(1+t)^{-1}\big(1+\frac{|x|^2}{1+t}\big)^{-1}$ (labeled {Riesz wave-I\!V}), which exhibits slower temporal decay and worse spatial regularity than the standard D-wave. Such Riesz waves also appear in other conservative systems such as the unipolar compressible Navier–Stokes–Poisson equations in Wang and Wu \cite{Wang}, where they verified that the electronic field containing the -1-order operator $\nabla(-\Delta)^{-1}$ will impede the propagation of acoustic wave and ultimately results that its solution does not obey the generalized Huygens principle. On the other hand, although there exists the {H-waves} and {Riesz wave-I\!V} for the bipolar compressible Navier-Stokes-Poisson equations in Wu and Wang \cite{Wu4}, these two waves do not interact with each other when dealing with the nonlinear coupling, since one can rewrite the original system in \cite{Wu4} into two subsystem by using some linear combinations of unknowns such that the {Riesz wave-I\!V} is only confronted with the D-wave (the propagation speed is zero). 

The key innovation of this paper lies in overcoming the simultaneous interaction between the H-wave and the Riesz wave within a {\it non-conservative} framework. Unlike previous models where these waves could be decoupled, here they are intrinsically coupled through the nonlinear terms. To handle this, we develop new pointwise estimates for the nonlinear coupling of distinct wave patterns, particularly the interaction between the Riesz wave and the H-wave. A crucial technical step is the reformulation of the non-conservative pressure term $\alpha^\pm\nabla P$ into a product structure that exposes improved decay properties, allowing us to close the nonlinear ansatz. This approach not only verifies the generalized Huygens principle for the non-conservative two-fluid system but also provides a set of convolution tools applicable to other non-conservative fluid models.

In the following, we first formulate the linearized system and present the Fourier representation of its Green's function. We then explain how the pointwise space-time description of Green's function, combined with refined nonlinear convolution estimates for different wave patterns, leads to the large-time asymptotics of the nonlinear problem.

\subsection{New formulation of system \eqref{1.4} and Main Results}
In this subsection, we devote ourselves to reformulating the system
\eqref{1.4} and stating the main results. The relations between the
pressures of $\eqref{1.4}_3$ implies
\begin{equation}\label{1.9}
{\rm d}P
  =
  s_+^2
 {\rm d}\rho^+
  =
   s_-^2
    {\rm d}\rho^-,
 \quad
  {\rm where}
   \quad
    s_\pm
     :=
      \sqrt{ \frac{{\rm d}P}{{\rm d}\rho^\pm}(\rho^\pm)}.
\end{equation}
Here $s_\pm$ represent the sound speed of each phase respectively.
As in \cite{Bresch1}, we introduce the fraction densities
\begin{equation}\label{1.10}
R^\pm
 =
  \alpha^\pm \rho^\pm,
\end{equation}
which together with the relation $\alpha^++\alpha^-=1$ leads to
\begin{equation}\label{1.11}
{\rm d}
 \rho^+
  =
   \frac{1}{\alpha_+}
    ({\rm d}R^+
       -
        \rho^+
         {\rm d}\alpha^+),
  \quad
   {\rm d}
    \rho^-
     =
      \frac{1}{\alpha_-}
       ({\rm d}R^-
         +
          \rho^-
           {\rm d}\alpha^+).
\end{equation}
From \eqref{1.9}--\eqref{1.10}, we finally get
\begin{equation}\notag
{\rm d}\alpha^+
 =
  \frac{\alpha^-s_+^2}{\alpha^-\rho^+s_+^2+\alpha^+\rho^-s_-^2}
   {\rm d}R^+
    -
     \frac{\alpha^+s_-^2}{\alpha^-\rho^+s_+^2+\alpha^+\rho^-s_-^2}
   {\rm d}R^-.
\end{equation}
Substituting the above equality into \eqref{1.11} yields
\begin{equation}\notag
{\rm d}\rho^+
 =
  \frac{s_-^2}{\alpha^-\rho^+s_+^2+\alpha^+\rho^-s_-^2}
   (\rho^-{\rm d}R^++\rho^+{\rm d}R^-),
\end{equation}
and
\begin{equation}\notag
{\rm d}\rho^-
 =
  \frac{s_+^2}{\alpha^-\rho^+s_+^2+\alpha^+\rho^-s_-^2}
   (\rho^-{\rm d}R^++\rho^+{\rm d}R^-),
\end{equation}
which combined with $(\ref{1.9})$ imply 
\begin{equation}\label{1.12}
{\rm d}P
 =
  \mathcal{C}^2(\rho^- {\rm d} R^+ +\rho^+ {\rm d}R^-),
\end{equation}
where
\begin{equation}\notag
\mathcal{C}^2
 :=
   \frac{s_+^2s_-^2}{\alpha^-\rho^+s_+^2+\alpha^+\rho^-s_-^2},
    \quad
     {\rm and}
      \quad
       s_\pm^2
        =
        \frac{{\rm d}P(\rho^\pm)}{{\rm d}\rho^\pm}
         =
          \tilde{\gamma}^\pm \frac{P(\rho^\pm)}{\rho^\pm}.
\end{equation}
Next, by using the relation: $\alpha^+ +\alpha^- =1$ again, we can
get
\begin{equation}\label{1.13}
\frac{R^+}{\rho^+}
 +
  \frac{R^-}{\rho^-}
   =1,
    \quad
     {\rm and \ therefore}
      \quad
       \rho^-
        =
         \frac{R^-\rho^+}{\rho^+-R^+}.
\end{equation}
From $\eqref{1.4}_3$, we have
\begin{equation}\notag
\varphi(\rho^+, R^+, R^-)
 :=
  P(\rho^+)
   -
    P\left(\frac{R^-\rho^+}{\rho^+-R^+}\right)
     =
      0.
\end{equation}
Consequently, for any given two positive constants $\tilde R^+$ and
$\tilde R^-$, there exists $\tilde \rho^+>\tilde R^+$ such that
\begin{equation}\notag
\varphi(\tilde \rho^+, \tilde R^+, \tilde R^-)=0.
\end{equation}
Differentiating $\varphi$ with respect to $\rho^+$, we get
\begin{equation}\notag
\frac{\partial\varphi}{\partial\rho^+}(\rho^+, R^+, R^-)
 =
  s_+^2+
   s_-^2
    \frac{R^-R^+}{(\rho^+-R^+)^2},
\end{equation}
which implies
\begin{equation}\notag
\frac{\partial\varphi}{\partial\rho^+}(\tilde\rho^+, \tilde R^+,
\tilde R^-)>0.
\end{equation}
Thus, this together with Implicit Function Theorem and \eqref{1.10},
\eqref{1.13} implies that the unknowns $\rho^\pm$, $\alpha^\pm$ and
$\mathcal{C}$ can be given by
\begin{equation}\notag
\rho^\pm
     =
      \varrho^\pm(R^+,R^-),
\qquad
      \alpha^\pm
       =
        \alpha^\pm(R^+,R^-),
         \quad
     {\rm and \ therefore}
      \quad
      \mathcal{C}=\mathcal{C}(R^+, R^-).
\end{equation}
 We refer to [\cite{Bresch1}, pp. 614] for the details.
\par

Note that the non-conservation of system (\ref{1.4}) is mainly from two pressure terms $\alpha^\pm\nabla P^\pm$ in the nonlinear part, since the other nonlinear terms are of the divergence form. As mentioned above, the conservation is critical when deducing the generalized Huygens' principle. To this end, we still need to use the divergence form of the other nonlinear terms when dealing with the nonlinear coupling. Therefore, we consider the following system with respect to the unknowns $(R^\pm,\ \mathbf{m}^\pm=R^\pm u^\pm)$:
\begin{equation}\label{1.14}
\left\{\begin{array}{l}
\partial_{t} R^{\pm}+\operatorname{div}\mathbf{m}^\pm=0, \\
\partial_{t}\mathbf{m}^++\operatorname{div}\big(\frac{\mathbf{m}^+\otimes\ \!\mathbf{m}^+}{R^+}\big)+\alpha^{+} \mathcal{C}^{2}\big[\rho^{-} \nabla
R^{+}+\rho^{+}\nabla R^{-}\big] \\
\hspace{1.8cm}=\operatorname{div}\big\{\alpha^{+}\big[\mu^{+}\big(\nabla
\frac{\mathbf{m}^+}{R^+}+\nabla^{T} \frac{\mathbf{m}^+}{R^+}\big)
+\lambda^{+} \operatorname{div} \frac{\mathbf{m}^+}{R^+} \mathbb{I}\big]\big\}+\sigma^+R^+\nabla\Delta R^+, \\
\partial_{t}\mathbf{m}^-+\operatorname{div}\big(\frac{\mathbf{m}^-\otimes\ \!\mathbf{m}^-}{R^-}\big)+\alpha^{-} \mathcal{C}^{2}\big[\rho^{-} \nabla
R^{+}+\rho^{+}\nabla R^{-}\big] \\
\hspace{1.8cm}=\operatorname{div}\big\{\alpha^{-}\big[\mu^{-}\big(\nabla
\frac{\mathbf{m}^-}{R^+}+\nabla^{T} \frac{\mathbf{m}^-}{R^-}\big)
+\lambda^{-} \operatorname{div} \frac{\mathbf{m}^-}{R^-} \mathbb{I}\big]\big\}+\sigma^-R^-\nabla\Delta R^-,
\end{array}\right.
\end{equation}
subject to the initial condition
\begin{equation}\label{1.15} (R^{+}, \mathbf{m}^{+}, R^-, \mathbf{m}^{-})(x,
0)=(R_{0}^{+}, \mathbf{m}_{0}^{+}, R_{0}^+, \mathbf{m}_{0}^{-})(x)\rightarrow(\bar
R^{+}, \mathbf{0}, \bar R^{-}, \mathbf{0}) \quad
\hbox{as}\quad |x|\rightarrow\infty,
\end{equation}
where two positive constants  $\bar{R}^{+}$ and $\bar{R}^{-}$ denote
the background doping profile, and in the present paper are taken as
1 for simplicity.

\medskip
Now, we are in a position to state our main result.
\smallskip
\begin{theorem}\label{l 1.1}Assume that the initial data $(R_{0}^{+}-1,~ R_{0}^{-}-1)\in H^8(\mathbb{R}^3)$
 and $(\mathbf{m}_{0}^{+},~ \mathbf{m}_{0}^{-})\in H^7(\mathbb{R}^3)$, and has the compact support or satisfies the pointwise assumption for $|\alpha|\leq 4$ and $|\alpha'|\leq 3$ that
\begin{equation}\label{1.16}
|\partial_x^\alpha(R_0^\pm-1)|+|\partial_x^{\alpha'}\mathbf{m}_0^\pm)|\leq \varepsilon_0(1+|x|^2)^{-r},\ \ r>\frac{19}{10},
\end{equation}
then the Cauchy problem \eqref{1.14}--\eqref{1.15} admits a unique
solution $\left(R^{+}, \mathbf{m}^{+}, R^{-}, \mathbf{m}^{-}\right)$ globally in time
and the solution obeys the following pointwise space-time descriptions
\begin{equation}\label{1.17}
|R^\pm-1|\lesssim (1+t)^{-1}\Big(1+\frac{|x|^2}{1+t}\Big)^{-1}+(1+t)^{-\frac{3}{2}}\Big(1+\frac{(|x|-{\rm c}t)^2}{1+t}\Big)^{-1},
\end{equation}
\begin{equation}\label{1.18}
|\partial_x R^\pm|+|\mathbf{m}^\pm|\lesssim (1+t)^{-\frac{3}{2}}\Big(1+\frac{|x|^2}{1+t}\Big)^{-1}+(1+t)^{-2}\Big(1+\frac{(|x|-{\rm c}t)^2}{1+t}\Big)^{-1},
\end{equation}
where the propagation speed ${\rm c}$ is 
\begin{equation}\label{1.19}
{\rm c}=\sqrt{\frac{P'(\bar\rho^+)P'(\bar\rho^-)}{\alpha^-\bar\rho^+P'(\bar\rho^+)+\alpha^+\bar\rho^-P'(\bar\rho^-)}\Big(\frac{\bar{\rho}^-}{\bar{\rho}^+}+\frac{\bar{\rho}^+}{\bar{\rho}^-}\Big)},\ {\rm with}\ \bar\rho^\pm=\rho^\pm(1,1).
\end{equation}
Moreover, we have the following refined pointwise description for the linear combination of two fraction densities
\begin{align}\label{1.20}
|\bar\rho^-(R^+-1)+\bar{\rho}^+(R^--1)|
\lesssim (1+t)^{-\frac{3}{2}}\Big(1+\frac{|x|^2}{1+t}\Big)^{-1}+(1+t)^{-2}\Big(1+\frac{(|x|\!-\!{\rm c}t)^2}{1+t}\Big)^{-1},
\end{align}
which results in the pointwise description for the densities $(\rho^+,\rho^-)$ as
\begin{align}\label{1.21}
|(\rho^+-\bar\rho^+,\rho^--\bar{\rho}^-)|
\lesssim (1+t)^{-\frac{3}{2}}\Big(1+\frac{|x|^2}{1+t}\Big)^{-1}+(1+t)^{-2}\Big(1+\frac{(|x|\!-\!{\rm c}t)^2}{1+t}\Big)^{-1}.
\end{align}
\end{theorem}

\begin{remark}It is obvious that the propagation speed for the two-phase fluid model (\ref{1.4}) is similar to $\sqrt{P'(\bar\rho)}$ for the one-phase fluid model (the compressible Navier-Stokes system) in \cite{Ls,Lw}  when $\bar\rho^+=\bar\rho^-$. 
\end{remark}

\begin{remark} The component like $\frac{1}{|\xi|}e^{-|\xi|^2t}$ in the low frequency of Green's function generates the Riesz wave-I\!V: $(1+t)^{-1}(1+\frac{|x|^2}{1+t})^{-1}$ in the pointwise space-time descriptions of both Green's function and  the fraction densities $R^\pm$, which  clearly has slower temporal decay than the heat kernel.
\end{remark}

\begin{remark}Although linear combinations of the fraction densities are applied, the model (\ref{1.4}) retains an essentially non-conservative structure. Consequently, the profile $\big(1+\frac{(|x|-{\rm c}t)^2}{1+t}\big)^{-1}$ of the Huygens wave in Theorem \ref{l 1.1} is weaker than the $\big(1+\frac{(|x|-{\rm c}t)^2}{1+t}\big)^{-\frac{3}{2}}$ for conservative systems like the Navier-Stokes equations \cite{Ls,Lw} and their bipolar Poisson-coupled counterparts \cite{Wu4,Wu5}. This weaker spacial-decay stems from the non-divergence pressures, however, it is exactly consistent with the Riesz wave-I\!V: $(1+t)^{-1}(1+\frac{|x|^2}{1+t})^{-1}$ mentioned above for the fraction densities.
\end{remark}

\begin{remark}The refined estimate (\ref{1.21}) can be traced to two sources: the linear combination inherent in Green's function (as seen in (\ref{1.20})), and the relation $(\rho^+-\bar\rho^+,\rho^--\bar{\rho}^-)\sim \bar\rho^-(R^+-1)+\bar{\rho}^+(R^--1)$ obtained via the Mean Value Theorem. This estimate is crutial for closing the nonlinear coupling ansatz, as it supplies the necessary decay properties to handle the problematic non-conservative terms $\alpha^\pm\nabla P$.
\end{remark}

\begin{remark}As a corollary, one can have the following $L^p$-estimate of the unknowns:
\begin{align}\label{1.22}
&\|R^\pm-1\|_{L^p(\mathbb{R}^3)}\lesssim \left\{\!\!\begin{array}{cc}
  (1+t)^{-(\frac{3}{2}-\frac{5}{2p})}, & \frac{3}{2}<p\leq2, \\
   (1+t)^{-(1-\frac{3}{2p})}, & 2\leq p\leq\infty,\end{array}\right.\\
&\|\rho^\pm-\bar{\rho}^\pm\|_{L^p(\mathbb{R}^3)}+\|\mathbf{m}^\pm\|_{L^p(\mathbb{R}^3)}\lesssim \left\{\!\!\begin{array}{cc}
  (1+t)^{-(2-\frac{5}{2p})}, & \frac{3}{2}<p\leq2, \\
   (1+t)^{-(\frac{3}{2}-\frac{3}{2p})}, & 2\leq p\leq\infty.
\end{array}\right.
\end{align}
\end{remark}

\begin{remark} Due to its importance in mathematics and physics, there also have been some efforts
in studying the related model (\ref{1.4}) without capillary forces effects, that is, $\sigma^\pm=0$ in the system (\ref{1.4}). As for this model, on one hand, it is relatively easy to deduce the global existence and temporal decay rate of the strong (or smooth) solutions under suitably small perturbation of initial data when two pressures satisfy $P^{+}\left(\rho^{+}\right)-P^{-}\left(\rho^{-}\right)=f(\alpha^{-}
\rho^{-})\neq0$ with a constructed condition of $f$, see the details in \cite{Evje9}. The pointwise space-time behavior of the case considered in \cite{Evje9} will be presented in a forthcoming paper, where a different phenomenon will be verified.
\end{remark}

At last, when the two pressures are the same (the common pressure) and there are not capillary forces effects, it is a more challenging open problem due to the fact that the system is partially dissipative. More recently, the first result on the global existence and the temporal decay rate of the strong (or smooth) solutions under suitably small perturbation of initial data  was established in \cite{Wug1} in Sobolev framework and \cite{Shou} in critical Besov framework, respectively. As for this case, the zero eigenvalues of Green's matrix result in the non-dissipative properties for densities and fraction densities. Thus, it is rather hard to deduce the similar results as in the present paper due to some nonlinear terms containing the densities (or fraction densities) only exhibiting the same decay rate as the linear term, since there are no ready-made convolution
estimates available on the nonlinear interaction of a Huygens wave and a non-decay diffusion wave up to now. This will be considered in future.

\smallskip
\smallskip

\indent Now, we outline the proof strategy for Theorem \ref{l 1.1} and highlight the main difficulties and key techniques. Note that only the two pressure terms are not of the divergence form in the nonlinear terms, to utilize the divergence form of nonlinear terms except these two terms in nonlinear analysis, we consider the unknowns $(R^\pm,\ \mathbf{m}^\pm=R^\pm \mathbf{u}^\pm)$ instead of the unknowns $(R^\pm,\ \mathbf{u}^\pm)$ in 
Li $et.\ al.$ \cite{Liy}. This transformation alters the structure of the nonlinear terms. To see this, by taking $n^{\pm}=R^{\pm}-1$,
one can write the corresponding linearized system of model \eqref{1.14}
in terms of the variables $(n^+, \mathbf{m}^+, n^-, \mathbf{m}^-)$:
\begin{equation}\label{1.27}
\left\{\begin{array}{l}
\partial_{t} n^++\operatorname{div}\mathbf{m}^+=0, \\
\partial_{t}\mathbf{m}^{+}+\beta_1\nabla n^++\beta_2\nabla
n^--\nu^+_1\Delta \mathbf{m}^+-\nu^+_2\nabla\operatorname{div} \mathbf{m}^+-\sigma^+\nabla\Delta n^+=F_1, \\
\partial_{t} n^-+\operatorname{div}\mathbf{m}^-=0, \\
\partial_{t}\mathbf{m}^{-}+\beta_3\nabla n^++\beta_4\nabla
n^--\nu^-_1\Delta \mathbf{m}^--\nu^-_2\nabla\operatorname{div} \mathbf{m}^--\sigma^-\nabla\Delta n^-=F_2, \\
\end{array}\right.
\end{equation}
where $\nu_{1}^{\pm}=\frac{\mu^{\pm}}{\bar{\rho}^{\pm}}$, 
$\nu_{2}^{\pm}=\frac{\mu^{\pm}+\lambda^{\pm}}{\bar{\rho}^{\pm}}>0$, and the coefficients $\beta_i$ with $i=1,2,3,4$ related to the pressure terms satisfy that
\begin{equation}\label{1.8}
\beta_{1}=\frac{\mathcal{C}^{2}(1,1)
\bar{\rho}^{-}}{\bar{\rho}^{+}},\ \ 
\beta_{2}=\beta_{3}=\mathcal{C}^{2}(1,1),\ \ 
\beta_{4}=\frac{\mathcal{C}^{2}(1,1)
\bar{\rho}^{+}}{\bar{\rho}^{-}}. 
\end{equation}
We shall also rewrite the nonlinear terms $F_1$ and $F_2$ as 
\begin{equation}\label{1.29}
F_1=-\underbrace{[\alpha^+\nabla P-(\beta_1\nabla n^++\beta_2\nabla n^-)]}_{Q_1}+\cdots,\ F_2=-\underbrace{[\alpha^-\nabla P-(\beta_3\nabla n^++\beta_4\nabla n^-)]}_{Q_2}+\cdots,
\end{equation}
 where ``$\cdots$" are the terms of the divergence form in $F_1$ and $F_2$, and they will not bring us any new difficulty. The detailed expressions of $F_1$ and $F_2$ will be given in the next section.
 
 The main difficulties stem from three interconnected features of the two-phase model \eqref{1.4}:  (i) the presence of a nonlocal component of the form $\frac{1}{|\xi|}e^{-|\xi|^2t}$ in the low frequency of Green function; (ii) the existence of Huygens waves in the low frequency of its Green's function; (iii) the inherent non-conservative structure of the system. Superficially, this bears a resemblance to the bipolar Navier-Stokes-Poisson system studied by Wu and Wang \cite{Wu4}, which also exhibits Huygens waves and a non-conservative term arising from the nonlocal electric field $\nabla\phi=-\nabla(-\Delta)^{-1}\rho$.
 While the analysis in \cite{Wu4} relied on rewriting the system via linear combinations of unknowns to exploit a partially conservative structure, coupled with the assumption $\nabla\phi_0\in L^1$ to enhance decay, such techniques are not directly transferable to the two-phase model \eqref{1.4}. The non-conservative pressure terms $\alpha^\pm\nabla P$ in our model are intrinsically coupled and cannot be decoupled through a linear transformation, presenting a fundamental obstacle that necessitates a novel analytical framework.

The main analytical difficulties arise from the low-frequency behavior of Green's function for the compressible part, analyzed via Hodge decomposition. The slower pointwise decay of the fraction densities in (\ref{1.17}) can be attributed to the following two key reasons.

 First, the low-frequency spectrum contains two distinct real eigenvalues $\lambda_3$ and $\lambda_4$:
\begin{equation}\label{1.30}\begin{cases}
\lambda_3=\frac{-(\beta_1\nu^-+\beta_4\nu^+)+\sqrt{(\beta_1\nu^-+\beta_4\nu^+)^2-4(\beta_1+\beta_4)(\beta_1\sigma^-+\beta_4\sigma^+)}}{2(\beta_1+\beta_4)}|\xi|^2+\mathcal O(|\xi|^4),\\
\lambda_4=\frac{-(\beta_1\nu^-+\beta_4\nu^+)-\sqrt{(\beta_1\nu^-+\beta_4\nu^+)^2-4(\beta_1+\beta_4)(\beta_1\sigma^-+\beta_4\sigma^+)}}{2(\beta_1+\beta_4)}|\xi|^2+\mathcal O(|\xi|^4),\\
\end{cases}\end{equation}
which directly influence the diffusion structure and lead to weakened temporal decay. Their corresponding project operators $P^3$ and $P^4$ are given by
\begin{equation*}\begin{split}\label{1.31}P^3(\xi)=&\begin{pmatrix}
\frac{\beta_1\beta_4(\nu^+-\nu^-)}{(\beta_1+\beta_4)\mathcal
R}-\frac{\beta_4\tilde r_4}{\mathcal R}&-\frac{\beta_4}{\mathcal
R|\xi|}&\frac{\beta_2\beta_4(\nu^+-\nu^-)}{(\beta_1+\beta_4)\mathcal
R}+\frac{\beta_2\tilde \lambda_4}{\mathcal R}
&\frac{\beta_2}{\mathcal R|\xi|}\\
0&-\frac{\beta_4(\beta_4\nu^++\beta_1\nu^-)}{(\beta_1+\beta_4)\mathcal
R}-\frac{\beta_4\tilde \lambda_4}{\mathcal R}&0&
\frac{\beta_2(\beta_4\nu^++\beta_1\nu^-)}{(\beta_1+\beta_4)\mathcal R}+\frac{\beta_2\tilde r_4}{\mathcal R}\\
\frac{\beta_1\beta_2(\nu^--\nu^+)}{(\beta_1+\beta_4)\mathcal
R}+\frac{\beta_2\tilde \lambda_4}{\mathcal
R}&\frac{\beta_2}{\mathcal
R|\xi|}&\frac{\beta_1\beta_4(\nu^--\nu^+)}{(\beta_1+\beta_4)\mathcal
R}-\frac{\beta_1\tilde \lambda_4}{\mathcal R}
&-\frac{\beta_1}{\mathcal R|\xi|}\\
0&\frac{\beta_2(\beta_4\nu^++\beta_1\nu^-)}{(\beta_1+\beta_4)\mathcal
R}+\frac{\beta_2\tilde \lambda_4}{\mathcal R}&0&
-\frac{\beta_1(\beta_4\nu^++\beta_1\nu^-)}{(\beta_1+\beta_4)\mathcal
R}-\frac{\beta_1\tilde \lambda_4}{\mathcal R}
\end{pmatrix}\\
&+\mathcal O(|\xi|),\end{split}
\end{equation*}
\begin{equation*}
\begin{split}\label{1.32}P^4(\xi)=&-\begin{pmatrix}
\frac{\beta_1\beta_4(\nu^+-\nu^-)}{(\beta_1+\beta_4)\mathcal
R}-\frac{\beta_4\tilde \lambda_3}{\mathcal
R}&-\frac{\beta_4}{\mathcal
R|\xi|}&\frac{\beta_2\beta_4(\nu^+-\nu^-)}{(\beta_1+\beta_4)\mathcal
R}+\frac{\beta_2\tilde \lambda_3}{\mathcal R}
&\frac{\beta_2}{\mathcal R|\xi|}\\
0&-\frac{\beta_4(\beta_4\nu^++\beta_1\nu^-)}{(\beta_1+\beta_4)\mathcal
R}-\frac{\beta_4\tilde \lambda_3}{\mathcal R}&0&
\frac{\beta_2(\beta_4\nu^++\beta_1\nu^-)}{(\beta_1+\beta_4)\mathcal R}+\frac{\beta_2\tilde \lambda_3}{\mathcal R}\\
\frac{\beta_1\beta_2(\nu^--\nu^+)}{(\beta_1+\beta_4)\mathcal
R}+\frac{\beta_2\tilde \lambda_3}{\mathcal
R}&\frac{\beta_2}{\mathcal
R|\xi|}&\frac{\beta_1\beta_4(\nu^--\nu^+)}{(\beta_1+\beta_4)\mathcal
R}-\frac{\beta_1\tilde \lambda_3}{\mathcal R}
&-\frac{\beta_1}{\mathcal R|\xi|}\\
0&\frac{\beta_2(\beta_4\nu^++\beta_1\nu^-)}{(\beta_1+\beta_4)\mathcal
R}+\frac{\beta_2\tilde \lambda_3}{\mathcal R}&0&
-\frac{\beta_1(\beta_4\nu^++\beta_1\nu^-)}{(\beta_1+\beta_4)\mathcal
R}-\frac{\beta_1\tilde \lambda_3}{\mathcal R}
\end{pmatrix}\\
&+\mathcal O(|\xi|).\end{split}
\end{equation*}
It is clear that there exists the term like $\frac{1}{|\xi|}e^{-|\xi|^2t}$ in the corresponding entries of $\hat{G}^{l}$ to the two fraction densities $R^\pm$, which will result that 
\begin{equation}\label{1.33}
\begin{array}{rl}
\|n^\pm\|_{L^2}\lesssim (1+t)^{-\frac{1}{4}}.
\end{array}
\end{equation}
Besides, as for the pointwise space-time behavior, the term like $|\xi|^{-1}e^{-|\xi|^2t}$ will give rise to the so-called Riesz wave-I\!V, whose pointwise decay is captured by the profile $(1+t)^{-\frac{n-1}{2}}\big(1+\frac{|x|^2}{1+t}\big)^{-\frac{n-1}{2}}$ for $x\in\mathbb{R}^n$ with $n\geq2$.
In contrast to the treatment in \cite{Wu4} for the bipolar compressible Navier-Stokes-Poisson system, it is unnatural here to impose additional assumptions on the initial data $\Lambda^{-1}n^\pm_0$ in order to improve the decay rate. Instead, we aim to derive sharp pointwise estimates for the solution—a generalized Huygens principle—under the realistic situation where the $L^2$-decay of certain unknowns is slower than that of the classical heat kernel.
The main innovation of this work lies in handling the nonlinear interaction between waves of fundamentally different types. As seen from the representations of 
of $P^3$ and $P^4$ above, and $P^1$ and $P^2$ below, the nonlinear convolution necessarily involves the coupling of the slowly decaying Riesz wave-I\!V with the faster decaying Huygens waves. This structure is essentially different from that in \cite{Wu4} and raises the core difficulty:
How can one obtain sharp estimates for the convolution of the Riesz wave-I\!V with a Huygens wave under
 the slower decay in the ansatz of the solution to the nonlinear problem? In fact,
according to the above initial propagation, we postulate the following ansatz for the nonlinear problem:
\begin{eqnarray}\label{1.34}
|\partial_x^\alpha(n^+,n^-)|\lesssim (1+t)^{-\frac{2+|\alpha|}{2}}\Big(1+\frac{|x|^2}{1+t}\Big)^{-\frac{2+|\alpha|}{2}}
          +(1+t)^{-\frac{3+|\alpha|}{2}}\Big(1+\frac{(|x|-{\rm c}t)^2}{1+t}\Big)^{-1}.
\end{eqnarray}
To meet the minimal requirements for deriving the desired results from the convolution involving the Riesz wave-I\!V, we must extract improved decay rates for the pressure terms $\alpha^\pm\nabla P$. Note that the other terms, except for the two pressure terms $\alpha^\pm\nabla P$, are in divergence form (conservative structure). Hence, we focus primarily on the nonlinear estimates for the pressure terms $\alpha^\pm\nabla P$. At first glance, $\alpha^\pm\nabla P \sim n^-\nabla n^+ + n^+\nabla n^- +n^+\nabla n^++ n^-\nabla n^-$. Then, according to the above ansatz and the fact that the entry $G_{12}$ in Green's function corresponding to $n^+$ contains the Riesz wave-I\!V (which has the worst decay in both time and space), one encounters the following nonlinear convolution of the Riesz wave-I\!V and Huygens waves:
\begin{eqnarray}
\mathcal{N}_1= \int_0^t\!\int_{\mathbb{R}^3}
(1+t-\tau)^{-1}\Big(1+\frac{|x-y|^2}{1+t-\tau}\Big)^{-1}(1+\tau)^{-\frac{7}{2}}\Big(1+\frac{(|y|-{\rm c}\tau)^2}{1+\tau}\Big)^{-2}dyd\tau.\label{9.7}
\end{eqnarray}
When analyzing the pointwise behavior, a central task is to study the space-time convolution between the slowly decaying Riesz wave–I\!V and the nonlinear Huygens waves. A standard approach involves partitioning the spatial domain into the near-field  $(x,t)\in D_1=\{|x|\leq \sqrt{1+t}\}$ and the wave region
$D_2=\{||x|-{\rm c}t|\leq \sqrt{1+t}\}$. However, even within the simpler region
$D_1$, the time integration must be split into the sub-intervals: $0\leq \tau\leq\frac{t}{2}$ and $\frac{t}{2}\leq\tau\leq t$, denoted by $\mathcal{N}_{11}$ and $\mathcal{N}_{12}$, respectively. A direct estimate for the latter yields
\begin{eqnarray}\label{1.35}
\mathcal{N}_{12}&=&\int_0^\frac{t}{2}\!\int_{\mathbb{R}^3}(1+t-\tau)^{-1}\Big(1+\frac{|x-y|^2}{1+t-\tau}\Big)^{-1}(1+\tau)^{-\frac{7}{2}}\Big(1+\frac{(|y|-{\rm c}\tau)^2}{1+\tau}\Big)^{-2}dyd\tau\nonumber\\
&\lesssim& (1+t)^{-1}\int_0^{\frac{t}{2}}(1+\tau)^{-\frac{7}{2}}(1+\tau)^{\frac{5}{2}}d\tau\nonumber\\
&\lesssim& (1+t)^{-1}\ln(1+t)\lesssim (1+t)^{-1}\ln(1+t)\bigg(1+\frac{|x|^2}{1+t}\bigg)^{-N},\label{9.8}
\end{eqnarray}
which introduces a logarithmic growth factor. This result is strictly weaker than the decay rate prescribed by the ansatz (\ref{5.4}) for the unknowns 
$n^\pm$, presenting a major obstacle to closing the nonlinear argument.

\vspace{2mm}
To overcome this difficulty, we should conduct a refined analysis of the nonlinear terms $Q_1=(\beta_1\nabla n^++\beta_2\nabla n^-)-\alpha^+\nabla P$ and $Q_2=(\beta_3\nabla n^++\beta_4\nabla n^-)-\alpha^-\nabla P$ in $F_1$ and $F_2$ arising from two pressure terms. Through careful observation and algebraic reformulation, we find that they can be recast as:
$$
Q_1\sim n^+( \bar\rho^-\nabla n^++\bar\rho^+\nabla n^-)+\cdots,\ \ \ Q_2\sim n^-(\bar\rho^-\nabla n^++\bar\rho^+\nabla n^-)+\cdots,
$$ 
where ``$\cdots$" denote cubic terms with faster decay. The crucial insight lies in the structure of the linear propagators $P^3$ and $P^4$.
A hidden cancellation between their first and third rows endows the combined quantity
 $\bar\rho^-\nabla n^++\bar\rho^+\nabla n^-$ with an extra $(1+t)^{-\frac{1}{2}}$-decay factor: based on the cancellation of the first row and third row in  $P^3$ and $P^4$, that is
\begin{eqnarray}\label{1.36}
|\bar\rho^-n^++\bar\rho^+n^-|
\lesssim\bigg\{(1+t)^{-\frac{3}{2}}\Big(1+\frac{|x|^2}{1+t}\Big)^{-1}
          +(1+t)^{-2}\Big(1+\frac{(|x|-{\rm c}t)^2}{1+t}\Big)^{-1}\bigg\}.
\end{eqnarray}
This cancellation is fundamental and distinguishes our analysis from prior work. It transforms the problematic convolution into a new, refined form with significantly improved integrability:
\begin{eqnarray}
\mathcal{N}_1= \int_0^t\!\int_{\mathbb{R}^3}(1+t-\tau)^{-1}\Big(1+\frac{|x-y|^2}{1+t-\tau}\Big)^{-1}(1+\tau)^{-4}\Big(1+\frac{(|y|-{\rm c}\tau)^2}{1+\tau}\Big)^{-2}dyd\tau.\label{9.10}
\end{eqnarray}
Applying the precise convolution estimate $K_4$ from Lemma \ref{A.5} to $\mathcal{N}_1$ finally yields the sharp, desired bound:
\begin{eqnarray}
\mathcal{N}_1\lesssim (1+t)^{-1}\Big(1+\frac{|x|^2}{1+t}\Big)^{-1}
          +(1+t)^{-\frac{3}{2}}\Big(1+\frac{(|x|-{\rm c}t)^2}{1+t}\Big)^{-1},\label{9.11}
\end{eqnarray}
which is fully consistent with the original ansatz and devoid of any logarithmic loss.

The second difficulty comes from the non-conservative structure of the system (\ref{1.4}) and the wave operators in the low frequency of Green's function for the compressible part $\hat{G}^{l}$. In particular, we have
\begin{equation}\label{1.34}
\begin{cases}
\lambda_1=-\frac{\beta_1\nu^++\beta_4\nu^-}{2(\beta_1+\beta_4)}|\xi|^2+\mathrm{i}\sqrt{\beta_1+\beta_4}|\xi|+\mathcal O(|\xi|^3),\\
\lambda_2=-\frac{\beta_1\nu^++\beta_4\nu^-}{2(\beta_1+\beta_4)}|\xi|^2-\mathrm{i}\sqrt{\beta_1+\beta_4}|\xi|+\mathcal O(|\xi|^3),\\
\end{cases}\end{equation}
and their corresponding project operators $P^1$ and $P^2$ are given by
\begin{equation}\label{1.35}P^1(\xi)=\begin{pmatrix}
\frac{\beta_1}{2(\beta_1+\beta_4)}&\frac{\beta_1}{2(\beta_1+\beta_4)^\frac{3}{2}}\mathrm{i}&\frac{\beta_2}{2(\beta_1+\beta_4)}
&\frac{\beta_2}{2(\beta_1+\beta_4)^\frac{3}{2}}\mathrm{i}\\
-\frac{\beta_1}{2(\beta_1+\beta_4)^\frac{1}{2}}\mathrm{i}&\frac{\beta_1}{2(\beta_1+\beta_4)}&-\frac{\beta_2}{2(\beta_1+\beta_4)^\frac{1}{2}}\mathrm{i}&
\frac{\beta_2}{2(\beta_1+\beta_4)}\\
\frac{\beta_2}{2(\beta_1+\beta_4)}&\frac{\beta_2}{2(\beta_1+\beta_4)^\frac{3}{2}}\mathrm{i}&\frac{\beta_4}{2(\beta_1+\beta_4)}
&\frac{\beta_4}{2(\beta_1+\beta_4)^\frac{3}{2}}\mathrm{i}\\
-\frac{\beta_2}{2(\beta_1+\beta_4)^\frac{1}{2}}\mathrm{i}&\frac{\beta_2}{2(\beta_1+\beta_4)}&-\frac{\beta_4}{2(\beta_1+\beta_4)^\frac{1}{2}}\mathrm{i}&
\frac{\beta_4}{2(\beta_1+\beta_4)}
\end{pmatrix}+\mathcal O(|\xi|),\end{equation}
\begin{equation}\label{1.36}P^2(\xi)=\begin{pmatrix}
\frac{\beta_1}{2(\beta_1+\beta_4)}&-\frac{\beta_1}{2(\beta_1+\beta_4)^\frac{3}{2}}\mathrm{i}&\frac{\beta_2}{2(\beta_1+\beta_4)}
&-\frac{\beta_2}{2(\beta_1+\beta_4)^\frac{3}{2}}\mathrm{i}\\
\frac{\beta_1}{2(\beta_1+\beta_4)^\frac{1}{2}}\mathrm{i}&\frac{\beta_1}{2(\beta_1+\beta_4)}&\frac{\beta_2}{2(\beta_1+\beta_4)^\frac{1}{2}}\mathrm{i}&
\frac{\beta_2}{2(\beta_1+\beta_4)}\\
\frac{\beta_2}{2(\beta_1+\beta_4)}&-\frac{\beta_2}{2(\beta_1+\beta_4)^\frac{3}{2}}\mathrm{i}&\frac{\beta_4}{2(\beta_1+\beta_4)}
&-\frac{\beta_4}{2(\beta_1+\beta_4)^\frac{3}{2}}\mathrm{i}\\
\frac{\beta_2}{2(\beta_1+\beta_4)^\frac{1}{2}}\mathrm{i}&\frac{\beta_2}{2(\beta_1+\beta_4)}&\frac{\beta_4}{2(\beta_1+\beta_4)^\frac{1}{2}}\mathrm{i}&
\frac{\beta_4}{2(\beta_1+\beta_4)}
\end{pmatrix}+\mathcal O(|\xi|).\end{equation}
We see that the leading terms in each entry of $P^1$ and $P^2$ are constants, which is the same as the compressible Navier-Stokes equations. However, the lack of a divergence form in the nonlinear terms prevents us from directly applying classical convolution estimates for Huygens waves in \cite{Ls,Lw}.
To see this, let us recall the classical nonlinear convolution estimates on the interaction of Huygens waves in \cite{Ls,Lw}:
\begin{eqnarray}
&&\int_{0}^{t}\int_{\mathbb{R}^3}\uwave{(1+t-\tau)^{-\frac{5}{2}}}\Big(1+\frac{(|x-y|-{\rm c}(t-\tau))^2}{1\!+\!t\!-\!\tau}\Big)^{-N}(1+\tau)^{-4}\Big(1+\frac{(|y|-{\rm c}\tau)^2}{1\!+\!\tau}\Big)^{-3}dyd\tau\nonumber\\
&\lesssim& (1+t)^{-2}\Big(\Big(1+\frac{|x|^2}{1+t}\Big)^{-\frac{3}{2}}+\Big(1+\frac{(|x|-{\rm c}t)^2}{1+t}\Big)^{-\frac{3}{2}}\Big),\label{1.37}
\end{eqnarray}
where the extra $(1+t-\tau)^{-\frac{1}{2}}$-decay in the underline term arises precisely from the conservative (divergence) structure of the nonlinearity.
Since the momenta $\mathbf{m}^\pm$ decay at the heat‑kernel rate in $L^2$ \cite{Liy}, refining estimate (\ref{1.37}) in the usual way seems unattainable.
To resolve this difficulty, we devise a bold strategy that preserves the temporal decay by strategically sacrificing spatial decay in sub‑region estimates, especially near the wave cone. This compensates for the derivative loss induced by the non‑conservative terms. Through a meticulous and lengthy analysis of multiple space‑time subregions, we ultimately establish the following key convolution estimate:
\begin{align}
&\int_{0}^{t}\!\int_{\mathbb{R}^3}(1+t-\tau)^{-2}\Big(1+\frac{(|x-y|-{\rm c}(t-\tau))^2}{1+t-\tau}\Big)^{-N}(1+\tau)^{-4}\Big(1+\frac{(|y|-{\rm c}\tau)^2}{1+\tau}\Big)^{-2}dyd\tau\nonumber\\[3mm]
\lesssim &\ (1+t)^{-\frac{3}{2}}\Big(1+\frac{|x|^2}{1+t}\Big)^{-\frac{3}{2}}+(1+t)^{-2}\Big(1+\frac{(|x|-{\rm c}t)^2}{1+t}\Big)^{-1}.\label{1.38}
\end{align}
The proof, detailed in Lemma \ref{A.5}, relies on a delicate partition of the integration domain and optimized pointwise bounds, enabling us to recover the sharp decay needed for the ansatz despite the system's non‑conservation.

\bigbreak
\textbf{Discussion}

From the linear analysis, it is expected that improved decay rates for the two fractional momenta $\mathbf{m}^\pm$ in (\ref{1.18}) can be achieved if two distinct linear combinations, $\mathbf{m}^+-a_1\mathbf{m}^-$ and $\mathbf{m}^+-a_2\mathbf{m}^-$, can be constructed to compensate for the non-conservative structure arising from the two pressure terms in the momentum equations. Ultimately, this would yield the refined pointwise estimates (\ref{1.18}) for $\mathbf{m}^\pm$:
\begin{equation}\label{1.39}
|\mathbf{m}^\pm|\lesssim (1+t)^{-\frac{3}{2}}\Big(1+\frac{|x|^2}{1+t}\Big)^{-\frac{3}{2}}+(1+t)^{-2}\Big(1+\frac{(|x|-{\rm c}t)^2}{1+t}\Big)^{-\frac{3}{2}}.
\end{equation}

The first combination is found through a direct observation of the low-frequency projection matrices $P^1$ and $P^2$. A crucial structural property, $\beta_1\beta_4=\beta_2^2$, implies that the vectors $(P_{22}^1,P_{24}^1)$ and $(P_{42}^1,P_{44}^1)$ associated with $\mathbf{m}^\pm$ are \emph{proportional}. This allows us to immediately choose
\begin{equation*}
\mathbf{m}^+ - \frac{\beta_1}{\beta_2}\mathbf{m}^-
\end{equation*}
as the first combination. In this combination, the leading low-frequency terms become $|\xi|$, aligning it with a divergence structure and leading to the desired decay:
\begin{equation}\label{1.40}
\Big|\mathbf{m}^+-\frac{\beta_1}{\beta_2}\mathbf{m}^-\Big|\lesssim (1+t)^{-\frac{3}{2}}\Big(1+\frac{|x|^2}{1+t}\Big)^{-\frac{3}{2}}+(1+t)^{-2}\Big(1+\frac{(|x|-{\rm c}t)^2}{1+t}\Big)^{-\frac{3}{2}}.
\end{equation}

Finding a second, genuinely different linear combination is more subtle, as the proportionality used above cannot be applied again. Our approach leverages another good structure of the original system: the conservation of the \emph{total} fractional momentum $\mathbf{m}^++\mathbf{m}^-$ of the two-fluid model with equal pressures in (\ref{1.1}), combined with detailed Green's function estimates.
We denote the nonlinear couplings containing the pressure terms for $\mathbf{m}^\pm$ by $\tilde{\mathbf{m}}^\pm$, expressed via the low-frequency Green's functions $G_{ij}^\ell$:
\begin{equation}\label{1.41}
\begin{aligned}
\tilde{\mathbf{m}}^+&=\int_0^t G_{22}^\ell(\cdot,t-\tau)\ast_x Q_1(\cdot,\tau)d\tau+\int_0^t G_{24}^\ell(\cdot,t-\tau)\ast_x Q_2(\cdot,\tau)d\tau+\cdots,\\
\tilde{\mathbf{m}}^-&=\int_0^t G_{42}^\ell(\cdot,t-\tau)\ast_x Q_1(\cdot,\tau)d\tau+\int_0^t G_{44}^\ell(\cdot,t-\tau)\ast_x Q_2(\cdot,\tau)d\tau+\cdots,
\end{aligned}
\end{equation}
where $Q_1$ and $Q_2$ are non-divergence terms originating from the pressures (\ref{2.1}), and the rest terms don't affect the results. We then construct a candidate combination $\mathbf{m}^+-a_2\mathbf{m}^-$:
\begin{equation}\label{1.42}
\begin{aligned}
\tilde{\mathbf{m}}^+-a_2\tilde{\mathbf{m}}^- = &\int_0^t (G_{22}^l-a_2G_{42}^l)(\cdot,t-\tau)\ast \underbrace{(Q_1+Q_2)}_{\text{divergence form}}d\tau \\
&+ \int_0^t \underbrace{(-G_{22}^l+G_{24}^l+a_2G_{42}^l-a_2G_{44}^l)}_{=:G^*}\ast_x Q_2 \, d\tau +\cdots,
\end{aligned}
\end{equation}
where ``$\cdots$'' denotes middle and high-frequency parts not affecting the main Huygens wave analysis. Since $Q_1+Q_2$ is of divergence form, the first integral is manageable. The critical step is to choose $a_2$ such that the kernel $G^*$ has a leading term of order $|\xi|$ (not a constant). From the representation of $P^1$ or $P^2$, this requires
\begin{equation}\label{1.43}
-\beta_1+\beta_2+a_2\beta_2-a_2\beta_4=0,
\end{equation}
which yields $a_2=\frac{\beta_1-\beta_2}{\beta_2-\beta_4}$.
Using the fundamental relation $\beta_1\beta_4=\beta_2^2$, we find
\begin{equation*}
\frac{\beta_1-\beta_2}{\beta_2-\beta_4} \color{red}{=} \frac{\beta_1}{\beta_2},
\end{equation*}
provided $\beta_1 \neq \beta_2$. This shows that the second candidate combination degenerates to the first one. Therefore, this direct algebraic approach fails to produce a distinct second linear combination.

\textbf{Conclusion and Outlook:} The discovery of a genuinely different second linear combination for $\mathbf{m}^\pm$ remains an open and interesting problem. Its construction likely requires a deeper exploitation of the system's structure beyond the simple algebraic cancellation attempted here. Successfully finding it is essential to finally establish the refined pointwise estimates (\ref{1.39}) for the individual momenta $\mathbf{m}^\pm$, a task left for future work.

\section{\leftline {\bf{Green's function}}}
\setcounter{equation}{0}
\subsection{Linearization and Reformulation} In this subsection, we first reformulate the system.
Setting   \[ n^{\pm}=R^{\pm}-1,
\]
the Cauchy problem \eqref{1.14}--\eqref{1.15} can be reformulated as
\begin{equation}\label{2.1}
\left\{\begin{array}{l}
\partial_{t} n^++\operatorname{div}\mathbf{m}^+=0, \\
\partial_{t}\mathbf{m}^{+}+\beta_1\nabla n^++\beta_2\nabla
n^--\nu^+_1\Delta \mathbf{m}^+-\nu^+_2\nabla\operatorname{div} \mathbf{m}^+-\sigma^+\nabla\Delta n^+=F_1, \\
\partial_{t} n^-+\operatorname{div}\mathbf{m}^-=0, \\
\partial_{t}\mathbf{m}^{-}+\beta_3\nabla n^++\beta_4\nabla
n^--\nu^-_1\Delta \mathbf{m}^--\nu^-_2\nabla\operatorname{div} \mathbf{m}^--\sigma^-\nabla\Delta n^-=F_2, \\
\end{array}\right.
\end{equation}
where $\nu_{1}^{\pm}=\frac{\mu^{\pm}}{\bar{\rho}^{\pm}}$,
$\nu_{2}^{\pm}=\frac{\mu^{\pm}+\lambda^{\pm}}{\bar{\rho}^{\pm}}>0$,
$\beta_{1}=\frac{\mathcal{C}^{2}(1,1)
\bar{\rho}^{-}}{\bar{\rho}^{+}}$,
$\beta_{2}=\beta_{3}=\mathcal{C}^{2}(1,1)$,
$\beta_{4}=\frac{\mathcal{C}^{2}(1,1)
\bar{\rho}^{+}}{\bar{\rho}^{-}}$,
which imply
\begin{equation}\label{2.1(1)}
\beta_{1}\beta_{4}=\beta_{2}\beta_{3}=\beta_{2}^2=\beta_{3}^2.
\end{equation}
 Additionally,
the nonlinear terms are given by
\begin{equation}\label{2.2}
\begin{array}{rl}
F_{1}=&-\operatorname{div}\big(\frac{\mathbf{m}^+\otimes \mathbf{m}^+}{n^++1}\big)+\sigma^+R^+\nabla\Delta R^+\\
&+\mu^+{\rm div}\Big\{\frac{n^+}{\rho^+}\nabla\frac{\mathbf{m}^+}{n^++1}-\frac{1}{\rho^+}\nabla\frac{n^+\mathbf{m}^+}{n^++1}+\big(\frac{1}{\rho^+}-\frac{1}{\bar{\rho}^+}\big)\nabla \mathbf{m}^+\Big\} \\
&+\mu^+{\rm div}\Big\{\frac{n^+}{\rho^+}\nabla^t\frac{\mathbf{m}^+}{n^++1}-\frac{1}{\rho^+}\nabla^t\frac{n^+\mathbf{m}^+}{n^++1}+\big(\frac{1}{\rho^+}-\frac{1}{\bar{\rho}^+}\big)\nabla^t \mathbf{m}^+\Big\} \\
&+\lambda^+{\rm div}\Big\{\frac{n^+}{\rho^+}{\rm div}\frac{\mathbf{m}^+}{n^++1}I_3-\frac{1}{\rho^+}{\rm div}\frac{n^+\mathbf{m}^+}{n^++1}I_3+\Big(\frac{1}{\rho^+}-\frac{1}{\bar{\rho}^+}\Big){\rm div} \mathbf{m}^+I_3\Big\}\\
&-\underbrace{\big\{\mathcal{C}^{2}n^+\Big(\frac{\rho^-}{\rho^+}\nabla n^++\nabla n^-\Big)+\Big(\frac{\mathcal{C}^2\rho^-}{\rho^+}-\frac{\mathcal{C}^2\rho^-}{\rho^+}(1,1)\Big)\nabla n^++(\mathcal{C}^2-\mathcal{C}^2(1,1))\nabla n^-\big\}}_{\alpha^+\nabla P-(\beta_1\nabla n^++\beta_2\nabla n^-):=Q_1},
\end{array}
\end{equation}
and
\begin{equation}\label{2.3}
\begin{array}{rl}
F_{2}=&-\operatorname{div}\big(\frac{\mathbf{m}^-\otimes \mathbf{m}^-}{n^-+1}\big)+\sigma^-R^-\nabla\Delta R^-\\
&+\mu^-{\rm div}\Big\{\frac{n^-}{\rho^-}\nabla\frac{\mathbf{m}^-}{n^-+1}-\frac{1}{\rho^-}\nabla\frac{n^-\mathbf{m}^-}{n^-+1}+\big(\frac{1}{\rho^-}-\frac{1}{\bar{\rho}^-}\big)\nabla \mathbf{m}^-\Big\} \\
&+\mu^-{\rm div}\Big\{\frac{n^-}{\rho^-}\nabla^t\frac{\mathbf{m}^-}{n^-+1}-\frac{1}{\rho^-}\nabla^t\frac{n^-\mathbf{m}^-}{n^-+1}+\big(\frac{1}{\rho^-}-\frac{1}{\bar{\rho}^-}\big)\nabla^t \mathbf{m}^-\Big\} \\
&+\lambda^-{\rm div}\Big\{\frac{n^-}{\rho^-}{\rm div}\frac{\mathbf{m}^-}{n^-+1}I_3-\frac{1}{\rho^-}{\rm div}\frac{n^-\mathbf{m}^-}{n^-+1}I_3+\big(\frac{1}{\rho^-}-\frac{1}{\bar{\rho}^-}\big){\rm div} \mathbf{m}^-I_3\Big\}\\
&-\underbrace{\big\{\mathcal{C}^{2}n^-\big(\nabla n^++\frac{\rho^-}{\rho^+}\nabla n^-\big)+\big(\frac{\mathcal{C}^2\rho^+}{\rho^-}-\frac{\mathcal{C}^2\rho^+}{\rho^-}(1,1)\big)\nabla n^-+(\mathcal{C}^2-\mathcal{C}^2(1,1))\nabla n^+\big\}}_{\alpha^-\nabla P-(\beta_3\nabla n^++\beta_4\nabla n^-):=Q_2},
\end{array}
\end{equation}
where the last terms $Q_1$ and $Q_2$ in $F_1$ and $F_2$ have been given in (\ref{1.29}) in another simplified  form.

\
In
terms of the semigroup theory for evolutionary equation, we will
investigate the following initial value problem for the
corresponding linear system of \eqref{2.1} on the unknown $U:=(n^+,\mathbf{m}^+,n^-,\mathbf{m}^-)$:
\begin{equation}
\begin{cases}
U_t=\mathcal BU,\\
U\big|_{t=0}={U}_0,
\end{cases}   \label{2.4}
\end{equation}
where the operator $\mathcal B$ is given by
\begin{equation}\nonumber\mathcal B=\begin{pmatrix}
0&-\text{div}&0&0\\
-\beta_1\nabla+\sigma^+ \nabla\Delta&\nu_{1}^{+} \Delta+\nu_{2}^{+} \nabla \otimes \nabla&-\beta_2\nabla&0\\
0&0&0&-\text{div}\\
-\beta_3\nabla&0&-\beta_4\nabla+\sigma^- \nabla\Delta&\nu_{1}^{-}
\Delta+\nu_{2}^{-} \nabla \otimes \nabla
\end{pmatrix}.
\end{equation}
Applying Fourier transform to the system \eqref{2.4}, one has
\begin{equation}\label{2.5}
\begin{cases}
\widehat {{U}}_t=\mathcal B(\xi)\widehat {U},\\
\widehat {U}\big|_{t=0}=\widehat U_0=(\widehat {n^+_0}, \widehat{ \mathbf{m}^+_0},
\widehat {n^-_0}, \widehat {\mathbf{m}^-_0}),
\end{cases}
\end{equation}
where $\widehat
{U}(\xi,t)=\mathcal{F}(U(x,t))$,
$\xi=(\xi_1,\xi_2,\xi_3)^t$ and $\mathcal B(\xi)$ is defined by
\begin{equation*}\label{2.6}
\mathcal B(\xi)=\begin{pmatrix}
0&-{\rm i}\xi^t&0&0\\
-\mathrm{i}\beta_1\xi-{\rm i}\sigma^+|\xi|^2\xi&-\nu_1^+|\xi|^2\mathbb{I}_{3\times 3}-\nu_2^+\xi\otimes\xi&-{\rm i}\beta_2\xi&0\\
0&0&0&-{\rm i}\xi^t\\
- {\rm i}\beta_3\xi&0&-
{\rm i}\beta_4-{\rm i}\sigma^-|\xi|^2\xi&-\nu_1^-|\xi|^2\mathbb{I}_{3\times
3}-\nu_2^-\xi\otimes\xi
\end{pmatrix}.
\end{equation*}
To facilitate narrative in later use, we also use the definition of Green's function $G(x,t)$  with the following standard form as our previous works:
\begin{equation}\label{2.6(1)}
\left\{\begin{array}{l}
G_{t}=\mathcal{A} G, \\
G|_{t=0}=\delta_0(x)\mathbb{I}_{8\times8}.
\end{array}\right.
\end{equation}

To derive the linear estimates,  by using a real method
as in \cite{ld}, one need to make a detailed
analysis on the properties of Green's function. Since the system \eqref{2.12} has eight equations and
the matrix $\mathcal B(\xi)$ may not be diagonalizable. In order to facilitate the analysis, we
take Hodge decomposition to system \eqref{2.4} such that it can be
decoupled into three systems. One has four equations, and the other
two are classic heat equations.

To begin with, let $\varphi^{\pm}=\Lambda^{-1}{\rm
div}\mathbf{m}^{\pm}$ be the ``compressible part" of the momenta
$\mathbf{m}^{\pm}$, and denote $\phi^{\pm}=\Lambda^{-1}{\rm
curl}\mathbf{m}^{\pm}$ (with $({\rm curl} z)_i^j
=\partial_{x_j}z^i-\partial_{x_i}z^j$) by the ``incompressible part"
of the momenta $\mathbf{m}^{\pm}$. Then, we can divide the system
\eqref{2.4} into the compressible part
\begin{equation}\label{2.7}
\begin{cases}
\partial_t{n^+}+\Lambda{\varphi^+}=0,\\
\partial_t{\varphi^+}-\beta_1\Lambda{n^+}-\beta_2\Lambda{n^-}+\nu^+\Lambda^2{\varphi^+}-\sigma^+\Lambda^3{n^+}=0,\\
\partial_t{n^-}+\Lambda{\varphi^-}=0,\\
\partial_t{\varphi^-}-\beta_3\Lambda{n^+}-\beta_4\Lambda{n^-}+\nu^-\Lambda^2{\varphi^-}-\sigma^-\Lambda^3{n^-}=0,\\
(n^+, \varphi^+, n^-, \varphi^-)\big|_{t=0}=({n}^+_0, \Lambda^{-1}{\rm div}{\mathbf{m}}^{+}_0, {n}^-_0, \Lambda^{-1}{\rm div}{\mathbf{m}}^{-}_0)(x),\\
\end{cases}
\end{equation}
with its Green's function $\tilde{G}(x,t)$,
and the incompressible part
\begin{equation}\label{2.8}
\begin{cases}
\partial_t\phi^++\nu^+_1\Lambda^2\phi^+=0,\\
\partial_t\phi^-+\nu^-_1\Lambda^2\phi^-=0,\\
(\phi^+,\phi^-)\big|_{t=0}=(\Lambda^{-1}{\rm curl}{\mathbf{m}}^{+}_0,
\Lambda^{-1}{\rm curl}{\mathbf{m}}^{-}_0)(x)
\end{cases}
\end{equation}
where $\nu^{\pm}=\nu^{\pm}_1+\nu^{\pm}_2$. It is obvious that the incompressible part are two heat kernels: 
\begin{eqnarray}\label{2.8(1)}
\hat{\phi}^+=e^{-\nu_1^+|\xi|^2t},\ \hat{\phi}^-=e^{-\nu_1^-|\xi|^2t},
\end{eqnarray}
and hence, we mainly focus on the compressible part of Green's function.

\subsection{Compressible part}

 In view of the semigroup
theory, we may represent the IVP \eqref{2.4} for $\mathcal
U=(n^+, \varphi^+, n^-, \varphi^-)^t$ as
\begin{equation}\label{2.9}
\begin{cases}
\mathcal U_t=\mathcal A\mathcal U,\\
\mathcal U\big|_{t=0}=\mathcal U_0,
\end{cases}
\end{equation}
where the operator $\mathcal A$ is defined by
\begin{equation}\nonumber\mathcal A=\begin{pmatrix}
0&-\Lambda&0&0\\
\beta_1\Lambda+\sigma^+\Lambda^3&-\nu^+\Lambda^2&\beta_2\Lambda&0\\
0&0&0&-\Lambda\\
\beta_3\Lambda&0&\beta_4\Lambda+\sigma^-\Lambda^3&-\nu^-\Lambda^2
\end{pmatrix}.
\end{equation}
Taking Fourier transform to system \eqref{2.9}, we obtain
\begin{equation}\label{2.10}
\begin{cases}
\widehat {\mathcal U}_t=\mathcal A(\xi)\widehat {\mathcal U},\\
\widehat {\mathcal U}\ \big|_{t=0}=\widehat {\mathcal U}_0,
\end{cases}   
\end{equation}
where $\widehat {\mathcal U}(\xi,t)=\mathfrak{F}({\mathcal U}(x,t))$
and $\mathcal A(\xi)$ is given by
\begin{equation}\label{2.11}
\mathcal A(\xi)=\begin{pmatrix}
0&-|\xi|&0&0\\
\beta_1|\xi|+\sigma^+|\xi|^3&-\nu^+|\xi|^2&\beta_2|\xi|&0\\
0&0&0&-|\xi|\\
\beta_3|\xi|&0&\beta_4|\xi|+\sigma^-|\xi|^3&-\nu^-|\xi|^2
\end{pmatrix}.
\end{equation}
Its eigenvalues satisfy
\begin{equation}\label{2.12}
\begin{split}&{\rm det}(\lambda{\rm I}-\mathcal A(\xi))\\
&=\lambda^4+(\nu^+|\xi|^2+\nu^-|\xi|^2)\lambda^3+(\beta_1|\xi|^2+\beta_4|\xi|^2+\sigma^+|\xi|^4+\sigma^-|\xi|^4
+\nu^+\nu^-|\xi|^4)\lambda^2\\&\quad+(\beta_1\nu^-|\xi|^4+\beta_4\nu^+|\xi|^4+\nu^+\sigma^-|\xi|^6+\nu^-\sigma^+|\xi|^6)\lambda+\beta_1\sigma^-|\xi|^6+\beta_4\sigma^+|\xi|^6+\sigma^+\sigma^-|\xi|^8\\
&=0.
\end{split}
\end{equation}

\vspace{3.8mm}
\textbf{\textit{Low frequency part.}}\vspace{2mm}

From discriminant of the quartic equation in one unknown, we know that when $|\xi|\ll1$ there exist two real roots and a pair of conjugate imaginary roots. Then, by using the method of undetermined coefficients, we have the following for the spectral in the low frequency part:
\begin{lemma}\label{l 2.1} When $|\xi| \leq \eta_{1}\ll1$ with a constant $\eta_1$, the spectrum has the Taylor series expansion:
\begin{equation}\label{2.13}
\begin{cases}\lambda_1=-\frac{\beta_1\nu^++\beta_4\nu^-}{2(\beta_1+\beta_4)}|\xi|^2+\mathrm{i}\sqrt{\beta_1+\beta_4}|\xi|+\mathcal O(|\xi|^3),\\
\lambda_2=-\frac{\beta_1\nu^++\beta_4\nu^-}{2(\beta_1+\beta_4)}|\xi|^2-\mathrm{i}\sqrt{\beta_1+\beta_4}|\xi|+\mathcal O(|\xi|^3),\\
\lambda_3=\frac{-(\beta_1\nu^-+\beta_4\nu^+)+\sqrt{(\beta_1\nu^-+\beta_4\nu^+)^2-4(\beta_1+\beta_4)(\beta_1\sigma^-+\beta_4\sigma^+)}}{2(\beta_1+\beta_4)}|\xi|^2+\mathcal O(|\xi|^4),\\
\lambda_4=\frac{-(\beta_1\nu^-+\beta_4\nu^+)-\sqrt{(\beta_1\nu^-+\beta_4\nu^+)^2-4(\beta_1+\beta_4)(\beta_1\sigma^-+\beta_4\sigma^+)}}{2(\beta_1+\beta_4)}|\xi|^2+\mathcal O(|\xi|^4).
\end{cases}
\end{equation}
Here ${\rm c}:=\sqrt{\beta_1+\beta_4}$ is the propagation speed of the Huygens wave.
\end{lemma}

We divide the issue into two cases. Since it is basically similar to \cite{Liy}, we just give the sketch.

 {\color{blue}Case 1: $\lambda_3\neq \lambda_4$.}
In this case, we know the four eigenvalues are different, then one can immediately get 
the semigroup $\mathrm{e}^{t \mathcal{A}}$ (denoted by Green's function $\tilde{G}(x,t)$) can be decomposed into
\begin{equation}\label{2.11}
\tilde{G}(x,t):=\mathrm{e}^{t \mathcal{A}(\xi)}=\sum_{i=1}^{4} \mathrm{e}^{\lambda_{i} t} P^{i}(\xi),\ \ |\xi|\ll1,
\end{equation}
and the projector  $P^{i}(\xi)=\prod_{j \neq i} \frac{\mathcal{A}(\xi)-\lambda_{j} \mathbb{I}}{\lambda_{i}-\lambda_{j}}, \quad i, j=1,2,3,4$. 

A direct and tedious computation gives that
\begin{align}\label{2.14}
P^1(\xi)=\bar{P}^2(\xi)=\begin{pmatrix}
\frac{\beta_1}{2(\beta_1+\beta_4)}&\frac{\beta_1}{2(\beta_1+\beta_4)^\frac{3}{2}}\mathrm{i}&\frac{\beta_2}{2(\beta_1+\beta_4)}
&\frac{\beta_2}{2(\beta_1+\beta_4)^\frac{3}{2}}\mathrm{i}\\
-\frac{\beta_1}{2(\beta_1+\beta_4)^\frac{1}{2}}\mathrm{i}&\frac{\beta_1}{2(\beta_1+\beta_4)}&-\frac{\beta_2}{2(\beta_1+\beta_4)^\frac{1}{2}}\mathrm{i}&
\frac{\beta_2}{2(\beta_1+\beta_4)}\\
\frac{\beta_2}{2(\beta_1+\beta_4)}&\frac{\beta_2}{2(\beta_1+\beta_4)^\frac{3}{2}}\mathrm{i}&\frac{\beta_4}{2(\beta_1+\beta_4)}
&\frac{\beta_4}{2(\beta_1+\beta_4)^\frac{3}{2}}\mathrm{i}\\
-\frac{\beta_2}{2(\beta_1+\beta_4)^\frac{1}{2}}\mathrm{i}&\frac{\beta_2}{2(\beta_1+\beta_4)}&-\frac{\beta_4}{2(\beta_1+\beta_4)^\frac{1}{2}}\mathrm{i}&
\frac{\beta_4}{2(\beta_1+\beta_4)}
\end{pmatrix}+\mathcal O(|\xi|),
\end{align}
\begin{align}\label{2.15}
P^3(\xi)=&\begin{pmatrix}\!
\frac{\beta_1\beta_4(\nu^+-\nu^-)}{(\beta_1+\beta_4)\mathcal
R}-\frac{\beta_4\tilde \lambda_4}{\mathcal R}&-\frac{\beta_4}{\mathcal
R|\xi|}&\frac{\beta_2\beta_4(\nu^+-\nu^-)}{(\beta_1+\beta_4)\mathcal
R}+\frac{\beta_2\tilde \lambda_4}{\mathcal R}
&\frac{\beta_2}{\mathcal R|\xi|}\\
0&-\frac{\beta_4(\beta_4\nu^++\beta_1\nu^-)}{(\beta_1+\beta_4)\mathcal
R}-\frac{\beta_4\tilde \lambda_4}{\mathcal R}&0&
\frac{\beta_2(\beta_4\nu^++\beta_1\nu^-)}{(\beta_1+\beta_4)\mathcal R}+\frac{\beta_2\tilde \lambda_4}{\mathcal R}\\
\frac{\beta_1\beta_2(\nu^--\nu^+)}{(\beta_1+\beta_4)\mathcal
R}+\frac{\beta_2\tilde \lambda_4}{\mathcal
R}&\frac{\beta_2}{\mathcal
R|\xi|}&\frac{\beta_1\beta_4(\nu^--\nu^+)}{(\beta_1+\beta_4)\mathcal
R}-\frac{\beta_1\tilde \lambda_4}{\mathcal R}
&-\frac{\beta_1}{\mathcal R|\xi|}\\
0&\frac{\beta_2(\beta_4\nu^++\beta_1\nu^-)}{(\beta_1+\beta_4)\mathcal
R}+\frac{\beta_2\tilde \lambda_4}{\mathcal R}&0&
-\frac{\beta_1(\beta_4\nu^++\beta_1\nu^-)}{(\beta_1+\beta_4)\mathcal
R}-\frac{\beta_1\tilde \lambda_4}{\mathcal R}
\!\end{pmatrix}\nonumber\\
&+\mathcal O(|\xi|),
\end{align}
\begin{align}\label{2.16}
P^4(\xi)=&\!-\!\!\begin{pmatrix}\!\!
\frac{\beta_1\beta_4(\nu^+-\nu^-)}{(\beta_1+\beta_4)\mathcal
R}-\frac{\beta_4\tilde \lambda_3}{\mathcal
R}\! &\! -\frac{\beta_4}{\mathcal
R|\xi|} & \frac{\beta_2\beta_4(\nu^+-\nu^-)}{(\beta_1+\beta_4)\mathcal
R}+\frac{\beta_2\tilde \lambda_3}{\mathcal R}
 & \frac{\beta_2}{\mathcal R|\xi|}\\
0 \!&\! -\frac{\beta_4(\beta_4\nu^++\beta_1\nu^-)}{(\beta_1+\beta_4)\mathcal
R}-\frac{\beta_4\tilde \lambda_3}{\mathcal R} & 0 &
\frac{\beta_2(\beta_4\nu^++\beta_1\nu^-)}{(\beta_1+\beta_4)\mathcal R}+\frac{\beta_2\tilde \lambda_3}{\mathcal R}\\
\frac{\beta_1\beta_2(\nu^--\nu^+)}{(\beta_1+\beta_4)\mathcal
R}+\frac{\beta_2\tilde \lambda_3}{\mathcal
R} \!&\! \frac{\beta_2}{\mathcal
R|\xi|}&\frac{\beta_1\beta_4(\nu^--\nu^+)}{(\beta_1+\beta_4)\mathcal
R}-\frac{\beta_1\tilde \lambda_3}{\mathcal R}
&-\frac{\beta_1}{\mathcal R|\xi|}\\
0 \!&\! \frac{\beta_2(\beta_4\nu^++\beta_1\nu^-)}{(\beta_1+\beta_4)\mathcal
R}+\frac{\beta_2\tilde \lambda_3}{\mathcal R}&0&
-\frac{\beta_1(\beta_4\nu^++\beta_1\nu^-)}{(\beta_1+\beta_4)\mathcal
R}-\frac{\beta_1\tilde \lambda_3}{\mathcal R}
\!\!\end{pmatrix}\nonumber\\
&+\mathcal O(|\xi|).
\end{align}
Here $\mathcal{R}\!=\!\sqrt{(\beta_1\nu^-\!+\!\beta_4\nu^+)^2\!-\!4(\beta_1\!+\!\beta_4)(\beta_1\sigma^-\!+\!\beta_4\sigma^+)}$, $\tilde{\lambda}_4=\frac{-(\beta_1\nu^-+\beta_4\nu^+)+\mathcal{R}}{2(\beta_1+\beta_4)}$ and $\tilde{\lambda}_4=\frac{-(\beta_1\nu^-+\beta_4\nu^+)-\mathcal{R}}{2(\beta_1+\beta_4)}$.

{\color{blue}Case 2: $\lambda_3=\lambda_4$.} In this case,
 we have the following decomposition:
\begin{equation}\begin{split}\label{2.30}\widehat {\mathcal U}(\xi, t)=&P\begin{pmatrix}
\text{e}^{\lambda_1t}&0&0&0\\
0&\text{e}^{\lambda_2t}&0&0\\
0&0&\text{e}^{\lambda_3t}&t\text{e}^{\lambda_3t}\\
0&0&0&\text{e}^{\lambda_3t}
\end{pmatrix}P^{-1}\widehat {\mathcal
U}_0(\xi)\\=&\left(\text{e}^{\lambda_1t}P^1(\xi)+\text{e}^{\lambda_2t}P^2(\xi)+\text{e}^{\lambda_3t}P^3(\xi)+t\text{e}^{\lambda_3t}
P^4(\xi)\right)\widehat {\mathcal U}_0(\xi).
\end{split}\end{equation}
Then, after a direct and complicated computation, we can get the
explicit expressions of $P^i(\xi)$:
\begin{equation}\begin{split}\label{2.31}P^1(\xi)=\bar{P}^2(\xi)
=&\begin{pmatrix}
\frac{\beta_1}{2(\beta_1+\beta_4)}&\frac{\beta_1}{2(\beta_1+\beta_4)^\frac{3}{2}}\mathrm{i}&\frac{\beta_2}{2(\beta_1+\beta_4)}
&\frac{\beta_2}{2(\beta_1+\beta_4)^\frac{3}{2}}\mathrm{i}\\
-\frac{\beta_1}{2(\beta_1+\beta_4)^\frac{1}{2}}\mathrm{i}&\frac{\beta_1}{2(\beta_1+\beta_4)}&-\frac{\beta_2}{2(\beta_1+\beta_4)^\frac{1}{2}}\mathrm{i}&
\frac{\beta_2}{2(\beta_1+\beta_4)}\\
\frac{\beta_2}{2(\beta_1+\beta_4)}&\frac{\beta_2}{2(\beta_1+\beta_4)^\frac{3}{2}}\mathrm{i}&\frac{\beta_4}{2(\beta_1+\beta_4)}
&\frac{\beta_4}{2(\beta_1+\beta_4)^\frac{3}{2}}\mathrm{i}\\
-\frac{\beta_2}{2(\beta_1+\beta_4)^\frac{1}{2}}\mathrm{i}&\frac{\beta_2}{2(\beta_1+\beta_4)}&-\frac{\beta_4}{2(\beta_1+\beta_4)^\frac{1}{2}}\mathrm{i}&
\frac{\beta_4}{2(\beta_1+\beta_4)}
\end{pmatrix}+\mathcal O(|\xi|),
\end{split}\end{equation}
%
\begin{align}\label{2.34}P^3(\xi)=&\begin{pmatrix}
\frac{\beta_4}{\beta_1+\beta_4}&0&-\frac{\beta_2}{\beta_1+\beta_4}
&0\\
0&\frac{\beta_4}{\beta_1+\beta_4}&0&
-\frac{\beta_2}{\beta_1+\beta_4}\\
-\frac{\beta_2}{\beta_1+\beta_4}&0 &\frac{\beta_1}{\beta_1+\beta_4}
&0\\
0&-\frac{\beta_2}{\beta_1+\beta_4}&0&
\frac{\beta_1}{\beta_1+\beta_4}
\end{pmatrix}+\mathcal O(|\xi|),
\end{align}
\begin{align}\label{2.33}
P^4(\xi)=&\begin{pmatrix}
\frac{\beta_1\beta_4(\nu^+-\nu^-)}{2(\beta_1+\beta_4)^2}+\frac{\beta_4\nu^+}{2(\beta_1+\beta_4)}&-\frac{\beta_4}{(\beta_1+\beta_4)|\xi|}&\frac{\beta_2\beta_4(\nu^+-\nu^-)}{2(\beta_1+\beta_4)^2}
-\frac{\beta_2\nu^-}{2(\beta_1+\beta_4)}
&\frac{\beta_2}{(\beta_1+\beta_4)|\xi|}\\
0&-\frac{\beta_4(\beta_4\nu^++\beta_1\nu^-)}{2(\beta_1+\beta_4)^2}&0&
\frac{\beta_2(\beta_4\nu^++\beta_1\nu^-)}{2(\beta_1+\beta_4)^2}\\
\frac{\beta_1\beta_2(\nu^--\nu^+)}{2(\beta_1+\beta_4)^2}-\frac{\beta_2\nu^+}{2(\beta_1+\beta_4)}&\frac{\beta_2}{(\beta_1+\beta_4)|\xi|}
&\frac{\beta_1\beta_4(\nu^--\nu^+)}{2(\beta_1+\beta_4)^2}+\frac{\beta_1\nu^-}{2(\beta_1+\beta_4)}
&-\frac{\beta_1}{(\beta_1+\beta_4)|\xi|}\\
0&\frac{\beta_2(\beta_4\nu^++\beta_1\nu^-)}{2(\beta_1+\beta_4)^2}&0&
-\frac{\beta_1(\beta_4\nu^++\beta_1\nu^-)}{2(\beta_1+\beta_4)^2}
\end{pmatrix}|\xi|^2\nonumber\\
&\quad+\mathcal O(|\xi|^3).
\end{align}

Then one can see that the two cases above are almost the same in the sense of deriving the space-time behavior of Green's function $\tilde{G}(x,t)$, and hence, we only focus on the first case in the following without loss of generality.

\vspace{2.5mm}
{\bf High frequency part.}
 
 From the  characteristic equation (\ref{2.12}), we have the following expansion of spectra in the high frequency.
\begin{lemma}\label{l 2.2} There exists a positive constant  $K \gg 1$  such that, for  $|\xi| \geq K$, the spectrum has the following Taylor series expansion:
\begin{equation}\label{2.35}
\begin{cases}\lambda_{1,2}=-\frac{\gamma^+\pm\sqrt{(\gamma^+)^2-4\sigma^+}}{2}|\xi|^2+\mathcal O(1),\\[2mm]
\lambda_{3,4}=-\frac{\gamma^-\pm\sqrt{(\gamma^-)^2-4\sigma^-}}{2}|\xi|^2+\mathcal O(1).
\end{cases}
\end{equation}
\end{lemma}
\begin{remark}It is obvious that the real part of the spectra are always positive whenever $(\gamma^\pm)^2-4\sigma^\pm\geq0$ or $(\gamma^\pm)^2-4\sigma^\pm<0$. For critical case $(\gamma^+)^2-4\sigma^+=0$ or $(\gamma^-)^2-4\sigma^-=0$, there are multiple roots of the  characteristic equation, which corresponds to the Case 2: $\lambda_3=\lambda_4$ in the low frequency and can be treated similar to the general case and we omit the details for this case.
\end{remark}

After a direct and tedious computation, we can get the following expansion of the projection $P^i$ with $i=1,2,3,4$ in the high frequency according to Lemma \ref{l 2.2}:
\begin{align}\label{2.36}
P^{i}(\xi)=&\begin{pmatrix}
\mathcal{O}(1)&\mathcal{O}(|\xi|^{-1})&\mathcal{O}(|\xi|^{-2})&\mathcal{O}(|\xi|^{-3})\\
\mathcal{O}(|\xi|)&\mathcal{O}(1)&\mathcal{O}(|\xi|^{-1})&\mathcal{O}(|\xi|^{-2})\\
\mathcal{O}(|\xi|^{-2})&\mathcal{O}(|\xi|^{-3})&\mathcal{O}(1)&\mathcal{O}(|\xi|^{-1})\\
\mathcal{O}(|\xi|^{-1})&\mathcal{O}(|\xi|^{-2})&\mathcal{O}(|\xi|)&\mathcal{O}(1)
\end{pmatrix}+\cdots,
\end{align}
which together with (\ref{2.11}) and Lemma \ref{l 2.2}) gives that
\begin{lemma}\label{l 2.2(1)}
For any $|\alpha|\geq0$ and any integer $N>0$, the high frequency $\tilde{G}^h(x,t)$ also satisfies
\begin{equation}\label{2.37}
\begin{array}{rl}
&\displaystyle |\partial_x^\alpha(\tilde{G}_{ij}^h(x,t))|\lesssim e^{-t/C}t^{-\frac{3+|\alpha|}{2}}\Big(1+\frac{|x|^2}{t}\Big)^{-N},\ \ \ (i,j)\neq(2,1), (4,3),\\[2mm]
&\displaystyle |\partial_x^\alpha(\tilde{G}_{21}^h(x,t),\tilde{G}_{43}^h(x,t))|\lesssim e^{-t/C}t^{-\frac{4+|\alpha|}{2}}\Big(1+\frac{|x|^2}{t}\Big)^{-N}.
\end{array}
\end{equation}
\end{lemma}
\begin{remark}\label{r 2.1} The above estimates are based on the real analysis method as in \cite{Wang2}, which also shows that the high frequency of Green's function $\tilde{G}(x,t)$ behaves like the heat kernel or the derivative of the heat kernel. Additionally, since the incompressible part of the system (\ref{1.1}) obeys two decoupled heat equations, the high frequency of Green's function $G(x,t)$ of  (\ref{1.1}) has the same estimates as in Lemma \ref{l 2.2(1)}. \end{remark}

\vspace{2.5mm}
{\bf Middle frequency part.}

\vspace{2.5mm}
For the middle frequency part, we only need to prove the following two lemmas. The first one is about the positivity of the real part of spectra based on the Routh-Hurwitz theorem.
\begin{lemma}\label{l 2.3}
When $\eta_1\leq|\xi|\leq K$ with two fixed positive constants $\eta_1$ and $K$, there exists a positive constant $b$ such that
\begin{equation}\label{2.41}
\begin{array}{ll}
{\rm Re}(\lambda_1(|\xi|),\lambda_2(|\xi|),\lambda_3(|\xi|),\lambda_4(|\xi|))\leq -b.
\end{array}
\end{equation}
\end{lemma}

The analyticity for the middle frequency is partially based on the idea in Li \cite{ld}.
\begin{lemma}\label{l 2.4}
The middle part $\hat{\tilde{G}}^m(\xi,t)$ is analytic when $|\xi|^2\geq \delta$, where $\delta$ is a positive constant.
\end{lemma}

Lemma \ref{l 2.3} and Lemma \ref{l 2.4} directly yield the space-time estimates for middle frequency $\tilde{G}^m$:
\begin{lemma}\label{l 2.5}For any $|\alpha|\geq0$, there exists a constant $b>0$ such that
\begin{equation*}
|\partial_x^\alpha \tilde{G}^m(x,t)|\leq Ce^{-bt}(1+|x|^2)^{-N},
\end{equation*}
where the constant $N>0$ can be arbitrarily large.
\end{lemma}

%

\section{Pointwise space-time behavior of Green's function}
\quad\quad We shall bring Green's function in Fourier space back to the physical space, and combine the compressible part and the incompressible part to obtain the pointwise description of Green's function of the original system (\ref{2.2}). The difficulties mainly include resolving the singularity from the Riesz operator in low frequency of Green's matrix by using suitable combinations, and giving the description of the singular part in high frequency arising from the definition of Green's function.

\subsection{Pointwise space-time behavior of low frequency}
\quad\quad  The Hodge decomposition, Lemma \ref{l 2.1} together with (\ref{2.8(1)}) and (\ref{2.14})-(\ref{2.16}) directly give the following:
\begin{align}\label{3.1}
\hat{n}^{+,l}=&\Big[\frac{\beta_1(e^{\lambda_1t}+e^{\lambda_2t})}{2(\beta_1+\beta_4)}+\Big(\frac{\beta_1\beta_4(\nu^+-\nu^-)}{(\beta_1+\beta_4)\mathcal
R}-\frac{\beta_4\tilde \lambda_4}{\mathcal R}\Big)e^{\lambda_3t}+\Big(\frac{\beta_1\beta_4(\nu^+-\nu^-)}{(\beta_1+\beta_4)\mathcal
R}-\frac{\beta_4\tilde \lambda_3}{\mathcal
R}\Big)e^{\lambda_4t}\Big]\hat{n}_0^+\nonumber\\
\!\!&\!+\!\Big[\mathrm{i}\frac{\beta_1(e^{\lambda_1t}\!-\!e^{\lambda_2t})}{2(\beta_1+\beta_4)^\frac{3}{2}}-\frac{\beta_4}{\mathcal
R|\xi|}\Big(e^{\lambda_3t}+e^{\lambda_4t}\Big)\Big]\frac{\mathrm{i}\xi^T}{|\xi|}\hat{\mathbf{m}}_{0}^{+}\nonumber\\
&+\Big[\frac{\beta_2(e^{\lambda_1t}\!+\!e^{\lambda_2t})}{2(\beta_1+\beta_4)}+\Big(\frac{\beta_2\beta_4(\nu^+-\nu^-)}{(\beta_1+\beta_4)\mathcal
R}+\frac{\beta_2\tilde \lambda_4}{\mathcal R}\Big)e^{\lambda_3t}+\Big(\frac{\beta_2\beta_4(\nu^+-\nu^-)}{(\beta_1+\beta_4)\mathcal
R}+\frac{\beta_2\tilde \lambda_3}{\mathcal R}\Big)e^{\lambda_4t}\Big]\hat{n}_0^{-} \nonumber\\[-0.5mm]
&+\Big[\mathrm{i}\frac{\beta_2(e^{\lambda_1t}-e^{\lambda_2t})}{2(\beta_1+\beta_4)^\frac{3}{2}}+\frac{\beta_2}{\mathcal R|\xi|}(e^{\lambda_3t}+e^{\lambda_4t})\Big]\frac{\mathrm{i}\xi^T}{|\xi|}\hat{\mathbf{m}}_{0}^{-}+\cdots\nonumber\\[-0.5mm]
:=&(G_{11}^l,G_{12}^l,G_{13}^l,G_{14}^l)\cdot(\hat{n}_{0}^+,\hat{\mathbf{m}}_0^+,\hat{n}_{0}^-,\hat{\mathbf{m}}_{0}^-),
\end{align}
\vspace{-8mm}
\begin{align}\label{3.2}
\!\!\!\!\!\!\!\!\!\!\!\!\!\!\!\!\!\!\!\!\!\!\!\!\!\!\!&\!\!\!\!\!\!\hat{\mathbf{m}}^{+,l} =-\widehat{\Lambda^{-1}\nabla \varphi^+}-\widehat{\Lambda ^{-1}{\rm div} \phi^+}\nonumber\\
=&\Big[-\frac{\beta_1}{2(\beta_1+\beta_4)^\frac{1}{2}}(e^{\lambda_1t}-e^{\lambda_2t})\Big]\frac{\mathrm{i}\xi\hat{n}_{0}^{+}}{|\xi|}\nonumber\\
&+\bigg[\Big(\frac{\beta_1(e^{\lambda_1t}\!+\!e^{\lambda_2t})}{2(\beta_1+\beta_4)}-\Big(\frac{\beta_4(\beta_4\nu^+\!+\!\beta_1\nu^-)}{(\beta_1+\beta_4)\mathcal
R}+\frac{\beta_4\tilde \lambda_4}{\mathcal R}\Big)e^{\lambda_3t}-\Big(\frac{\beta_4(\beta_4\nu^+\!+\!\beta_1\nu^-)}{(\beta_1+\beta_4)\mathcal
R}+\frac{\beta_4\tilde \lambda_3}{\mathcal R}\Big)e^{\lambda_4t}\Big)\frac{\xi\xi^T}{|\xi|^{2}}\nonumber\\
&\ \ \ \ \ \ \ \ \ \ +e^{-\nu^+|\xi|^2t}\Big(I-\frac{\xi\xi^T}{|\xi|^{2}}\Big)\bigg]\hat{\mathbf{m}}_{0}^+\nonumber\\
&+\Big[-\frac{\beta_2}{2(\beta_1+\beta_4)^\frac{1}{2}}(e^{\lambda_1t}+e^{\lambda_2t})\mathrm{i}\xi\Big]\hat{n}_{0}^{-}\nonumber\\[-0.5mm]
&\!+\!\bigg[\frac{\beta_2(e^{\lambda_1t}\!+\!e^{\lambda_2t})}{2(\beta_1+\beta_4)}\!+\!\Big(\frac{\beta_2(\beta_4\nu^+\!\!+\!\beta_1\nu^-)}{(\beta_1+\beta_4)\mathcal R}\!+\!\frac{\beta_2\tilde \lambda_4}{\mathcal R}\Big)e^{\lambda_3t}\!+\!\Big(\frac{\beta_2(\beta_4\nu^+\!\!+\!\beta_1\nu^-)}{(\beta_1+\beta_4)\mathcal R}\!+\!\frac{\beta_2\tilde \lambda_3}{\mathcal R}\Big)e^{\lambda_4t}\bigg]\frac{\xi\xi^T}{|\xi|^{2}}\hat{\mathbf{m}}_{0}^{-}\!+\cdots\nonumber\\
:=&(G_{21}^l,G_{22}^l,G_{23}^l,G_{24}^l)\cdot(\hat{n}_{0}^+,\hat{\mathbf{m}}_0^+,\hat{n}_{0}^-,\hat{\mathbf{m}}_{0}^-),
\end{align}
\vspace{-8mm}
\begin{align}\label{3.3}
\hat{n}^{-,l}=&\Big[\frac{\beta_2(e^{\lambda_1t}+e^{\lambda_2t})}{2(\beta_1+\beta_4)}+\Big(\frac{\beta_1\beta_2(\nu^--\nu^+)}{(\beta_1+\beta_4)\mathcal
R}+\frac{\beta_4\tilde \lambda_4}{\mathcal R}\Big)e^{\lambda_3t}+\Big(\frac{\beta_1\beta_4(\nu^+-\nu^-)}{(\beta_1+\beta_4)\mathcal
R}+\frac{\beta_4\tilde \lambda_3}{\mathcal
R}\Big)e^{\lambda_4t}\Big]\hat{n}_0^+\nonumber\\
\!\!&\!+\!\Big[\mathrm{i}\frac{\beta_2}{2(\beta_1+\beta_4)^\frac{3}{2}}(e^{\lambda_1t}\!-\!e^{\lambda_2t})+\frac{\beta_2}{\mathcal
R|\xi|}\Big(e^{\lambda_3t}+e^{\lambda_4t}\Big)\Big]\frac{\mathrm{i}\xi^T}{|\xi|}\hat{\mathbf{m}}_{0}^{+}\nonumber\\
&+\Big[\frac{\beta_4(e^{\lambda_1t}\!+\!e^{\lambda_2t})}{2(\beta_1+\beta_4)}+\Big(\frac{\beta_1\beta_4(\nu^--\nu^+)}{(\beta_1+\beta_4)\mathcal
R}-\frac{\beta_1\tilde \lambda_4}{\mathcal R}\Big)e^{\lambda_3t}+\Big(\frac{\beta_1\beta_4(\nu^--\nu^+)}{(\beta_1+\beta_4)\mathcal
R}-\frac{\beta_1\tilde \lambda_3}{\mathcal R}\Big)e^{\lambda_4t}\Big]\hat{n}_0^{-} \nonumber\\[-0.5mm]
&+\Big[\mathrm{i}\frac{\beta_4}{2(\beta_1+\beta_4)^\frac{3}{2}}(e^{\lambda_1t}-e^{\lambda_2t})-\frac{\beta_1}{\mathcal R|\xi|}(e^{\lambda_3t}+e^{\lambda_4t})\Big]\frac{\mathrm{i}\xi^T}{|\xi|}\hat{\mathbf{m}}_{0}^{-}+\cdots\nonumber\\[-0.5mm]
:=&(G_{11}^l,G_{12}^l,G_{13}^l,G_{14}^l)\cdot(\hat{n}_{0}^+,\hat{\mathbf{m}}_0^+,\hat{n}_{0}^-,\hat{\mathbf{m}}_{0}^-),
\end{align}
\vspace{-8mm}
\begin{align}\label{3.4}
&\!\!\!\!\!\!\!\!\!\hat{\mathbf{m}}^{-,l}=-\widehat{\Lambda^{-1}\nabla \varphi^-}-\widehat{\Lambda ^{-1}{\rm div} \phi^-}\nonumber\\
=&\Big[-\frac{\beta_2}{2(\beta_1+\beta_4)^\frac{1}{2}}(e^{\lambda_1t}-e^{\lambda_2t})\Big]\frac{\mathrm{i}\xi\hat{n}_{0}^{+}}{|\xi|}\nonumber\\
&+\bigg[\Big(\frac{\beta_2(e^{\lambda_1t}\!+\!e^{\lambda_2t})}{2(\beta_1+\beta_4)}+\Big(\frac{\beta_2(\beta_4\nu^+\!+\!\beta_1\nu^-)}{(\beta_1+\beta_4)\mathcal
R}+\frac{\beta_2\tilde \lambda_4}{\mathcal R}\Big)e^{\lambda_3t}+\Big(\frac{\beta_2(\beta_4\nu^+\!+\!\beta_1\nu^-)}{(\beta_1+\beta_4)\mathcal
R}+\frac{\beta_2\tilde \lambda_3}{\mathcal R}\Big)e^{\lambda_4t}\Big)\frac{\xi\xi^T}{|\xi|^{2}}\bigg]\hat{\mathbf{m}}_{0}^+\nonumber\\
&+\Big[-\frac{\beta_4}{2(\beta_1+\beta_4)^\frac{1}{2}}(e^{\lambda_1t}+e^{\lambda_2t})\mathrm{i}\xi\Big]\hat{n}_{0}^{-}\nonumber\\[-0.5mm]
&+\bigg[\Big(\frac{\beta_4(e^{\lambda_1t}\!+\!e^{\lambda_2t})}{2(\beta_1+\beta_4)}\!-\!\Big(\frac{\beta_1(\beta_4\nu^+\!+\!\beta_1\nu^-)}{(\beta_1+\beta_4)\mathcal R}+\frac{\beta_2\tilde \lambda_4}{\mathcal R}\Big)e^{\lambda_3t}-\Big(\frac{\beta_1(\beta_4\nu^+\!+\!\beta_1\nu^-)}{(\beta_1+\beta_4)\mathcal R}+\frac{\beta_2\tilde \lambda_3}{\mathcal R}\Big)e^{\lambda_4t}\Big)\frac{\xi\xi^T}{|\xi|^{2}}\nonumber\\
&\ \ \ \ \ \ \ +e^{-\nu^-|\xi|^2t}\Big(I-\frac{\xi\xi^T}{|\xi|^{2}}\Big)\bigg]\hat{\mathbf{m}}_{0}^{-}+\cdots\nonumber\\
:=&(G_{21}^l,G_{22}^l,G_{23}^l,G_{24}^l)\cdot(\hat{n}_{0}^+,\hat{\mathbf{m}}_0^+,\hat{n}_{0}^-,\hat{\mathbf{m}}_{0}^-).
\end{align}
We can mainly deal with the leading term of each entry, since the rest terms (denoted by ``$\cdots$")  in (\ref{3.1})-(\ref{3.4})) are analytic and have faster temporal decay rates than the leading terms. The main difficulty focuses on the interaction of the singular Riesz operators (the first and second-order) in low frequency, the wave operators and the heat kernels. We merely take a few typical leading terms for instances to elaborate on the avoidance of the seeming singularity at $\xi=0$ and derive different wave patterns in the low frequency based on suitable combinations and cancellations. 

To facilitate explanation, we denote $\hat{\mathbf{w}}_t=\cos({\rm c}|\xi|t)$ and $\hat{\mathbf{w}}=\frac{\sin({\rm c}|\xi|t)}{|\xi|}$ are the Fourier transform of wave operators, where ${\rm c}$ is the propagation speed of the bulk in Lemma \ref{l 2.1}. Then, we first reformulate the following typical term in $\hat{G}_{12}^l$:
\begin{equation}\label{3.5}
\begin{array}{rl}
&\frac{e^{\lambda_1t}-e^{\lambda_2t}}{2}\frac{\xi^T}{|\xi|}=\frac{\xi^T}{|\xi|}e^{-b_1|\xi|^2t+\mathcal{O}(|\xi|^4)t}\sin({\rm c}|\xi|t+|\xi|\beta(|\xi|^2)t)\\[2mm]
=&\xi^T e^{-b_1|\xi|^2t+\mathcal{O}(|\xi|^4)t}\Big\{\sin({\rm c}|\xi|t)\cos(|\xi|\beta(|\xi|^2)t)+\cos({\rm c}|\xi|t)\sin(|\xi|\beta(|\xi|^2)t)\Big\}\\[2mm]
=&\xi^T\hat{\mathbf{w}}e^{-b_1|\xi|^2t+\mathcal{O}(|\xi|^4)t}+\xi^T\hat{\mathbf{w}}(\cos(|\xi|\beta(|\xi|^2)t)-1)e^{-b_1|\xi|^2t+\mathcal{O}(|\xi|^4)t}\\[2mm]
&+\xi^T\hat{\mathbf{w}}_t\frac{\sin(|\xi|\beta(|\xi|^2)t)}{|\xi|}e^{-b_1|\xi|^2t+\mathcal{O}(|\xi|^4)t},
\end{array}
\end{equation}
where $b_1=\frac{\beta_1\nu^++\beta_4\nu^-}{2(\beta_1+\beta_4)}>0$ and  $\beta(\cdot)$ is an analytic function.

Recall the estimates given in Li \cite{ld} for any $N>0$ that: 
\begin{equation}
\begin{array}{rl}
&\Big|\mathcal{F}^{-1}\Big((\cos(|\xi|\beta(|\xi|^2)t)-1)e^{-b_1|\xi|^2t+\mathcal{O}(|\xi|^4)t}\Big)\Big|+\Big|\mathcal{F}^{-1}(\frac{\sin(|\xi|\beta(|\xi|^2)t)}{|\xi|}e^{-b_1|\xi|^2t+\mathcal{O}(|\xi|^4)t})\Big|\\[2mm]
\lesssim &(1+t)^{-\frac{3}{2}}\big(1+\frac{|x|^2}{1+t}\big)^{-N},
\end{array}
\end{equation}
which together with  Kirchhoff formula in dimension three give that
\begin{equation}\label{3.5(0)}
\Big|\partial_x^\alpha\mathcal{F}^{-1}\Big(\frac{e^{\lambda_1t}-e^{\lambda_2t}}{2}\frac{\xi^T}{|\xi|}\Big)\Big|\lesssim (1+t)^{-\frac{4+|\alpha|}{2}}\Big(1+\frac{(|x|-{\rm c}t)^2}{1+t}\Big)^{-N}.
\end{equation}

Next, we study the first typical entry $G_{22}$, which is different from that of the one-phase fluid model (the compressible Navier-Stokes model) in \cite{Ls,Lw}, where 
\begin{equation}\label{3.5(1)}
\begin{array}{rl}
\hat{G}_{22}^l=&\frac{e^{\lambda_1t}+e^{\lambda_2t}}{2}\frac{\xi\xi^T}{|\xi|^{2}}-e^{-|\xi|^2t}\frac{\xi\xi^T}{|\xi|^{2}}+\cdots\\[2mm]
\sim &\underbrace{(\hat{\mathbf{w}}_t-1)\frac{\xi\xi^T}{|\xi|^{2}}e^{-b_1|\xi|^2t}}_{{\rm Riesz\ wave\shortbar I}}+\underbrace{\frac{\xi\xi^T}{|\xi|^{2}}(e^{-b_1|\xi|^2t}-e^{-|\xi|^2t})}_{{\rm Riesz\ wave\shortbar I\!I}}+\cdots.
\end{array}
\end{equation}
We use $\hat{R}_i$ with $i=1,2$ to denote the Riesz wave-I and Riesz wave-I\!I, respectively. As in Liu-Noh \cite{Ls} and Li \cite{ld} for the isentropic and non-isentropic compressible Navier-Stokes system, their pointwise space-time descriptions hold for an arbitrarily large integer $N$ that
\begin{align}\label{3.6}
\begin{array}{rl}
&|\partial_x^\alpha(\mathcal{F}^{-1}(\hat{R_1}))|\lesssim (1+t)^{-\frac{3+|\alpha|}{2}}\big(1+\frac{|x|^2}{1+t}\big)^{-\frac{3+|\alpha|}{2}}+(1+t)^{-\frac{4+|\alpha|}{2}}\big(1+\frac{(|x|-{\rm c}t)^2}{1+t}\big)^{-N},\\
&|\partial_x^\alpha(\mathcal{F}^{-1}(\hat{R_2}))|\lesssim (1+t)^{-\frac{3+|\alpha|}{2}}\big(1+\frac{|x|^2}{1+t}\big)^{-N}.
\end{array}
\end{align}

Due to the complicity of the two-phase fluid model, its entry $G_{22}$ in low frequency only can be rewritten as 
\begin{equation}\label{3.7}
\begin{array}{rl}
\hat{G}_{22}^l=&a_1\frac{e^{\lambda_1t}+e^{\lambda_2t}}{2}\frac{\xi\xi^T}{|\xi|^{2}}-a_2e^{-|\xi|^2t}\frac{\xi\xi^T}{|\xi|^{2}}+\cdots\ \ \ \ \ \ \ \ \ \ (a_1\neq a_2)\\[2mm]
\sim &\underbrace{(\hat{\mathbf{w}}_t-1)\frac{\xi\xi^T}{|\xi|^{2}}e^{-b_1|\xi|^2t}}_{{\rm Riesz\ wave\shortbar I}}+\underbrace{\frac{\xi\xi^T}{|\xi|^{2}}e^{-|\xi|^2t}}_{{\rm Riesz\ wave\shortbar I\!I\!I}}+\cdots,
\end{array}
\end{equation}
and we also use $\hat{R}_3$ to denote the Riesz wave-I\!I\!I. Although Riesz wave-I\!I\!I is different from Riesz wave-I\!I, we can also get the following according Lemma \ref{A.6}:
\begin{align}\label{3.8}
\begin{array}{rl}
&|\partial_x^\alpha(\mathcal{F}^{-1}(\hat{R_3}))|\lesssim (1+t)^{-\frac{3+|\alpha|}{2}}\big(1+\frac{|x|^2}{1+t}\big)^{-\frac{3+|\alpha|}{2}}.
\end{array}
\end{align}

Next, we turn to the Riesz wave-I\!V (like $|\xi|^{-1}e^{-|\xi|^2t}$ in Fourier space) in $G_{12}^l,G_{14}^l,G_{32}^l,G_{34}^l$, which is denoted by $R_4$ and has been studied by the first author and the second author in the present paper for the unipolar compressible Navier-Stokes-Poisson system in \cite{Wang}. In particular, 
\begin{align}\label{3.9}
\begin{array}{rl}
&|\partial_x^\alpha(\mathcal{F}^{-1}(\hat{R_4}))|\lesssim (1+t)^{-\frac{2+|\alpha|}{2}}\big(1+\frac{|x|^2}{1+t}\big)^{-\frac{2+|\alpha|}{2}}.
\end{array}
\end{align}
One can see that the Riesz wave-I\!V has the worst decay for both temporal and spacial variables. Different from the difficulty in \cite{Wang}, where its Green's function does not contain the Huygens waves, that is, the interaction of the Riesz wave-IV and the Huygens wave will not happen. However, due to the complication of the two-phase fluid model (\ref{1.1}), we have to deal with the nonlinear coupling of these two waves (at least). In fact, one of the large difficulty in the present paper is to derive its sharp estimate. See the details in the proof of $K_4$ in Lemma \ref{A.5}.

The second difficulty is arising from the non-conservative structure of the system (\ref{1.1}) due to two pressure terms and the presence of the Huygens waves in the low frequency of Green's function. In fact, in existing research works in this direction rely heavily on the nonlinear convolutions in Lemma \ref{A.4}, and these estimates are sharp in some sense. On the other hand, one can hardly find two kinds of linear combination to compensate this non-conservation as the previous work in this direction for the other related compressible fluid models. To achieve the relatively sharp pointwise space-time description of the solution to the nonlinear problem, we have to construct a series of new nonlinear convolution estimates of different wave patterns. See the details in Lemma \ref{A.5}.

In a conclusion, we have the pointwise estimate for Green's function in the low frequency:
\begin{lemma}\label{l 3.1} For $|\alpha|\geq0$, it holds that
\begin{align}\label{3.10}
&|\partial_x^\alpha(G_{11}^{l},G_{13}^{l}, G_{21}^{ l},G_{23}^{ l},G_{31}^{ l},G_{33}^{ l},G_{34}^{ l},G_{41}^{l},G_{42}^{l},G_{43}^{l},G_{44}^{l})|\nonumber\\
\lesssim& (1+t)^{-\frac{3+|\alpha|}{2}}\Big(1+\frac{|x|^2}{1+t}\Big)^{-N}+(1+t)^{-\frac{4+|\alpha|}{2}}\Big(1+\frac{(|x|-{\rm c}t)^2}{1+t}\Big)^{-N},\ \ \ 
\end{align}
\begin{align}\label{3.11}
|\partial_x^\alpha (G_{22}^{l},G_{24}^{l})|\lesssim(1+t)^{-\frac{3+|\alpha|}{2}}\Big(1+\frac{|x|^2}{1+t}\Big)^{-\frac{3+|\alpha|}{2}}+(1+t)^{-\frac{4+|\alpha|}{2}}\Big(1+\frac{(|x|-{\rm c}t)^2}{1+t}\Big)^{-N},
\end{align}
\begin{align}\label{3.12(1)}
|\partial_x^\alpha(G_{12}^{l},G_{14}^{l},G_{32}^{l},G_{34}^{l})|\lesssim (1+t)^{\!-\frac{2+|\alpha|}{2}}\Big(1\!+\!\frac{|x|^2}{1+t}\Big)^{\!-\frac{2+|\alpha|}{2}}\!\!+\!(1+t)^{\!-\frac{4+|\alpha|}{2}}\Big(1\!+\!\frac{(|x|\!-\!{\rm c}t)^2}{1+t}\Big)^{\!-N},
\end{align}
and a refined estimate for the following linear combination
\begin{align}\label{3.12(1)}
&|\partial_x^\alpha(\bar\rho^-G_{12}^{l}  +\bar\rho^+G_{32}^{l},\bar\rho^-G_{14}^{l}  +\bar\rho^+G_{34}^{l})|=\mathcal{O}(1)|\partial_x^\alpha(\beta_1G_{12}^{l}+\beta_2G_{12}^{l},\beta_1G_{14}^{l}+\beta_2G_{34}^{l})|\nonumber\\
\lesssim &(1+t)^{\!-\frac{3+|\alpha|}{2}}\Big(1\!+\!\frac{|x|^2}{1+t}\Big)^{\!-\frac{3+|\alpha|}{2}}\!\!+\!(1+t)^{\!-\frac{4+|\alpha|}{2}}\Big(1\!+\!\frac{(|x|\!-\!{\rm c}t)^2}{1+t}\Big)^{\!-N},
\end{align}
where the constant $N>0$ can be arbitrarily large.
\end{lemma}
We emphasize again that the estimate (\ref{3.12(1)}) is crucial for us to gain the desired, sharp estimate concerning the nonlinear convolution of the Riesz wave-I\!V and the Huygens wave in Lemma \ref{A.5}.

\vspace{5mm}
Combining Lemma \ref{l 3.1}, Remark \ref{r 2.1} and Lemma \ref{l 2.5}, we conclude the following.
\begin{proposition}{\rm [\textbf{Space-time\ description\ of\ Green\ function}]}\label{l 3.3}\\
For $|\alpha|\geq0$, one has the following for a arbitrarily large constant $N>0$:
\begin{align}\label{3.17}
&When\ (i,j)=(1,1),(1,3),(2,3),(3,1),(3,3),(4,1),(4,2),(4,4),\nonumber\\
&\ \ \ |\partial_x^\alpha(G_{ij}\!-\!G_{S1})|\lesssim (1+t)^{-\frac{3+|\alpha|}{2}}\Big(1+\frac{|x|^2}{1+t}\Big)^{-N}+(1+t)^{-\frac{4+|\alpha|}{2}}\Big(1+\frac{(|x|-{\rm c}t)^2}{1+t}\Big)^{-N},
\end{align}
\begin{align}\label{3.18}
\!\!\!\!|\partial_x^\alpha (G_{21}\!-\!G_{S2},G_{43}\!-\!G_{S2})|\lesssim(1+t)^{\!-\frac{3+|\alpha|}{2}}\Big(1\!+\!\frac{|x|^2}{1+t}\Big)^{\!-\frac{3+|\alpha|}{2}}\!+\!(1+t)^{\!-\frac{4+|\alpha|}{2}}\Big(1\!+\!\frac{(|x|\!-\!{\rm c}t)^2}{1+t}\Big)^{\!-N}\!\!,\!\!
\end{align}
\begin{align}\label{3.18}
\!\!\!\!|\partial_x^\alpha (G_{22}\!-\!G_{S1},G_{24}\!-\!G_{S1})|\lesssim(1\!+\!t)^{-\frac{3+|\alpha|}{2}}\Big(1\!+\!\frac{|x|^2}{1+t}\Big)^{\!-\frac{3+|\alpha|}{2}}\!+\!(1+t)^{\!-\frac{4+|\alpha|}{2}}\Big(1+\frac{(|x|\!-\!{\rm c}t)^2}{1+t}\Big)^{\!-N}\!\!,\!\!
\end{align}
\begin{align}\label{3.19}
|\partial_x^\alpha(G_{12}\!-\!G_{S1},G_{14}-G_{S1},G_{32}\!-\!G_{S1},G_{34}\!-\!G_{S1})|\ \ \ \ \ \ \ \ \ \ \ \ \ \ \ \ \ \ \ \ \ \ \ \ \ \ \ \ \ \ \ \ \ \ \ \ \ \ \ \ \ \  \ \ \ \ \ \nonumber\\
\lesssim (1+t)^{-\frac{2+|\alpha|}{2}}\Big(1+\frac{|x|^2}{1+t}\Big)^{\!-\frac{2+|\alpha|}{2}}+(1+t)^{-\frac{4+|\alpha|}{2}}\Big(1+\frac{(|x|-{\rm c}t)^2}{1+t}\Big)^{\!-N},
\end{align}
where the singular terms $G_{S1}=e^{-t/C}t^{-\frac{3+|\alpha|}{2}}\Big(1+\frac{|x|^2}{t}\Big)^{-N}$ and $G_{S2}=e^{-t/C}t^{-\frac{4+|\alpha|}{2}}\Big(1+\frac{|x|^2}{t}\Big)^{-N}$.
\end{proposition}

\section{Pointwise estimates for nonlinear system}

\quad\quad In this section, by using the representation of the solution, we derive the pointwise estimates for the solution of the nonlinear system based on the pointwise space-time description of Green's function and the convolution estimates of nonlinear coupling. 

Without loss of generality, we use $\partial_x^k$ to denote $\partial_x^\alpha$ with $|\alpha|=k$.
By using Duhamel principle, we can get the representation of the solution $(n^+,\mathbf{m}^+,n^-,\mathbf{m}^-)$ for the nonlinear problem (\ref{2.1}):
\begin{eqnarray}\label{5.1}
 &&\partial_x^k(n^+,\mathbf{m}^+,n^-,\mathbf{m}^-)^T\nonumber\\
 &=&\partial_x^k G(x,t)\ast_x(n^+_0,\mathbf{m}^+_0,n^-_0,\mathbf{m}^-_0)^T
          +\int_0^t\partial_x^k G(\cdot,t-\tau)\ast_x(0,F_1,0,F_2)^T(\cdot,\tau)d\tau,
\end{eqnarray}
where the nonlinear terms $F_1$ and $F_2$ are defined in (\ref{2.2}).

\textbf{Initial propagation.} Use $(\breve{n}^+,\breve{\mathbf{m}}^+,\breve{n}^-,\breve{\mathbf{m}}^-)$ to denote the linear part of the solution in (\ref{5.1}). According to the representation (\ref{5.1}), Proposition \ref{l 3.3}, the initial condition (\ref{1.3}) and the convolution estimate in Lemma \ref{A.4} for the initial propagation, we can immediately get the following:
\begin{equation}\label{5.2}
\begin{array}{rl}
 &\left|\partial_x^k\left(\!\!
            \begin{array}{ccc}
              \breve{n}^+ \\
              \breve{n}^-
            \end{array}
          \!\!\right)\right|=\!\left|\partial_x^k\!\left(\!
                            \begin{array}{cccc}
                              G_{11} & G_{12} & G_{13} &G_{14}\\
                               G_{31} & G_{32} & G_{33} &G_{34}\\
                            \end{array}
                          \!\right)\!\!\ast_x(n^+_0,\mathbf{m}^+_0,n^-_0,\mathbf{m}^-_0)^T
          \right|\\
          \leq\!\! & \left|\partial_x^k\!\bigg(\!
                            \begin{array}{cccc}
                             G_{11}-G_{S1} & G_{12}-G_{S1} & G_{13}-G_{S1} & G_{14}-G_{S1}\\
                               G_{31}-G_{S1} & G_{32}-G_{S1} & G_{33}-G_{S1} & G_{34}-G_{S2}\\
                            \end{array}
                          \!\bigg)\!\!\ast_x (n^+_0,\mathbf{m}^+_0,n^-_0,\mathbf{m}^-_0)^T\right|\\
          &+\left| \left(\!
                            \begin{array}{cccc}
                              G_{S1} &G_{S1} & G_{S1} &  G_{S1}\\
                              G_{S1} & G_{S1} & G_{S1} & G_{S1}\\
                            \end{array}
                          \!\right)\!\ast_x\partial_x^k (n^+_0,\mathbf{m}^+_0,n^-_0,\mathbf{m}^-_0)^T\right|\\
          \leq &\displaystyle C\epsilon\bigg((1+t)^{-\frac{2+k}{2}}\Big(1+\frac{|x|^2}{1+t}\Big)^{-\frac{2+k}{2}}
          +(1+t)^{-\frac{4+k}{2}}\Big(1+\frac{(|x|-{\rm c}t)^2}{1+t}\Big)^{-1}\bigg),
\end{array}
\end{equation}
\begin{equation}\label{5.3}
\begin{array}{rl}
 &\left|\partial_x^k\left(\!\!
            \begin{array}{ccc}
              \breve{\mathbf{m}}^+ \\
              \breve{\mathbf{m}}^-
            \end{array}
          \!\!\right)\right|=\!\left|\partial_x^k\!\left(\!
                            \begin{array}{cccc}
                              G_{21} & G_{22} & G_{23} &G_{24}\\
                               G_{41} & G_{42} & G_{43} &G_{44}\\
                            \end{array}
                          \!\right)\!\!\ast_x (n^+_0,\mathbf{m}^+_0,n^-_0,\mathbf{m}^-_0)^T\right|\\
          \leq\!\! & \left|\partial_x^k\!\left(\!
                            \begin{array}{cccc}
                             G_{21}-G_{S2} & G_{22}-G_{S1} & G_{23}-G_{S1} &G_{24}-G_{S1}\\
                               G_{41}-G_{S1} & G_{42}-G_{S1} & G_{43}-G_{S2} &G_{44}-G_{S1}\\
                            \end{array}
                          \!\right)\!\!\ast_x (n^+_0,\mathbf{m}^+_0,n^-_0,\mathbf{m}^-_0)^T\right|\\
          &+\left| \bigg(\!
                            \begin{array}{cccc}
                              \uwave{G_{S2}} & G_{S1} & G_{S1} & G_{S1}\\
                              G_{S1} & G_{S1} & \uwave{G_{S2}} &G_{S1}
                            \end{array}
                          \!\bigg)\!\ast_x(\uwave{\partial_x^{k+1}n_{0}^+},\partial_x^k \mathbf{m}_0^+,\uwave{\partial_x^{k+1}n_{0}^-},\partial_x^k \mathbf{m}_{0}^-)^T\right|\\
          \leq &\displaystyle C\epsilon\bigg((1+t)^{-\frac{3+k}{2}}\Big(1+\frac{|x|^2}{1+t}\Big)^{-\frac{3}{2}}
          +(1+t)^{-\frac{4+k}{2}}\Big(1+\frac{(|x|-{\rm c}t)^2}{1+t}\Big)^{-1}\bigg).
\end{array}
\end{equation}
Note that it requires one higher order regularity assumption on the initial value of the fraction densities $(n_{0}^+,n_{0}^-)$, which is consistent with the standard energy estimates for the compressible Naiver-Stokes-Korteweg system.
\begin{remark}It seems that we can only obtain the corresponding estimate by replacing the decay $(1+t)^{-a}$ in (\ref{5.2})-(\ref{5.3}) with $t^{-a}$, since $G_{S1}$ and $G_{S2}$ arising from the high frequency of Green's function are singular at $t=0$. Nevertheless, similar to the study on the initial propagation of the one-phase fluid model in Proposition 6.3 of \cite{Ls} by considering two subcases $0\leq t<1$ and $t\geq1$, one can easily achieve the above two estimates of initial propagation for the two-phase fluid model when $t\geq0$. We omit the details here since it has no essential impact on the conclusion.
\end{remark}

\textbf{Nonlinear Coupling.} According to the above initial propagation, we give the ansatz for the nonlinear problem with $k\leq3$:
\begin{eqnarray}\label{5.4}
|(n^+,n^-)|\leq 2C\epsilon\bigg\{(1+t)^{-1}\Big(1+\frac{|x|^2}{1+t}\Big)^{-1}
          +(1+t)^{-\frac{3}{2}}\Big(1+\frac{(|x|-{\rm c}t)^2}{1+t}\Big)^{-1}\bigg\},\ \ \ \ \ \ \ \ \ \ \ \ \ \ \ \ \ \ \ \ 
\end{eqnarray}
\vspace{-4mm}
\begin{eqnarray}\label{5.5}
&&|\bar\rho^-n^++\bar\rho^+n^-|+|\partial_x^k(n^+,n^-)|\nonumber\\
&&\ \ \ \ \ \ \ \ \ \ \ \leq 2C\epsilon\bigg\{(1+t)^{-\frac{3}{2}}\Big(1+\frac{|x|^2}{1+t}\Big)^{-1}
          +(1+t)^{-2}\Big(1+\frac{(|x|-{\rm c}t)^2}{1+t}\Big)^{-1}\bigg\},\ \ k=1,2,3, \ \ \ \ \ \ \ \ \ \ 
\end{eqnarray}
\begin{equation}\label{5.6}
|\partial_x^k (\mathbf{m}^+,\mathbf{m}^-)|\leq 2C\epsilon\bigg\{(1+t)^{-\frac{3}{2}}\Big(1+\frac{|x|^2}{1+t}\Big)^{-1}
          +(1+t)^{-2}\Big(1+\frac{(|x|-{\rm c}t)^2}{1+t}\Big)^{-1}\bigg\},\ k=0,1,2,3.\ \ 
\end{equation}

\begin{remark}In order to meet the minimal requirements for deriving the desired results from this convolution concerning Riesz wave-I\!V, we must extract improved decay rates for the pressure terms $\alpha^\pm\nabla P$. Since the other terms except two pressure terms $\alpha^\pm\nabla P$ are the divergence form (conservative structure), we mainly focus on the nonlinear estimates concerning the pressure terms $\alpha^\pm\nabla P$. 
At first glance, $\alpha^\pm\nabla P \sim  n^-\nabla n^+ + n^+\nabla n^- +n^+\nabla n^+ + n^-\nabla n^-$, then according to the above ansatz and the fact the entries $G_{12}$ in Green's function corresponding to the unknowns $n^\pm$ contain the Riesz wave-I\!V with worst decay on both temporal and spacial variables, one will meet the following nonlinear convolution of the Riesz wave-I\!I and the nonlinear Huygens waves:
\begin{eqnarray}
\mathcal{N}_1= \int_0^t\!\int_{\mathbb{R}^3}
(1+t-\tau)^{-1}\Big(1+\frac{|x-y|^2}{1+t-\tau}\Big)^{-1}(1+\tau)^{-\frac{7}{2}}\Big(1+\frac{(|y|-{\rm c}\tau)^2}{1+\tau}\Big)^{-2}dyd\tau.\label{5.7}
\end{eqnarray}
When estimating this convolution, one should deal with the domains $D_1=\{|x|\leq \sqrt{1+t}\}$ and $D_2=\{||x|-{\rm c}t|\leq \sqrt{1+t}\}$, and the estimates in these two domains are direct and simplest. However, even for $(x,t)\in D_1$, one should at least consider two subcases: $0\leq \tau\leq\frac{t}{2}$ and $\frac{t}{2}\leq\tau\leq t$ denoted by $\mathcal{N}_{11}$ and $\mathcal{N}_{12}$, respectively. Thus,
\begin{eqnarray}\label{5.8}
\mathcal{N}_{12}&=&\int_0^\frac{t}{2}\!\int_{\mathbb{R}^3}(1+t-\tau)^{-1}\Big(1+\frac{|x-y|^2}{1+t-\tau}\Big)^{-1}(1+\tau)^{-\frac{7}{2}}\Big(1+\frac{(|y|-{\rm c}\tau)^2}{1+\tau}\Big)^{-2}dyd\tau\nonumber\\
&\lesssim& (1+t)^{-1}\int_0^{\frac{t}{2}}(1+\tau)^{-\frac{7}{2}}(1+\tau)^{\frac{5}{2}}d\tau\nonumber\\
&\lesssim& (1+t)^{-1}\ln(1+t)\lesssim (1+t)^{-1}\ln(1+t)\bigg(1+\frac{|x|^2}{1+t}\bigg)^{-N},\label{5.8}
\end{eqnarray}
which is obviously weaker than the ansatz (\ref{5.4}) on the unknowns $n^\pm$. \end{remark}

To resolve this difficulty, we have to find out a refined decay estimate of the nonlinear terms $Q_1=(\beta_1\nabla n^++\beta_2\nabla n^-)-\alpha^+\nabla P$ and $Q_2=(\beta_3\nabla n^++\beta_4\nabla n^-)-\alpha^-\nabla P$ in $F_1$ and $F_2$ arising from two pressure terms. To this end, after a careful observation and reformulation, we can have
\begin{eqnarray}
Q_1&=&\mathcal{C}^{2}(n^+\!+\!1,n^-\!+\!1)n^+\Big(\frac{\rho^-}{\rho^+}\nabla n^++\nabla n^-\Big)\nonumber\\
&&+\Big(\frac{\mathcal{C}^2(n^+\!+\!1,n^-\!+\!1)\rho^-}{\rho^+}-\frac{\mathcal{C}^2(1,1)\bar\rho^-}{\bar\rho^+}\Big)\nabla n^++(\mathcal{C}^2(n^+\!+\!1,n^-\!+\!1)-\mathcal{C}^2(1,1))\nabla n^-\nonumber\\
&=&\frac{\mathcal{C}^2(n^+\!+\!1,n^-\!+\!1)}{\rho^+}n^+\bigg[(\rho^--\bar\rho^-)\nabla n^+\!+\!(\rho^+-\bar\rho^+)\nabla n^-\bigg]\nonumber\\
&&+\frac{\mathcal{C}^2(n^+\!+\!1,n^-\!+\!1)}{\rho^+}n^+(\uwave{\bar\rho^-\nabla n^++\bar\rho^+\nabla n^-})\nonumber\\
&&+\frac{\mathcal{C}^2(n^+\!+\!1,n^-\!+\!1)-\mathcal{C}^2(1,1)}{\rho^+}[(\rho^--\bar\rho^-)\nabla n^++(\rho^+-\bar\rho^+)\nabla n^-]\nonumber\\
&&+\frac{1}{\rho^+}(\mathcal{C}^2(n^+\!+\!1,n^-\!+\!1)-\mathcal{C}^2(1,1))(\uwave{\bar\rho^-\nabla n^++\bar\rho^+\nabla n^-})\nonumber\\
&&+\mathcal{C}^2(1,1)\bigg(\frac{1}{\rho^+}-\frac{1}{\bar\rho^+}\bigg)(\uwave{\bar\rho^-\nabla n^++\bar\rho^+\nabla n^-})\nonumber\\
&&+\mathcal{C}^2(1,1)\bigg(\frac{1}{\rho^+}-\frac{1}{\bar\rho^+}\bigg)[(\rho^--\bar\rho^-)\nabla n^++(\rho^+-\bar\rho^+)\nabla n^-]\nonumber\\
&&+\frac{\mathcal{C}^2(1,1)}{\bar\rho^+}[(\uwave{\rho^--\bar\rho^-})\nabla n^++(\uwave{\rho^+-\bar\rho^+})\nabla n^-]:=\sum\limits_{i=1}^7 \mathcal{S}_i,\label{5.9}
\label{5.9}
\end{eqnarray}
where $\rho^\pm=\rho^\pm(n^++1,n^-+1)$ and $\bar\rho^\pm=\rho^\pm(1,1)$. It is obvious that $\mathcal{S}_1$, $\mathcal{S}_3$ and $\mathcal{S}_6$ are the cube of the unknowns and naturally have faster decay from the ansatz. The term $\bar\rho^-\nabla n^++\bar\rho^+\nabla n^-$ in $\mathcal{S}_2$, $\mathcal{S}_4$ and $\mathcal{S}_5$ has an additional $(1+t)^{-\frac{1}{2}}$-decay rate than $\nabla n^\pm$ according to the ansatz (\ref{5.5}). Finally, since $\rho^\pm-\bar\rho^\pm\sim \bar\rho^-\nabla n^++\bar\rho^+\nabla n^-$ due to the relation between the densities $\rho^\pm$ and the fraction densities $n^\pm$, it also has an additional $(1+t)^{-\frac{1}{2}}$-decay rate than $n^\pm$ due to the corresponding estimate in \ref{3.12(1)}.

Similarly, one has
\begin{eqnarray}
Q_2&=&\mathcal{C}^{2}(n^+\!+\!1,n^-\!+\!1)n^-\Big(\frac{\rho^+}{\rho^-}\nabla n^-+\nabla n^+\Big)\nonumber\\
&&+\Big(\frac{\mathcal{C}^2(n^+\!+\!1,n^-\!+\!1)\rho^+}{\rho^-}-\frac{\mathcal{C}^2(1,1)\bar\rho^+}{\bar\rho^-}\Big)\nabla n^-+(\mathcal{C}^2(n^+\!+\!1,n^-\!+\!1)-\mathcal{C}^2(1,1))\nabla n^+\nonumber\\
&=&\frac{\mathcal{C}^2(n^+\!+\!1,n^-\!+\!1)}{\rho^-}n^-\bigg[(\rho^--\bar\rho^-)\nabla n^-\!+\!(\rho^+-\bar\rho^+)\nabla n^+\bigg]\nonumber\\
&&+\frac{\mathcal{C}^2(n^+\!+\!1,n^-\!+\!1)}{\rho^-}n^-(\uwave{\bar\rho^-\nabla n^++\bar\rho^+\nabla n^-})\nonumber\\
&&+\frac{\mathcal{C}^2(n^+\!+\!1,n^-\!+\!1)-\mathcal{C}^2(1,1)}{\rho^-}[(\rho^--\bar\rho^-)\nabla n^-+(\rho^+-\bar\rho^+)\nabla n^+]\nonumber\\
&&+\frac{1}{\rho^+}(\mathcal{C}^2(n^+\!+\!1,n^-\!+\!1)-\mathcal{C}^2(1,1))(\uwave{\bar\rho^-\nabla n^++\bar\rho^+\nabla n^-})\nonumber\\
&&+\mathcal{C}^2(1,1)\bigg(\frac{1}{\rho^-}-\frac{1}{\bar\rho^-}\bigg)(\uwave{\bar\rho^-\nabla n^++\bar\rho^+\nabla n^-})\nonumber\\
&&+\mathcal{C}^2(1,1)\bigg(\frac{1}{\rho^-}-\frac{1}{\bar\rho^-}\bigg)[(\rho^--\bar\rho^-)\nabla n^-+(\rho^+-\bar\rho^+)\nabla n^+]\nonumber\\
&&+\frac{\mathcal{C}^2(1,1)}{\bar\rho^-}[(\uwave{\rho^--\bar\rho^-})\nabla n^-+(\uwave{\rho^+-\bar\rho^+})\nabla n^+]:=\sum\limits_{i=8}^{14} \mathcal{S}_i.\label{5.9(1)}
\label{5.9}
\end{eqnarray}

 As a result, the convolution of the Riesz wave-I\!V and the nonlinear terms about the pressure terms $\alpha^\pm\nabla P$ has the following refine form instead of that in (\ref{5.8}):
\begin{eqnarray}
\mathcal{N}_1= \int_0^t\!\int_{\mathbb{R}^3}(1+t-\tau)^{-1}\Big(1+\frac{|x-y|^2}{1+t-\tau}\Big)^{-1}(1+\tau)^{-4}\Big(1+\frac{(|y|-{\rm c}\tau)^2}{1+\tau}\Big)^{-2}dyd\tau.\label{5.10}
\end{eqnarray}
Based on the estimate $K_4$ established in Lemma \ref{A.5}, one has
\begin{eqnarray}
\mathcal{N}_1\lesssim (1+t)^{-1}\Big(1+\frac{|x|^2}{1+t}\Big)^{-1}
          +(1+t)^{-\frac{3}{2}}\Big(1+\frac{(|x|-{\rm c}t)^2}{1+t}\Big)^{-1}.\label{5.11}
\end{eqnarray}

On the other hand, we consider another difficulty in dealing with the interaction of the Huygens wave in Green's function and the Huygens waves in nonlinear terms for the non-conservative fluid models. The estimate is totally new, which is given in $K_7$ in Lemma \ref{A.5}. In particular, we have for arbitrarily large constant $N>0$ that
\begin{align}
\mathcal{N}_2=&\int_{0}^{t}\!\int_{\mathbb{R}^3}(1+t-\tau)^{-2}\Big(1+\frac{(|x-y|-{\rm c}(t-\tau))^2}{1+t-\tau}\Big)^{-N}(1+\tau)^{-4}\Big(1+\frac{(|y|-{\rm c}\tau)^2}{1+\tau}\Big)^{-2}dyd\tau\nonumber\\[3mm]
\lesssim &\ (1+t)^{-\frac{3}{2}}\Big(1+\frac{|x|^2}{1+t}\Big)^{-\frac{3}{2}}+(1+t)^{-2}\Big(1+\frac{(|x|-{\rm c}t)^2}{1+t}\Big)^{-1}.\label{5.12}
\end{align}
This is different from the classical estimate for the conservative system given in \cite{Ls,Lw}:
\begin{eqnarray}
&&\int_{0}^{t}\int_{\mathbb{R}^3}\uwave{(1+t-\tau)^{-\frac{5}{2}}}\Big(1+\frac{(|x-y|-{\rm c}(t-\tau))^2}{1\!+\!t\!-\!\tau}\Big)^{-N}(1+\tau)^{-4}\Big(1+\frac{(|y|-{\rm c}\tau)^2}{1\!+\!\tau}\Big)^{-3}dyd\tau\nonumber\\
&\lesssim& (1+t)^{-2}\Big(\big(1+\frac{|x|^2}{1+t}\big)^{-\frac{3}{2}}+\Big(1+\frac{(|x|-{\rm c}t)^2}{1+t}\Big)^{-\frac{3}{2}}\Big),\label{5.12(1)}
\end{eqnarray}
where the extra $(1+t)^{-\frac{1}{2}}$-decay of Green's function (the underlined term in (\ref{5.12(1)})) is arising from the conservation of the nonlinear fluid model studied in the previous work \cite{Ls,Lw,Wu4}. 

Going back to the estimate (\ref{5.12}), we have to overcome the difficulty in the subregion space-time estimates caused by the deteriorated time decay of Green's function in the convolution due to the non-conservation, and simultaneously ensure that the obtained pointwise estimates are optimal in terms of time decay (consistent with the $L^2$-decay rate derived by combining the energy method with spectral analysis), we have carried out more elaborate calculations in the subregion estimates. In particular, we have fully exploited the interchange between temporal and spatial variables within specific regions, and thus obtained relatively sharp pointwise space-time estimates. See the details in the proof of $K_7$ in Appendix.

Finally, we shall close the ansatz on the solution of the nonlinear problem by substituting (\ref{5.4})-(\ref{5.6}) into the representation of the solution $(n^+,\mathbf{m}^+,n^-,\mathbf{m}^-)$ in (\ref{3.1})-(\ref{3.2}) and using the aforementioned nonlinear convolution estimates. Note that the singular term $G_{S2}$ in $G_{21}$ and $G_{43}$ correspond to the nonlinear term ``0" in the two continuity equations of the system (\ref{1.1}). Hence, in the following, we only focus on the unique singular term $G_{S1}$ when dealing with nonlinear coupling. In particular, we have to split Green's function into the regular term $G_{ij}-G_{S1}$ and the singular term $G_{S1}$. Then, one can put all of the derivatives on $G_{ij}-G_{S1}$ to deal with the convolution between the regular part $G_{ij}-G_{S1}$ and the nonlinear terms, however, one have to put all of the derivatives on the nonlinear terms to deal with the convolution of the singular part and the nonlinear terms.

We just take nonlinear coupling of the fraction densities $n^+$ for example, since that of $n^-$ is the same and those of $\mathbf{m}^\pm$ can be estimated similarly. To facilitate the description, we use $\tilde{n}^+$ to denote the nonlinear part of $n^+$ in (\ref{5.1}). Then by using Proposition \ref{l 3.3} and the nonlinear convolution estimates in Lemma \ref{A.5}, one has for $0\leq k\leq2$ that
\begin{align}\label{5.12}
|\partial_x^k\tilde{n}^+|\leq &\int_0^t|\partial_x^k(G_{12}-G_{S1})(\cdot,t-\tau)\ast F_1(\cdot,\tau)|d\tau+\int_0^t|\partial_x^k(G_{14}-G_{S1})(\cdot,t-\tau)\ast F_2(\cdot,\tau)|d\tau\nonumber\\
&+\int_0^t|G_{S1}(\cdot,t-\tau)\ast \partial_x^k (F_1(\cdot,\tau),F_2(\cdot,\tau))|d\tau\nonumber\\
\lesssim &\ 2C\epsilon\bigg((1+t)^{-1}\Big(1+\frac{|x|^2}{1+t}\Big)^{-1} +(1+t)^{-\frac{3}{2}}\Big(1+\frac{(|x|-{\rm c}t)^2}{1+t}\Big)^{-1}\bigg)\nonumber\\
      &+\int_0^t|G_{S1}(\cdot,t-\tau)\ast \partial_x^k (n^+\partial_x^3n^++\partial_x^2(n^+\mathbf{m}^+)+\cdots)(\cdot,\tau)|d\tau\nonumber.   
\end{align}
When $k=0$, we need the pointwise ansatz of $\partial_x^3 n^\pm$ and $\partial_x^2\mathbf{m}^\pm$ according to the nonlinear terms and the nonlinear convolution in Lemma \ref{A.5}. Then when $k=3$, it also requires the $L^\infty$-estimate of $\partial_x^6n^\pm$ and $\partial_x^5\mathbf{m}^\pm$, which results that we can close the ansatz in $H^8\times H^7$-framework based on the Sobolev embedding theorem. Finally, by using the smallness of $\epsilon$, one can close the ansatz (\ref{5.5})-(\ref{5.6}) and hence proves Theorem \ref{l 1.1}.

\section{Appendix: Analytic tools}\label{1section_appendix}

Some useful lemmas are given here. The first one is used to derive the pointwise estimates of Green's function in the low frequency. Note that we have replaced $t^{-(n+|\alpha|)}B_N(|x|,t)$ in Wang-Yang\cite{Wang2} with $(1+t)^{-(n+|\alpha|)}B_N(|x|,t)$ in the following lemma since it is in the low frequency part.


\begin{lemma}\label{A.1}[Wu-Wang\cite{Wu4}] If there exists a
constant $C>0$ such that when $|\xi|\leq1$, $\hat{f}(\xi,t)$ satisfies
$$
|D_\xi^\beta(\xi^\alpha\hat{f}(\xi,t))|\leq
C(|\xi|^{(|\alpha|-|\beta|)_+}
+|\xi|^{|\alpha|}t^{|\beta|/2})(1+(t|\xi|^2))^a \exp(-b|\xi|^2t),
$$
for some constant $b>0$ and any multi-indexes $\alpha, \beta$ with $|\beta|\leq 2N$, then
\begin{equation}\label{7.1}
|D_x^\alpha f(x,t)|\leq C_N (1+t)^{-(n+|\alpha|)}B_N(|x|,t),
\end{equation}
where $a$ is any fixed integer, $(e)_+=\max(0,e)$ and
$$
B_N(|x|,t)=\Big(1+\frac{|x|^2}{1+t}\Big)^{-N}.
$$
 \end{lemma}


The next two lemmas are often used to deal with initial propagation and nonlinear coupling, respectively. We also state several typical cases for completeness.

\begin{lemma}\label{A.3}
There exists a constant $C>0$, such that
\begin{eqnarray*}
\mathcal{I}_1&:=&\int_{\mathbb{R}^3} \Big(1+\frac{|x-y|^2}{1+t}\Big)^{-1}(1+|y|^2)^{-r_1}dy\leq C\left(1+\frac{|x|^2}{1+t}\right)^{-1},\ {\rm for}\ r_1>\frac{3}{2};\\
\mathcal{I}_2&:=&\int_{\mathbb{R}^3}\left(1+\frac{|x-y|^2}{1+t}\right)^{-\frac{3}{2}}(1+|y|^2)^{-r_1}dy\leq C\left(1+\frac{x^2}{1+t}\right)^{-\frac{3}{2}},\ {\rm for}\ r_1>\frac{3}{2};\\
\mathcal{I}_3&:=&\int_{\mathbb{R}^3} \bigg(1+\frac{(|x-y|-{\rm c}t)^2}{1+t}\bigg)^{-N}(1+|y|^2)^{-r_2}dy\leq C\left(1+\frac{(|x|\!-\!{\rm c}t)^2}{1+t}\right)^{-1},\ {\rm for}\ r_2>\frac{19}{10},
\end{eqnarray*}
where the positive constant $N$ is suitably large.
\end{lemma}
\begin{proof}
The estimates $\mathcal{I}_1$ and $\mathcal{I}_2$ are used to deal with the initial propagation for the convolution of Riesz wave-I and Riesz wave-I\!I respectively, and the proof of them can be seen in \cite{Wu5}. Hence, we only prove $\mathcal{I}_3$. When $(|x|-{\rm c}t)^2\leq 4(1+t)$,
$$
\mathcal{I}_3\leq C\leq C\left(1+\frac{(|x|-{\rm c}t)^2}{1+t}\right)^{-r_1}.
$$
When $|x|-{\rm c}t\geq 2\sqrt{1+t}$, we break integration into two parts. If $|y|\geq\frac{|x|-ct}{2}$, then 
\bess
\mathcal{I}_3&\leq & C(1+(|x|-{\rm c}t)^2)^{-1}\int_{\mathbb{R}^3} \bigg(1+\frac{(|x-y|-{\rm c}t)^2}{1+t}\bigg)^{-N}(1+|y|^2)^{-(r_2-1)}dy\nm\\
&\leq& C(1+(|x|-{\rm c}t)^2)^{-1}(1+t)\left(\int_{\mathbb{R}^3}(1+|y|^2)^{-\frac{5}{3}(r_2-1)}dy\right)^{\frac{3}{5}}\nm\\
&\leq& C\left(1+\frac{(|x|\!-\!{\rm c}t)^2}{1+t}\right)^{-1},\ {\rm for}\ r_2>\frac{19}{10}.
\eess
Here we have used Young inequality and the fact
$$
\int_{\mathbb{R}^n}\left(1+\frac{(|y|-a)^2}{b}\right)^{-r_2}dy\leq C(b^{\frac{n}{2}}+b^{\frac{1}{2}}a^{n-1}),\ {\rm for}\ r_2>\frac{n}{2}.
$$
If $|y|<\frac{|x|-{\rm c}t}{2}$, then $|x-y|-{\rm c}t\geq\frac{||x|-{\rm c}t|}{2}$. Thus, it holds that
\bess
\mathcal{I}_3\leq C\bigg(1+\frac{(|x|-{\rm c}t)^2}{1+t}\bigg)^{\!-N/2} \!\!\int_{\mathbb{R}^3}\!\!\bigg(1+\frac{(|x-y|-{\rm c}t)^2}{1+t}\bigg)^{\!-N/2}\!\!(1+|y|^2)^{-r_1}dy
\leq C\left(1+\frac{(|x|\!-\!{\rm c}t)^2}{1+t}\right)^{\!-\frac{N}{2}}\!\!.
\eess
The last case: ${\rm c}t-|x|\geq 2\sqrt{1+t}$ can be similarly treated as the case: $|x|-{\rm c}t\geq 2\sqrt{1+t}$. This proves the third estimate.
\end{proof}


\begin{lemma}\label{A.4}[Liu-Noh\cite{Ls}] There exists a constant $C>0$ such that
\begin{equation*}\label{6.1}
\begin{array}{ll}
\displaystyle K_1=\!\int_0^t\!\!\int_{\mathbb{R}^3}(1+t-\tau)^{-2}\Big(1+\frac{|x-y|^2}{1\!+\!t\!-\!s}\Big)^{-2}(1+\tau)^{-3}\Big(1+\frac{|y|^2}{1\!+\!\tau}\Big)^{-3}\!\!dyd\tau
\leq C(1+t)^{-2}\Big(1+\frac{|x|^2}{1+t}\Big)^{-\frac{3}{2}},\\[3.5mm]
\displaystyle K_2=\!\int_0^t\!\!\int_{\mathbb{R}^3}(1+t-\tau)^{-2}\Big(1+\frac{|x-y|^2}{1+t-\tau}\Big)^{-2}(1+\tau)^{-4}\Big(1+\frac{(|y|-{\rm c}\tau)^2}{1+\tau}\Big)^{-3}dyd\tau\\[2mm]
\ \ \ \ \leq \ \displaystyle C(1+t)^{-2}\Big(\Big(1+\frac{|x|^2}{1+t}\Big)^{-\frac{3}{2}}+\Big(1+\frac{(|x|-{\rm c}t)^2}{1+t}\Big)^{-\frac{3}{2}}\Big),\\[2mm]
K_3=\displaystyle\! \int_{0}^{t}\!\!\int_{\mathbb{R}^3}(1+t-\tau)^{-\frac{5}{2}}\Big(1+\frac{(|x-y|-{\rm c}(t-\tau))^2}{1\!+\!t\!-\!s}\Big)^{-N}(1+\tau)^{-4}\Big(1+\frac{(|y|-{\rm c}\tau)^2}{1\!+\!s}\Big)^{-3}dyd\tau\\[2mm]
\ \ \ \ \leq \ \displaystyle C(1+t)^{-2}\Big(\big(1+\frac{|x|^2}{1+t}\big)^{-\frac{3}{2}}+\Big(1+\frac{(|x|-{\rm c}t)^2}{1+t}\Big)^{-\frac{3}{2}}\Big),
\end{array}
\end{equation*}
where the constant $N>0$ can be arbitrarily large.
\end{lemma}

Due to the non-conservation of the system (\ref{1.1}), and the slower decay rate of the two fraction densities of the system (\ref{1.1}), we have to construct some refined nonlinear convolution estimates, especially for the convolutions containing the Huygens waves or the Riesz\ wave-I\!V (from Riesz operator of negative first order) and the Riesz\ wave-I\!I\!I (from Riesz operator of zero order). In fact, Lin, Wang in \cite{lin1,lin2} investigated the Boltzmann equation and proved that
\begin{align}
&\int_0^t\!\int_{\mathbb{R}^3}\uwave{\mathbf{1}_{|x|\leq {\rm c}t}(1+t-\tau)^{-2}}\Big(1+\frac{|x-y|^2}{1+t-\tau}\Big)^{-\frac{3}{2}}(1+\tau)^{-4}\Big(1+\frac{(|y|-{\rm c}\tau)^2}{1+\tau}\Big)^{-2}dyd\tau\nonumber\\
\lesssim &\ (1+t)^{-2}\Big(1+\frac{|x|^2}{1+t}\Big)^{-\frac{3}{2}}+(1+t)^{-\frac{5}{2}}\Big(1+\frac{(|x|-{\rm c}t)^2}{1+t}\Big)^{-1}.\label{6.1(1)}
\end{align}
Here the underlined is arising from the aforementioned Riesz wave-I and the conservative structure. Different from (\ref{6.1(1)}),  we obtained a similar result in the following $K_5$  without the factor $\mathbf{1}_{|x|\leq {\rm c}t}$. Moreover, we also achieve a satisfactory pointwise space-time description of the Riesz wave-I\!V convolved with the Huygens wave in $K_4$.

\begin{lemma}\label{A.5} There exists a constant $C>0$ such that

\noindent {\rm (``Riesz\ wave-I\!V"\ convolved\ with\ ``Huygens\ wave")}
\begin{align}
K_4=& \int_0^t\!\int_{\mathbb{R}^3}
(1+t-\tau)^{-1}\Big(1+\frac{|x-y|^2}{1+t-\tau}\Big)^{-1}(1+\tau)^{-4}\Big(1+\frac{(|y|-{\rm c}\tau)^2}{1+\tau}\Big)^{-2}dyd\tau\nonumber\\
\leq &\ C(1+t)^{-1}\Big(1+\frac{|x|^2}{1+t}\Big)^{-1}+C(1+t)^{-\frac{3}{2}}\Big(1+\frac{(|x|-{\rm c}t)^2}{1+t}\Big)^{-1},\label{6.2}
\end{align}
\noindent {\rm (``Riesz\ wave-I\!I\!I"\ convolved\ with\ ``Huygens\ wave")}
\begin{align}
K_5=&\int_0^t\!\int_{\mathbb{R}^3}(1+t-\tau)^{-\frac{3}{2}}\Big(1+\frac{|x-y|^2}{1+t-\tau}\Big)^{-\frac{3}{2}}(1+\tau)^{-4}\Big(1+\frac{(|y|-{\rm c}\tau)^2}{1+\tau}\Big)^{-2}dyd\tau\nonumber\\
\leq &\ C(1+t)^{-\frac{3}{2}}\Big(1+\frac{|x|^2}{1+t}\Big)^{-\frac{3}{2}}+C(1+t)^{-2}\Big(1+\frac{(|x|-{\rm c}t)^2}{1+t}\Big)^{-1}.\label{6.4}
\end{align}
\end{lemma}

We mainly focus on the proof of (\ref{6.2}) due to  the slowest decay on both the temporal and spacial variables of Riesz wave-I\!V in Green's function, and  the other estimate could be derived similarly.
To facilitate the estimates, we first decompose space-time domain into the
following $5$ regions:
\begin{align*}
D_{1}& =\left \{ |x|\leq \sqrt{1+t}\right \} \,, \\
\ D_{2}& =\left \{ {\rm c}t-\sqrt{1+t}\leq |x|\leq {\rm c}t+\sqrt{1+t}%
\right \} \,, \\
D_{3}& =\left \{ |x|\geq {\rm c}t+\sqrt{1+t}\right \} \,, \\
D_{4}& =\Big \{ \sqrt{1+t}\leq |x|\leq \frac{1}{2}{\rm c}t\Big \} \,,
\\
D_{5}& =\Big \{ \frac{1}{2}{\rm c}t\leq |x|\leq {\rm c}t-\sqrt{1+t}%
\Big \} \,.
\end{align*}

\vspace{3mm}

\noindent\textit{\textbf{Proof of $K_4$}}. (Riesz wave-I\!V convolved with Huygens wave)

\textbf{Case 1:} $\left( x,t\right) \in D_{1}$. Direct computation gives%
\begin{eqnarray*}
K_4 &=&\left( \int_{0}^{\frac{t}{2}}+\int_{\frac{t}{2}}^{t}\right) \int_{%
\mathbb{R}^{3}}\left( \cdots \right) dyd\tau \\
&\lesssim &\left( 1+t\right) ^{-1}\int_{0}^{\frac{t}{2}}\int_{\mathbb{R}%
^{3}}\left( 1+\tau \right) ^{-4}\bigg( 1+\frac{\left( \left \vert y\right
\vert -{\rm c}\tau \right) ^{2}}{1+\tau }\bigg) ^{-2}dyd\tau \\
&&+\left( 1+t\right) ^{-4}\int_{\frac{t}{2}}^{t}\left( 1+t-\tau \right)
^{-1}(1+t-\tau)^{\frac{3}{4}}(1+\tau)^{\frac{5}{4}}d\tau \\
&\lesssim &\left( 1+t\right) ^{-1}+\left( 1+t\right) ^{-2} \lesssim ( 1+t) ^{-1}\bigg( 1+\frac{\left \vert
x\right \vert ^{2}}{1+t}\bigg) ^{-1}\hbox{,}
\end{eqnarray*}
where we have used Young inequality and the fact
\begin{equation}\label{6.8}
\int_{\mathbb{R}^{3}}\bigg( 1+\frac{\left( \left \vert y\right
\vert -{\rm c}\tau \right) ^{2}}{1+\tau }\bigg) ^{-2}dy\lesssim (1+\tau)^{\frac{5}{2}}.
\end{equation}

\textbf{Case 2:} $\left( x,t\right) \in D_{2}$. We split the integral $I$
into two parts%
\begin{equation*}
K_4=\left( \int_{\frac{t}{4}}^{t}+\int_{0}^{\frac{t}{4}}\right) \int_{\mathbb{R%
}^{3}}\left( \cdots \right) dyd\tau =:I_{1}+I_{2}\hbox{.}
\end{equation*}

For $I_{1}$, one can see that
\begin{eqnarray*}
I_{1} &\lesssim &\left( 1+t\right) ^{-4}\int_{\frac{t}{4}}^{t}\left( 1+t-\tau \right)
^{-1}(1+t-\tau)^{\frac{3}{4}}(1+\tau)^{\frac{5}{4}}d\tau \\
&\lesssim & \left( 1+t\right) ^{-2}\lesssim \left( 1+t\right) ^{-2}\left( 1+\frac{\left( \left \vert
x\right \vert -{\rm c}t\right) ^{2}}{1+t}\right) ^{-1}\hbox{.}
\end{eqnarray*}

For $I_{2}$, we further decompose $\mathbb{R}^{3}$ into two parts%
\begin{equation*}
I_{2}=\left( \int_{0}^{\frac{t}{4}}\int_{\left \vert y\right \vert \leq
\frac{2}{3}\left \vert x\right \vert }+\int_{0}^{\frac{t}{4}}\int_{\left
\vert y\right \vert >\frac{2}{3}\left \vert x\right \vert }\right) \left(
\cdots \right) dyd\tau =I_{11}+I_{12}\hbox{.}
\end{equation*}%
If $\left \vert y\right \vert \leq \frac{2}{3}\left \vert x\right \vert $,
then we have
\begin{equation*}
\left \vert x-y\right \vert \geq \left \vert x\right \vert -\left \vert
y\right \vert \geq \frac{\left \vert x\right \vert }{3}\hbox{,}
\end{equation*}%
and thus
\begin{eqnarray*}
I_{11} &\lesssim &\left( 1+t\right) ^{-1}\bigg( 1+\frac{\left \vert x\right
\vert ^{2}}{1+t}\bigg) ^{-1}\int_{0}^{\frac{t}{4}}\int_{\mathbb{R}%
^{3}}\left( 1+\tau \right) ^{-4}\bigg( 1+\frac{\left( \left \vert y\right
\vert -{\rm c}\tau \right) ^{2}}{1+\tau }\bigg) ^{-2}dyd\tau \\
&\lesssim &\left( 1+t\right) ^{-1}\bigg( 1+\frac{\left \vert x\right \vert
^{2}}{1+t}\bigg) ^{-1}\int_{0}^{\frac{t}{4}}\left( 1+\tau \right)
^{-4}\left( 1+\tau \right) ^{\frac{5}{2}}d\tau
\lesssim \left( 1+t\right) ^{-1}\bigg( 1+\frac{\left \vert x\right \vert
^{2}}{1+t}\bigg) ^{-1}\hbox{.}
\end{eqnarray*}%
If $\left \vert y\right \vert >\frac{2}{3}\left \vert x\right \vert $ and $%
0\leq \tau \leq \frac{t}{4}$,
\begin{equation*}
\left \vert y\right \vert -{\rm c}\tau \geq \frac{2}{3}\left \vert
x\right \vert -{\rm c}\tau \geq \frac{\left \vert x\right \vert }{3}+%
\frac{1}{3}\left( {\rm c}t-\sqrt{1+t}\right) -\frac{{\rm c}t}{4}\geq
\frac{\left \vert x\right \vert }{3}+\frac{1}{24}{\rm c}t
\end{equation*}%
$\allowbreak \allowbreak \allowbreak $for $t\geq \frac{32}{{\rm c}^2}$. Note that $%
\left( x,t\right) \in D_{2}$ with $t\geq \frac{2+2\sqrt{1+{\rm c}^{2}}}{%
{\rm c}^{2}}$ implies that $\left \vert x\right \vert
\geq \sqrt{1+t}$. Hence for $t\geq \frac{32}{{\rm c}^2}$,
\begin{eqnarray*}
I_{12} &\lesssim &(1+t)^{-1}\left( 1+\left \vert x\right \vert ^{2}\right) ^{-1}\int_{0}^{\frac{t}{4}}\int_{\left \vert y\right \vert >\frac{2}{3}\left
\vert x\right \vert }\left( 1+t-\tau \right) ^{-1}\bigg( 1+\frac{\left
\vert x-y\right \vert ^{2}}{ 1+t-\tau  }\bigg) ^{-1} \\
&&\cdot \left( 1+\tau \right) ^{-1}\bigg( 1+\frac{\left( \left
\vert y\right \vert -{\rm c}\tau \right) ^{2}}{1+\tau }\bigg) ^{-1}dyd\tau \\
&\lesssim &\left( 1+t \right) ^{-1}\left( 1+\left \vert x\right \vert ^{2}\right) ^{-1}\int_{0}^{\frac{t}{4}}(1+t-\tau)^{-1}(1+t-\tau)^{\frac{3}{4}}(1+\tau)^{-1}(1+\tau)^{-\frac{5}{4}}d\tau \\
&\lesssim &\left( 1+t \right) ^{-1}\left( 1+\left \vert x\right \vert ^{2}\right) ^{-1}(1+t)\\
&\lesssim& (1+t)^{-1}\bigg( 1+\frac{\left \vert
x\right \vert ^{2}}{1+t}\bigg) ^{-1}\hbox{,}
\end{eqnarray*}%
and for $0\leq t\leq 10$,
\begin{eqnarray*}
I_{12} &\lesssim &\left( 1+t\right) ^{-1}\int_{0}^{\frac{t}{4}%
}\int_{\mathbb{R}^{3}}\left( 1+\tau \right) ^{-4}\left( 1+\frac{\left( \left
\vert y\right \vert -{\rm c}\tau \right) ^{2}}{1+\tau }\right)
^{-2}dyd\tau \\
&\lesssim &C\lesssim \left( 1+t\right) ^{-2}\bigg( 1+\frac{\left( \left
\vert x\right \vert -{\rm c}t\right) ^{2}}{1+t}\bigg) ^{-1}\hbox{.}
\end{eqnarray*}%
Thereupon,
\begin{equation*}
K_4\lesssim \left( 1+t\right) ^{-1}\bigg( 1+\frac{\left \vert x\right \vert
^{2}}{1+t}\bigg) ^{-1}+\left( 1+t\right) ^{-2}\bigg( 1+\frac{%
\left( \left \vert x\right \vert -{\rm c}t\right) ^{2}}{1+t}\bigg) ^{-1}%
\hbox{.}
\end{equation*}
\newline
\textbf{Case 3:} $\left( x,t\right) \in D_{3}$. We split the integral $I$
into two parts%
\begin{equation*}
K_4=\left( \int_{0}^{\frac{t}{2}}+\int_{\frac{t}{2}}^{t}\right) \int_{\mathbb{R%
}^{3}}\left( \cdots \right) dyd\tau =:I_{1}+I_{2}\hbox{.}
\end{equation*}

For $I_{1}$, we decompose $\mathbb{R}^{3}$ into two parts as follows:
\begin{equation*}
I_{1}=\int_{0}^{\frac{t}{2}}\left( \int_{\left \vert y\right \vert \leq
\frac{2}{3}\left \vert x\right \vert }+\int_{\left \vert y\right \vert >%
\frac{2}{3}\left \vert x\right \vert }\right) \left( \ldots \right) dyd\tau
=:I_{11}+I_{12}\hbox{.}
\end{equation*}%
If $\left \vert y\right \vert \leq \frac{2}{3}\left \vert x\right \vert $,
then
\begin{equation*}
\left \vert x-y\right \vert \geq \frac{\left \vert x\right \vert }{3}\hbox{,}
\end{equation*}%
and thus%
\begin{eqnarray*}
I_{11} &\lesssim &\left( 1+t\right) ^{-1}\bigg( 1+\frac{\left \vert x\right
\vert ^{2}}{1+t}\bigg) ^{-1}\int_{0}^{\frac{t}{2}}\int_{\mathbb{R}%
^{3}}\left( 1+\tau \right) ^{-4}\bigg( 1+\frac{\left( \left \vert y\right
\vert -{\rm c}\tau \right) ^{2}}{1+\tau }\bigg) ^{-2}dyd\tau \\
&\lesssim &\left( 1+t\right) ^{-1}\bigg( 1+\frac{\left \vert x\right \vert
^{2}}{1+t}\bigg) ^{-1}\hbox{.}
\end{eqnarray*}%
If $\left \vert y\right \vert >\frac{2}{3}\left \vert x\right \vert $ and $%
0\leq \tau \leq \frac{t}{2}$, we have%
\begin{equation*}
\left \vert y\right \vert -{\rm c}\tau \geq \frac{2}{3}\left \vert
x\right \vert -\frac{{\rm c}t}{2}=\frac{1}{6}\left \vert x\right \vert +%
\frac{1}{2}\left( \left \vert x\right \vert -{\rm c}t\right) \geq \frac{1%
}{6}{\rm c}t\hbox{,}
\end{equation*}%
so that%
\begin{eqnarray*}
I_{12} &\lesssim &\big( 1+\left \vert x\right \vert ^{2}\big) ^{\!-1}\!\int_{0}^{\frac{t}{2}}\!\int_{\left \vert y\right \vert >\frac{2}{3}\left
\vert x\right \vert }\left( 1\!+\!t\!-\!\tau \right) ^{-1}\left( 1+\frac{\left
\vert x-y\right \vert ^{2}}{ 1\!+\!t\!-\!\tau  }\right) ^{-1} 
\left( 1+\tau \right) ^{-3}\left( 1\!+\!\frac{\left( \left
\vert y\right \vert -{\rm c}\tau \right) ^{2}}{1+\tau }\right) ^{\!-1}\!\!dyd\tau  \\
&\lesssim &\big( 1+\left \vert x\right \vert ^{2}\big) ^{-1}\int_{0}^{\frac{t}{2}}(1+t-\tau)^{-1}(1+t-\tau)^{\frac{3}{4}}(1+\tau)^{-3}(1+\tau)^{\frac{5}{4}}d\tau \\
&\lesssim &(1+t)^{-\frac{1}{4}}\big( 1+\left \vert x\right \vert ^{2}\big) ^{-1}\lesssim (1+t)^{-\frac{5}{4}}\bigg( 1+\frac{\left \vert x\right
\vert ^{2}}{1+t}\bigg) ^{-1},
\end{eqnarray*}%
since $\left \vert x\right \vert \geq \sqrt{1+t}$. Therefore,%
\begin{equation*}
I_{1}\lesssim \left( 1+t\right) ^{-1}\bigg( 1+\frac{\left \vert x\right
\vert ^{2}}{1+t}\bigg) ^{-1}\hbox{.}
\end{equation*}

For $I_{2}$, we decompose $\mathbb{R}^{3}$ into two parts
\begin{equation*}
I_{2}=\int_{\frac{t}{2}}^{t}\left( \int_{\left \vert y\right \vert -\mathbf{c%
}\tau \leq \frac{\left \vert x\right \vert -ct}{2}}+\int_{\left \vert
y\right \vert -{\rm c}\tau >\frac{\left \vert x\right \vert -ct}{2}%
}\right) \left( \cdots \right) dyd\tau =I_{21}+I_{22}\hbox{. }
\end{equation*}%
If $\left \vert y\right \vert -{\rm c}\tau \leq \frac{\left \vert
x\right
\vert -{\rm c}t}{2}$, we have
\begin{equation*}
\left \vert x-y\right \vert \geq \left \vert x\right \vert - \frac{%
\left \vert x\right \vert -{\rm c}t}{2} -{\rm c}\tau \geq \frac{%
\left \vert x\right \vert -{\rm c}t}{2}\hbox{,}
\end{equation*}%
and thus%
\begin{eqnarray*}
I_{21} &\lesssim &\left( 1+\left( \left \vert x\right \vert -{\rm c}%
t\right) ^{2}\right) ^{-1}\int_{\frac{t}{2}}^{t}\int_{\mathbb{R}%
^{3}}\left( 1+\tau \right) ^{-4}\bigg( 1+\frac{\left( \left
\vert y\right \vert -{\rm c}\tau \right) ^{2}}{1+\tau }\bigg)
^{-2}dyd\tau \\
&\lesssim &\left( 1+\left( \left \vert x\right \vert -{\rm c}%
t\right) ^{2}\right) ^{-1}\int_{\frac{t}{2}}^t(1+\tau)^{-4}(1+\tau)^{\frac{5}{2}}d\tau \\
&\lesssim &(1+t)^{-\frac{1}{2}}\big( 1+\left( \left \vert x\right \vert -{\rm c}%
t\right) ^{2}\big) ^{-1}\lesssim (1+t)^{-\frac{3}{2}}\bigg(1+\frac{(|x|-ct)^2}{1+t}\bigg)^{-1}\hbox{,}
\end{eqnarray*}%
since $|x|-ct>\sqrt{1+t}$.
\newline
If $\left \vert y\right \vert -{\rm c}\tau >\frac{\left \vert
x\right
\vert -{\rm c}t}{2}$, then
\begin{eqnarray*}
I_{22} &\lesssim &\left( 1+t\right) ^{-4}\bigg( 1+\frac{\left( \left \vert
x\right \vert -{\rm c}t\right) ^{2}}{1+t}\bigg) ^{-1}\\
&& \times\int_{\frac{t}{2}%
}^{t}\int_{\mathbb{R}^3}\left( 1+t-\tau \right) ^{-1}\bigg( 1+\frac{\left \vert x-y\right
\vert ^{2}}{ 1+t-\tau }\bigg) ^{-1} \bigg( 1+\frac{\left( \left
\vert y\right \vert -{\rm c}\tau \right) ^{2}}{1+\tau }\bigg)
^{-1}   dyd\tau \\
&\lesssim &\left( 1+t\right) ^{-4}\bigg( 1+\frac{\left( \left \vert x\right
\vert -{\rm c}t\right) ^{2}}{1+t}\bigg) ^{-1}\int_{\frac{t}{2}%
}^{t}(1+t-\tau)^{-1} \left( 1+t-\tau \right)^{\frac{3}{4}}(1+\tau)^{\frac{5}{4}} d\tau \\
& \lesssim& \left( 1+t\right) ^{-2}\bigg( 1+\frac{\left( \left \vert x\right
\vert -{\rm c}t\right) ^{2}}{1+t}\bigg) ^{-1}\hbox{.}
\end{eqnarray*}%
Therefore,
\begin{equation*}
I_{2}\lesssim \left( 1+t\right) ^{-3/2}\bigg( 1+\frac{\left( \left \vert
x\right \vert -{\rm c}t\right) ^{2}}{1+t}\bigg) ^{-1}\hbox{.}
\end{equation*}%
Combining all the estimates, we get desired estimate%
\begin{equation*}
K_4\lesssim \left( 1+t\right) ^{-1}\bigg( 1+\frac{\left \vert x\right \vert
^{2}}{1+t}\bigg) ^{-1}+\left( 1+t\right) ^{-3/2}\bigg( 1+\frac{%
\left( \left \vert x\right \vert -{\rm c}t\right) ^{2}}{1+t}\bigg) ^{-1}\hbox{.}
\end{equation*}
\newline
\textbf{Case 4:} $\left( x,t\right) \in D_{4}$. In this region $\sqrt{1+t}%
\leq \left \vert x\right \vert \leq \frac{{\rm c}t}{2}$. We split the
integral $K_4 $ into three parts:%
\begin{eqnarray*}
K_4&=&\int_{0}^{t}\int_{\mathbb{R}^{3}}\left( 1+t-\tau \right) ^{-4}\left( 1+%
\frac{\left( \left \vert x-y\right \vert -{\rm c}\left( t-\tau \right)
\right) ^{2}}{1+t-\tau }\right) ^{-2}\left( 1+\tau \right) ^{-1}\left( 1+%
\frac{\left \vert y\right \vert ^{2}}{ 1+\tau }\right) ^{-1}dyd\tau \\
&=&\left( \int_{0}^{\frac{t}{2}-\frac{\left \vert x\right \vert }{2{\rm c}%
}}+\int_{\frac{t}{2}-\frac{\left \vert x\right \vert }{2{\rm c}}}^{t-%
\frac{\left \vert x\right \vert }{2{\rm c}}}+\int_{t-\frac{\left \vert
x\right \vert }{2{\rm c}}}^{t}\right) \int_{\mathbb{R}^{3}}\left( \cdots
\right) dyd\tau =:I_{1}+I_{2}+I_{3}\hbox{.}
\end{eqnarray*}

For $I_{1}$, we decompose $\mathbb{R}^{3}$ into two parts%
\begin{equation*}
I_{1}=\int_{0}^{\frac{t}{2}-\frac{\left \vert x\right \vert }{2{\rm c}}%
}\left( \int_{\left \vert y\right \vert \leq \frac{{\rm c}\left( t-\tau
\right) -\left \vert x\right \vert }{2}}+\int_{\left \vert y\right \vert >%
\frac{{\rm c}\left( t-\tau \right) -\left \vert x\right \vert }{2}%
}\right) =:I_{11}+I_{12}\hbox{.}
\end{equation*}%
If $0\leq \tau \leq \frac{t}{2}-\frac{\left \vert x\right \vert }{2{\rm c}%
}$ and $\left \vert y\right \vert \leq \frac{{\rm c}\left( t-\tau \right)
-\left \vert x\right \vert }{2}$, then%
\begin{equation*}
{\rm c}\left( t-\tau \right) -\left \vert x-y\right \vert \geq {\rm c}%
\left( t-\tau \right) -\left \vert x\right \vert -\left \vert y\right \vert
\geq \frac{{\rm c}t-\left \vert x\right \vert }{4}\hbox{,}
\end{equation*}%
so that
\begin{eqnarray*}
I_{11} &\lesssim &\left( 1+t\right) ^{-1}\bigg( 1+%
\frac{\left( {\rm c}t-\left \vert x\right \vert \right) ^{2}}{1+t}\bigg)
^{-1}\\
&&\times\int_{0}^{\frac{t}{2}}\int_{\mathbb{R}^3}(1+t-\tau)^{-3}\bigg( 1+%
\frac{\left( \left \vert x-y\right \vert -{\rm c}\left( t-\tau \right)
\right) ^{2}}{1+t-\tau }\bigg) ^{-1}\left( 1+\tau \right) ^{-1}\left( 1+%
\frac{\left \vert y\right \vert ^{2}}{ 1+\tau }\right) ^{-1}dyd\tau \\
&\lesssim &\left( 1+t\right) ^{-1}\bigg( 1+%
\frac{\left( {\rm c}t-\left \vert x\right \vert \right) ^{2}}{1+t}\bigg)
^{-1}\int_0^{\frac{t}{2}}(1+t-\tau)^{-3}(1+t-\tau)^{-\frac{5}{4}}(1+\tau)^{-1}(1+\tau)^{\frac{3}{4}}d\tau\\
&\lesssim& \left( 1+t\right) ^{-1}\bigg( 1+%
\frac{\left( {\rm c}t-\left \vert x\right \vert \right) ^{2}}{1+t}\bigg)
^{-1}(1+t)^{-1}\lesssim \left( 1+t\right) ^{-2}\bigg( 1+%
\frac{\left( {\rm c}t-\left \vert x\right \vert \right) ^{2}}{1+t}\bigg)^{-1}\hbox{.}
\end{eqnarray*}%
\ If $0\leq \tau \leq \frac{t}{2}-\frac{\left \vert x\right \vert }{2{\rm c}}$ and $\left \vert y\right \vert >\frac{{\rm c}\left( t-\tau \right)
-\left \vert x\right \vert }{2}$, then%
\begin{equation*}
\left \vert y\right \vert >\frac{{\rm c}\left( t-\tau \right) -\left
\vert x\right \vert }{2}\geq \frac{{\rm c}t-\left \vert x\right \vert }{4}%
\hbox{,}
\end{equation*}%
and thus%
\begin{eqnarray*}
I_{12} &\lesssim &\left( 1+t\right) ^{-1}\bigg( 1+%
\frac{\left( {\rm c}t-\left \vert x\right \vert \right) ^{2}}{1+t}\bigg)
^{-1}\\
&&\times\int_{0}^{\frac{t}{2}-\frac{|x|}{c}}\int_{\mathbb{R}^3}(1+t-\tau)^{-4}\bigg( 1+%
\frac{\left( \left \vert x-y\right \vert -{\rm c}\left( t-\tau \right)
\right) ^{2}}{1+t-\tau }\bigg) ^{-2}dyd\tau \\
&\lesssim& \left( 1+t\right) ^{-1}\bigg( 1+%
\frac{\left( {\rm c}t-\left \vert x\right \vert \right) ^{2}}{1+t}\bigg)
^{-1}(1+t)^{-\frac{1}{2}}\lesssim \left( 1+t\right) ^{-\frac{3}{2}}\bigg( 1+%
\frac{\left( {\rm c}t-\left \vert x\right \vert \right) ^{2}}{1+t}\bigg)^{-1}\hbox{.}
\end{eqnarray*}%
Therefore,%
\begin{equation*}
I_{1}\lesssim \left( 1+t\right) ^{-\frac{5}{2}}\bigg( 1+\frac{\left( {\rm c}%
t-\left \vert x\right \vert \right) ^{2}}{1+t}\bigg) ^{-1}\hbox{.}
\end{equation*}

For $I_{2}$, we have to we decompose $\mathbb{R}^{3}$ into two parts%
\begin{equation*}
I_{2}=\int_{\frac{t}{2}-\frac{\left \vert x\right \vert }{2%
{\rm c}}}^{t-\frac{\left \vert x\right \vert }{2{\rm c}}%
}\left( \int_{\left \vert y\right \vert \leq {\rm c}\tau}+\int_{\left \vert y\right \vert >%
{\rm c}\tau}\right) =:I_{21}+I_{22}\hbox{.}
\end{equation*}%
We use the spherical coordinates to obtain
\begin{eqnarray*}
I_{21} &\lesssim &\int_{\frac{t}{2}-\frac{\left \vert x\right \vert }{2%
{\rm c}}}^{t-\frac{\left \vert x\right \vert }{2{\rm c}}%
}\int_{0}^{c\tau}\int_{0}^{\pi }\left( 1+t-\tau \right) ^{-4}\bigg( 1+%
\frac{\big( \sqrt{\left \vert x\right \vert ^{2}+r^{2}-2r\left \vert
x\right \vert \cos \theta }-{\rm c}\left( t-\tau \right) \big) ^{2}}{%
1+t-\tau }\bigg) ^{-2} \\
&&\cdot \left( 1+\tau \right) ^{-1}\bigg( 1+\frac{\left \vert r\right \vert
^{2}}{ 1+\tau }\bigg) ^{-1}r^{2}\sin \theta d\theta drd\tau \\
&\lesssim &\int_{\frac{t}{2}-\frac{\left \vert x\right \vert }{2{\rm c}}%
}^{t-\frac{\left \vert x\right \vert }{2{\rm c}}}\int_{0}^{c\tau}\int_{\left \vert \left \vert x\right \vert -r\right \vert }^{\left \vert
x\right \vert +r}\left( 1+t-\tau \right) ^{-4}\bigg( 1+\frac{\left( z-%
{\rm c}\left( t-\tau \right) \right) ^{2}}{1+t-\tau }\bigg) ^{-2}z(
1+\tau ) ^{-1} \\
&&\cdot \left( 1+\frac{r^{2}}{ 1+\tau }\right) ^{-1}%
\frac{r}{\left \vert x\right \vert }dzdrd\tau \\
&\lesssim &\left \vert x\right \vert ^{-1}\int_{\frac{t}{2}-\frac{\left
\vert x\right \vert }{2{\rm c}}}^{t-\frac{\left \vert x\right \vert }{2%
{\rm c}}}\int_{0}^{c\tau}\left( 1+t-\tau \right) ^{-4+\frac{3}{2}%
}\left( 1+\tau \right) ^{-1}\cdot r\left( 1+\frac{r^{2}}{ 1+\tau }\right) ^{-1}drd\tau
\\
&\lesssim &\left \vert x\right \vert ^{-1}\int_{\frac{t}{2}-\frac{\left
\vert x\right \vert }{2{\rm c}}}^{t-\frac{\left \vert x\right \vert }{2%
{\rm c}}}\left( 1+t-\tau \right) ^{-\frac{5}{2}}\ln( 1+{\rm c}\tau)d\tau \\
&\lesssim &\left \vert x\right \vert ^{-1}\big(1+\frac{|x|}{2c}\big)^{-1}\int_{\frac{t}{2}-\frac{\left
\vert x\right \vert }{2{\rm c}}}^{t-\frac{\left \vert x\right \vert }{2%
{\rm c}}}\left( 1+t-\tau \right) ^{-\frac{3}{2}}\ln( 1+{\rm c}\tau)d\tau \\
&\lesssim &\left( 1+\left \vert x\right \vert \right)
^{-2}\lesssim \left( 1+t\right) ^{-1}\bigg( 1+\frac{\left \vert x\right
\vert ^{2}}{1+t}\bigg) ^{-1}\hbox{,}
\end{eqnarray*}%
where we used $z=\sqrt{\left \vert x\right \vert ^{2}+r^{2}-2r\left \vert
x\right \vert \cos \theta }$ and $\sin \theta d\theta =\frac{z}{r\left \vert
x\right \vert }dz$. 

For $I_{22}$, since $|y|>{\rm c}\tau$, one has
\begin{align*}
|y|>{\rm c}\tau>{\rm c}\big(\frac{t}{2}-\frac{|x|}{2{\rm c}}\big)=\frac{{\rm c}t-|x|}{2}\geq\frac{|x|}{2},
\end{align*}
which gives that
\begin{eqnarray*}
I_{22}&\lesssim& (1+t)^{-1}\big(1+\frac{|x|^2}{1+t}\big)^{-1}\int_{\frac{t}{2}-\frac{\left \vert x\right \vert }{2%
{\rm c}}}^{t-\frac{\left \vert x\right \vert }{2{\rm c}}%
}(1+t-\tau)^{-4}\left( 1+%
\frac{\left( \left \vert x-y\right \vert -{\rm c}\left( t-\tau \right)
\right) ^{2}}{1+t-\tau }\right) ^{-2}dyd\tau\nonumber\\
&\lesssim& (1+t)^{-1}\big(1+\frac{|x|^2}{1+t}\big)^{-1}.
\end{eqnarray*}

For $I_{3}$, we decompose $\mathbb{R}^{3}$ into two parts%
\begin{equation*}
I_{3}=\int_{t-\frac{\left \vert x\right \vert }{2{\rm c}}}^{t}\left(
\int_{\left \vert y\right \vert \leq \frac{\left \vert x\right \vert -%
{\rm c}\left( t-\tau \right) }{2}}+\int_{\left \vert y\right \vert >\frac{%
\left \vert x\right \vert -{\rm c}\left( t-\tau \right) }{2}}\right)
=:I_{31}+I_{32}\hbox{.}
\end{equation*}%
If $t-\frac{\left \vert x\right \vert }{2{\rm c}}\leq \tau \leq t$, $%
\left \vert y\right \vert \leq \frac{\left \vert x\right \vert -{\rm c}%
\left( t-\tau \right) }{2}$, then
\begin{equation*}
\left \vert x-y\right \vert -{\rm c}\left( t-\tau \right) \geq \left
\vert x\right \vert -\left \vert y\right \vert -{\rm c}\left( t-\tau
\right) \geq \frac{\left \vert x\right \vert -{\rm c}\left( t-\tau
\right) }{2}\geq \frac{\left \vert x\right \vert }{4}\hbox{.}
\end{equation*}%
If $t-\frac{\left \vert x\right \vert }{2{\rm c}}\leq \tau \leq t$, $%
\left \vert y\right \vert >\frac{\left \vert x\right \vert -{\rm c}\left(
t-\tau \right) }{2}$, then%
\begin{equation*}
\left \vert y\right \vert >\frac{\left \vert x\right \vert -{\rm c}\left(
t-\tau \right) }{2}\geq \frac{\left \vert x\right \vert }{4}\hbox{.}
\end{equation*}%
Hence, since $\frac{{\rm c}t}{2}\geq\left \vert x\right \vert \geq \sqrt{1+t}$, it holds that 
\begin{eqnarray*}
I_{31} &\lesssim &(1+t)^{-1}\left( 1+\frac{|x|^2}{1+t}\right) ^{-1}\\
&&\cdot\int_{t-\frac{\left \vert x\right \vert }{2{\rm c}}}^{t}\int_{\left
\vert y\right \vert \leq \frac{\left \vert x\right \vert \!-\!{\rm c}\left(
t\!-\!\tau \right) }{2}}\left( 1\!+\!t\!-\!\tau \right) ^{-3}\left( 1+\frac{\left( \left \vert x\!-\!y\right \vert \!-\!{\rm c}%
\left( t-\tau \right) \right) ^{2}}{1+t-\tau }\right) ^{\!-1}\!\!\!(1\!+\!\tau)^{-1}\big(1+\frac{|y|^2}{1+\tau}\big)^{-1}\!dyd\tau \\
&\lesssim &(1+t)^{-1}\left( 1+\frac{|x|^2}{1+t}\right) ^{-1}\int_{t-\frac{\left \vert x\right \vert }{2{\rm c}}}^{t}(1+t-\tau)^{-\frac{7}{4}}(1+\tau)^{-\frac{1}{4}}d\tau\\
&\lesssim& \left( 1+t\right) ^{-\frac{5}{4}}\bigg( 1+\frac{\left \vert
x\right \vert ^{2}}{1+t}\bigg) ^{-1}\hbox{,}
\end{eqnarray*}%
 and%
\begin{eqnarray*}
I_{32} &\lesssim &\left( 1+t\right) ^{-1}\left( 1+\frac{\left \vert x\right
\vert ^{2}}{ 1+t }\right) ^{-1}\int_{t-\frac{\left
\vert x\right \vert }{2{\rm c}}}^{t}\int_{\mathbb{R}^{3}}\left( 1+t-\tau
\right) ^{-4}\left( 1+\frac{\left( \left \vert x-y\right \vert -{\rm c}%
\left( t-\tau \right) \right) ^{2}}{1+t-\tau }\right) ^{-2}dyd\tau \\
&\lesssim &\left( 1+t\right) ^{-1}\bigg( 1+\frac{\left \vert x\right \vert
^{2}}{1+t}\bigg) ^{-1}\hbox{.}
\end{eqnarray*}%
Combining all the estimates yields
\begin{equation*}
K_4\lesssim \left( 1+t\right) ^{-2}\bigg( 1+\frac{\left \vert x\right \vert
^{2}}{1+t }\bigg) ^{-\frac{3}{2}}+\left( 1+t\right)
^{-5/2}\bigg( 1+\frac{\left( {\rm c}t-\left \vert x\right \vert \right)
^{2}}{1+t}\bigg) ^{-1}\hbox{.}
\end{equation*}
\newline
\textbf{Case 5:} $\left( x,t\right) \in D_{5}$. In this region $%
\frac{{\rm c}t}{2}\leq \left \vert x\right \vert \leq {\rm c}t-\sqrt{1+t}$. We split $K_4$ into three parts%
\begin{eqnarray*}
K_4&=&\int_{0}^{t}\int_{\mathbb{R}^{3}}\left( 1+t-\tau \right) ^{-4}\bigg( 1+%
\frac{\left( \left \vert x-y\right \vert -{\rm c}\left( t-\tau \right)
\right) ^{2}}{1+t-\tau }\bigg) ^{-2}\left( 1+\tau \right) ^{-2 }\bigg( 1+%
\frac{\left \vert y\right \vert ^{2}}{ 1+\tau}\bigg) ^{-1}
dyd\tau \\
&=&\left( \int_{0}^{\frac{t}{2}-\frac{\left \vert x\right \vert }{2{\rm c}%
}}+\int_{\frac{t}{2}-\frac{\left \vert x\right \vert }{2{\rm c}}}^{\frac{1%
}{2}\left( t+\frac{3}{2}\left( t-\frac{\left \vert x\right \vert }{{\rm c}%
}\right) \right) }+\int_{\frac{1}{2}\left( t+\frac{3}{2}\left( t-\frac{\left
\vert x\right \vert }{{\rm c}}\right) \right) }^{t}\right) \int_{\mathbb{R%
}^{3}}\left( \cdots \right) dyd\tau =:I_{1}+I_{2}+I_{3}\hbox{.}
\end{eqnarray*}

For $I_{1}$, the estimate is the same as the $I_{1}$ of Case 4, so
\begin{equation*}
I_{1}\lesssim \left( 1+t\right) ^{-\frac{3}{2}}\bigg( 1+\frac{\left( {\rm c}%
t-\left \vert x\right \vert \right) ^{2}}{1+t}\bigg) ^{-1}\hbox{.}
\end{equation*}

For $I_{2}$, we have to we decompose $\mathbb{R}^{3}$ into two parts%
\begin{equation*}
I_{2}=\int_{\frac{t}{2}-\frac{\left \vert x\right \vert }{2{\rm c}}}^{\frac{1%
}{2}\left( t+\frac{3}{2}\left( t-\frac{\left \vert x\right \vert }{{\rm c}%
}\right) \right) }\left( \int_{\left \vert y\right \vert \leq {\rm c}(t-\tau)}+\int_{\left \vert y\right \vert >%
{\rm c}(t-\tau)}\right) =:I_{21}+I_{22}\hbox{.}
\end{equation*}%
In terms of $I_{21}$, similar to the $I_{21}$ of Case 4, we use the spherical coordinates to obtain%
\begin{eqnarray*}
I_{21} &\lesssim &\int_{\frac{t}{2}-\frac{\left \vert x\right \vert }{2{\rm c}}}^{\frac{1%
}{2}\left( t+\frac{3}{2}\left( t-\frac{\left \vert x\right \vert }{{\rm c}%
}\right) \right) }\int_{0}^{{\rm c}(t-\tau)}\int_{0}^{\pi }\left( 1+t-\tau \right) ^{-4}\Bigg( 1+%
\frac{\big( \sqrt{\left \vert x\right \vert ^{2}+r^{2}-2r\left \vert
x\right \vert \cos \theta }-{\rm c}\left( t-\tau \right) \big)^{2}}{%
1+t-\tau }\Bigg)^{-2} \\
&&\cdot \left( 1+\tau \right) ^{-1}\bigg( 1+\frac{\left \vert r\right \vert
^{2}}{ 1+\tau }\bigg) ^{-1}r^{2}\sin \theta d\theta drd\tau \\
&\lesssim &\int_{\frac{t}{2}-\frac{\left \vert x\right \vert }{2{\rm c}}}^{\frac{1%
}{2}\left( t+\frac{3}{2}\left( t-\frac{\left \vert x\right \vert }{{\rm c}%
}\right) \right) }\int_{0}^{{\rm c}(t-\tau)}\int_{\left \vert \left \vert x\right \vert -r\right \vert }^{\left \vert
x\right \vert +r}\left( 1+t-\tau \right) ^{-4}\left( 1+\frac{\left( z-%
{\rm c}\left( t-\tau \right) \right) ^{2}}{1+t-\tau }\right) ^{-2}z\left(
1+\tau \right) ^{-1} \\
&&\cdot \bigg( 1+\frac{r^{2}}{ 1+\tau }\bigg) ^{-1}%
\frac{r}{\left \vert x\right \vert }dzdrd\tau \\
&\lesssim &\left \vert x\right \vert ^{-1}\int_{\frac{t}{2}-\frac{\left
\vert x\right \vert }{2{\rm c}}}^{t-\frac{\left \vert x\right \vert }{2%
{\rm c}}}\int_{0}^{{\rm c}(t-\tau)}\left( 1+t-\tau \right) ^{-4+\frac{3}{2}%
}\left( 1+\tau \right) ^{-1}\cdot r\left( 1+\frac{r^{2}}{ 1+\tau }\right) ^{-1}drd\tau
\\
&\lesssim &\left \vert x\right \vert ^{-1}\int_{\frac{t}{2}-\frac{\left \vert x\right \vert }{2{\rm c}}}^{\frac{1%
}{2}\left( t+\frac{3}{2}\left( t-\frac{\left \vert x\right \vert }{{\rm c}%
}\right) \right) }\left( 1+t-\tau \right) ^{-\frac{5}{2}}\ln( 1+{\rm c}(t-\tau))d\tau \\
&\lesssim &\left \vert x\right \vert ^{-1}\big(1+\frac{|x|}{2{\rm c}}\big)^{-1}\int_{\frac{t}{2}-\frac{\left \vert x\right \vert }{2{\rm c}}}^{\frac{1%
}{2}\left( t+\frac{3}{2}\left( t-\frac{\left \vert x\right \vert }{{\rm c}%
}\right) \right) }\left( 1+t-\tau \right) ^{-\frac{3}{2}}\ln( 1+{\rm c}(t-\tau))d\tau \\
&\lesssim &\left( 1+\left \vert x\right \vert \right)
^{-2}\lesssim \left( 1+t\right) ^{-1}\bigg( 1+\frac{\left \vert x\right
\vert ^{2}}{1+t}\bigg) ^{-1}\hbox{.}
\end{eqnarray*}%
\vspace {3mm}%
For $I_{22}$, since $|y|>{\rm c}(t-\tau)$, one has
\begin{align*}
|y|>{\rm c}(t-\tau)\geq {\rm c}\bigg(\frac{t}{2}-\frac{3}{4}\bigg(t-\frac{|x|}{{\rm c}}\bigg)\bigg)\geq {\rm c}\bigg(\frac{t}{2}-\frac{3}{4}\bigg(t-\frac{{\rm c}t/2}{2}\bigg)\bigg)=\frac{{\rm c}t}{8}>\frac{|x|}{8},
\end{align*}
which gives that
\begin{eqnarray*}
I_{22}&\lesssim& (1+t)^{-1}\bigg(1+\frac{|x|^2}{1+t}\bigg)^{-1}\int_{\frac{t}{2}-\frac{\left \vert x\right \vert }{2{\rm c}}}^{\frac{1%
}{2}\left( t+\frac{3}{2}\left( t-\frac{\left \vert x\right \vert }{{\rm c}%
}\right) \right) }(1+t-\tau)^{-4}\bigg( 1+%
\frac{\left( \left \vert x-y\right \vert -{\rm c}\left( t-\tau \right)
\right) ^{2}}{1+t-\tau }\bigg) ^{-2}dyd\tau\nonumber\\
&\lesssim& (1+t)^{-1}\bigg(1+\frac{|x|^2}{1+t}\bigg)^{-1}.
\end{eqnarray*}

For $I_{3}$, we decompose $\mathbb{R}^{3}$ into two parts%
\begin{equation*}
I_{3}=\int_{\frac{1}{2}\big( t+\frac{3}{2}\big( t-\frac{\left \vert
x\right \vert }{{\rm c}}\big) \big) }^{t}\left( \int_{\left \vert
y\right \vert \leq \frac{\left \vert x\right \vert -{\rm c}\left( t-\tau
\right) }{2}}+\int_{\left \vert y\right \vert >\frac{\left \vert x\right
\vert -{\rm c}\left( t-\tau \right) }{2}}\right) \left( \cdots \right)
dyd\tau =:I_{31}+I_{32}\hbox{.}
\end{equation*}%
If $\frac{1}{2}\big( t+\frac{3}{2}\big( t-\frac{\left \vert x\right \vert
}{{\rm c}}\big) \big) \leq \tau \leq t$ and $\left \vert y\right \vert
\leq \frac{\left \vert x\right \vert -{\rm c}\left( t-\tau \right) }{2}$,
then
\begin{equation*}
\left \vert x-y\right \vert -{\rm c}\left( t-\tau \right) \geq \frac{%
\left \vert x\right \vert -{\rm c}\left( t-\tau \right) }{2}\geq \frac{%
{\rm c}t+\left \vert x\right \vert }{8}= \frac{{\rm c}t-\left \vert
x\right \vert }{8}+\frac{|x|}{4}\hbox{.}
\end{equation*}%
If $\frac{1}{2}\big( t+\frac{3}{2}\big( t-\frac{\left \vert x\right \vert
}{{\rm c}}\big) \big) \leq \tau \leq t$ and $\left \vert y\right \vert >%
\frac{\left \vert x\right \vert -{\rm c}\left( t-\tau \right) }{2}$, then
\begin{equation*}
\left \vert y\right \vert >\frac{\left \vert x\right \vert -{\rm c}\left(
t-\tau \right) }{2}\geq \frac{{\rm c}t-\left \vert x\right \vert }{8}+%
\frac{|x|}{4} \hbox{.}
\end{equation*}%
Hence,
\begin{eqnarray*}
I_{31} \!\!\!\!\!\!\!\!&&\lesssim(1+t)^{-1}\left( 1+\frac{|x|^2}{1+t}\right) ^{-1}\int_{\frac{1}{2}\left( t+\frac{3}{2}\left( t-\frac{\left \vert
x\right \vert }{{\rm c}}\right) \right) }^{t}\int_{\left
\vert y\right \vert \leq \frac{\left \vert x\right \vert -{\rm c}\left(
t-\tau \right) }{2}}\left( 1+t-\tau \right) ^{-3}\\
&&\cdot\bigg( 1+\frac{\left( \left \vert x\!-\!y\right \vert \!-\!{\rm c}%
\left( t-\tau \right) \right) ^{2}}{1\!+\!t\!-\!\tau }\bigg) ^{\!\!-1}\!\!\!\!(1\!+\!\tau)^{-1}\big(1+\frac{|y|^2}{1+\tau}\big)^{-1}\!dyd\tau \\
&&\lesssim(1+t)^{-1}\bigg( 1+\frac{|x|^2}{1+t}\bigg) ^{-1}\int_{\frac{1}{2}\left( t+\frac{3}{2}\left( t-\frac{\left \vert
x\right \vert }{{\rm c}}\right) \right) }^{t}(1+t-\tau)^{-\frac{7}{4}}(1+\tau)^{-\frac{1}{4}}d\tau\\
&&\lesssim \left( 1+t\right) ^{-\frac{5}{4}}\bigg( 1+\frac{\left \vert
x\right \vert ^{2}}{1+t}\bigg) ^{-1}\hbox{.}
\end{eqnarray*}%
Similarly, $I_{32}$ can be estimated as follows:
\begin{eqnarray*}
I_{32} &\lesssim &(1+t)^{-1}\bigg( 1+\frac{|x|^2}{1+t}\bigg) ^{-1}\\
&&\cdot\int_{\frac{1}{2}\left( t+\frac{3}{2}\left( t-\frac{\left \vert
x\right \vert }{{\rm c}}\right) \right) }^{t}\int_{\left
\vert y\right \vert \geq \frac{\left \vert x\right \vert \!-\!{\rm c}\left(
t\!-\!\tau \right) }{2}}\left( 1\!+\!t\!-\!\tau \right) ^{-4}\left( 1+\frac{\left( \left \vert x\!-\!y\right \vert \!-\!{\rm c}%
\left( t-\tau \right) \right) ^{2}}{1+t-\tau }\right) ^{\!-2}dyd\tau \\
&\lesssim &(1+t)^{-1}\bigg( 1+\frac{|x|^2}{1+t}\bigg) ^{-1}\hbox{.}
\end{eqnarray*}%
By summing up the estimates in five domains $D_i$ with $i=1,2,3,4,5$, one has
\begin{equation*}
K_4\leq \left( 1+t\right) ^{-1}\bigg( 1+\frac{\left \vert x\right \vert ^{2}}{%
 1+t }\bigg) ^{-\frac{3}{2}}+\left( 1+t\right) ^{-1}\bigg( 1+%
\frac{\left( {\rm c}t-\left \vert x\right \vert \right) ^{2}}{1+t}\bigg)
^{-1}\hbox{,}
\end{equation*}
and has completed the proof of (\ref{6.2}) about Riesz wave-$I$ convolved with Huygens wave. \ \ \ \ \ \ \ \ \  \textsquare

\vspace{5mm}
Finally, we shall establish two nonlinear convolution estimates on the Huygens wave convolved with nonlinear Huygens wave or nonlinear diffusion wave for the non-conservation of the system (\ref{1.1}), which is completely different from the previous results due to the slower decay rate of Huygens wave in Green's function: $(1+t)^{-2}\Big(1+\frac{(|x|-{\rm c}t)^2}{1+t}\Big)^{-N}$ compared with $K_3$ in Lemma \ref{A.4}. 

\begin{lemma}\label{A.6} There exists a constant $C>0$ such that

\noindent 
{\rm (``Huygens\ wave"\ convolved\ with\ ``diffusion\ wave")}
\begin{align}
K_6=&\int_0^t\!\int_{\mathbb{R}^3}(1+t-\tau)^{-2}\Big(1+\frac{(|x-y|-{\rm c}(t-\tau))^2}{1+t-\tau}\Big)^{-2}(1+\tau)^{-3}\Big(1+\frac{|y|^2}{1+\tau}\Big)^{-3}\!\!dyd\tau\nonumber\\
\leq &\ C(1+t)^{-\frac{3}{2}}\Big(1+\frac{|x|^2}{1+t}\Big)^{-\frac{3}{2}}+C(1+t)^{-2}\Big(1+\frac{(|x|-{\rm c}t)^2}{1+t}\Big)^{-1},\label{6.3}
\end{align}
\noindent
{\rm (``Huygens\ wave"\ convolved\ with\ ``Huygens\ wave")}
\begin{align}
K_7=&\int_{0}^{t}\!\int_{\mathbb{R}^3}(1+t-\tau)^{-2}\Big(1+\frac{(|x-y|-{\rm c}(t-\tau))^2}{1+t-\tau}\Big)^{-N}(1+\tau)^{-4}\Big(1+\frac{(|y|-{\rm c}\tau)^2}{1+\tau}\Big)^{-2}dyd\tau\nonumber\\
\leq &\ C(1+t)^{-\frac{3}{2}}\Big(1+\frac{|x|^2}{1+t}\Big)^{-\frac{3}{2}}+C(1+t)^{-2}\Big(1+\frac{(|x|-{\rm c}t)^2}{1+t}\Big)^{-1},\label{6.5}
\end{align}
where the constant $N>0$ can be arbitrarily large.
\end{lemma}
We only focus on the proof of $K_7$ since the proof of $K_6$ is relatively easier.

\noindent\textit{\textbf{Proof of $K_7$}}. (Huygens wave convolved with Huygens wave for {\bf{non-conservative system}})

\textbf{Case 1:} $\left( x,t\right) \in D_{1}\cup D_{2}$. Direct computation
gives
\begin{eqnarray*}
K_7 &\lesssim &\int_{0}^{\frac{t}{2}}\int_{\mathbb{R}^{3}}( 1+t)
^{-2}\left( 1+\tau \right) ^{-4}\bigg( 1+\frac{\left( \left \vert y\right
\vert -{\rm c}\tau \right) ^{2}}{1+\tau }\bigg) ^{-2}dyd\tau \\
&&+\int_{\frac{t}{2}}^{t}\int_{\mathbb{R}^{3}}\left( 1+t-\tau \right)
^{-2}\bigg(1+\frac{(|x-y|-{\rm c}(t-\tau))^2}{1+t-\tau}\bigg)^{-N}\left( 1+t\right)
^{-4}dyd\tau \\
&\lesssim &\left( 1+t\right)^{-2}\int_{0}^{\frac{t}{2}}\left( 1+\tau
\right)^{-4}\left( 1+\tau \right) ^{\frac{5}{2}}d\tau +\left( 1+t\right)^{-4}\int_{\frac{t}{2}}^{t}\left( 1+t-\tau \right) ^{-2}\left( 1+t-\tau
\right) ^{\frac{5}{2}}d\tau \\
&\lesssim &(1+t)^{-2}\hbox{,}
\end{eqnarray*}
which implies that
\begin{eqnarray*}
K_7 \lesssim ( 1+t)^{-2}\bigg(1+\frac{|x|^2}{1+t}\bigg)^{-N}\hbox{,}\ \ \ \ \ \ {\rm in}\ D_1,\\[2mm]
K_7 \lesssim ( 1+t)^{-2}\bigg(1+\frac{(|x|-{\rm c}t)^2}{1+t}\bigg)^{-N}\hbox{,}\ \ \ \ \ {\rm in}\ D_2.
\end{eqnarray*}

\noindent\textbf{Case 2:} $\left( x,t\right) \in D_{3}$. We split the integral $K_7$
into four parts%
\begin{eqnarray*}
J \!&\!=\!&\!\int_{0}^{\frac{t}{2}}\bigg( \int_{\left \vert y\right \vert -{\rm c}%
\tau \leq \frac{\left \vert x\right \vert -{\rm c}t}{2}}+\int_{\left
\vert y\right \vert -{\rm c}\tau >\frac{\left \vert x\right \vert -%
{\rm c}t}{2}}\bigg) \left( \cdots \right) dyd\tau\\
&& +\int_{\frac{t}{2}%
}^{t}\bigg( \int_{\left \vert y\right \vert -{\rm c}\tau \leq \frac{\left
\vert x\right \vert -{\rm c}t}{2}}+\int_{\left \vert y\right \vert -%
{\rm c}\tau >\frac{\left \vert x\right \vert -{\rm c}t}{2}}\bigg)
\left( \cdots \right) dyd\tau\\[2mm]
\!&\!=\!&\!J_{11}+J_{12}+J_{21}+J_{22}\hbox{.}
\end{eqnarray*}

Note that $\left \vert y\right \vert -{\rm c}\tau \leq \frac{%
\left
\vert x\right \vert -{\rm c}t}{2}$ implies $\left \vert x-y\right \vert -{\rm c}\left( t-\tau \right) \geq \left
\vert x\right \vert -\left \vert y\right \vert -{\rm c}\left( t-\tau
\right) \geq \frac{\left \vert x\right \vert -{\rm c}t}{2}$. 
Hence,
\begin{eqnarray*}
J_{11} &\lesssim &\left(
1+t\right) ^{-2}\Big(1+\frac{(|x|-{\rm c}t)^2}{1+t}\Big)^{-N}\int_{0}^{\frac{t}{2}}\int_{\left \vert y\right \vert -%
{\rm c}\tau \leq \frac{\left \vert x\right \vert -{\rm c}t}{2}}\left( 1+\tau \right) ^{-4}\bigg( 1+%
\frac{\left( \left \vert y\right \vert -{\rm c}\tau \right) ^{2}}{1+\tau }%
\bigg) ^{-2}dyd\tau \\
&\lesssim &\left(
1+t\right) ^{-2}\Big(1+\frac{(|x|-{\rm c}t)^2}{1+t}\Big)^{-N}\int_{0}^{\frac{t}{2}}\left( 1+\tau \right) ^{-4}\left( 1+\tau \right) ^{\frac{5}{2}}d\tau\\
&\lesssim& \left( 1+t\right) ^{-2}\Big(1+\frac{(|x|-{\rm c}t)^2}{1+t}\Big)^{-N}\hbox{.}
\end{eqnarray*}

For $J_{12}$, by using Young inequality, we have
\begin{eqnarray*}
J_{12} &\lesssim &\int_{0}^{\frac{t}{2}}\int_{\left \vert y\right \vert -%
{\rm c}\tau >\frac{\left \vert x\right \vert -{\rm c}t}{2}}\left(
1+t-\tau \right)^{-2}\Big(1+\frac{(|x-y|-{\rm c}(t-\tau))^2}{1+t-\tau}\Big)^{-N}\\
&&\ \ \ \ \ \ \ \cdot \left( 1+\tau \right) ^{-4}\bigg( 1+\frac{\left( \left \vert x\right \vert -%
{\rm c}t \right) ^{2}}{1+\tau }\bigg) ^{-1}\bigg( 1+\frac{\left( \left \vert y\right \vert -%
{\rm c}\tau \right) ^{2}}{1+\tau }\bigg) ^{-1}dyd\tau \\
&\lesssim &(1+t)^{-2}\Big( 1+\frac{\left( \left \vert x\right \vert -%
{\rm c}t \right) ^{2}}{1+t}\Big) ^{-1}\int_{0}^{\frac{t}{2}}\int_{\left \vert y\right \vert -%
{\rm c}\tau >\frac{\left \vert x\right \vert -{\rm c}t}{2}}\left(
1+t-\tau \right) ^{-1}\Big(1+\frac{(|x-y|-{\rm c}(t-\tau))^2}{1+t-\tau}\Big)^{-N}\\
&&\ \ \ \ \ \ \ \ \ \ \ \ \ \ \ \ \ \ \ \ \ \ \ \ \ \ \ \cdot \left( 1+\tau \right) ^{-3}\bigg( 1+\frac{\left( \left \vert y\right \vert -%
{\rm c}\tau \right)^{2}}{1+\tau }\bigg) ^{-1}dyd\tau \\
&\lesssim &(1+t)^{-2}\Big( 1+\frac{\left( \left \vert x\right \vert -%
{\rm c}t \right) ^{2}}{1+t}\Big) ^{-1}\int_{0}^{\frac{t}{2}}(1+t-\tau)^{-1}(1+t-\tau)(1+\tau)^{-3}(1+\tau)^{\frac{3}{2}}d\tau \\
&\lesssim &(1+t)^{-2}\Big( 1+\frac{\left( \left \vert x\right \vert -%
{\rm c}t \right) ^{2}}{1+t}\Big) ^{-1}.
\end{eqnarray*}%

For $J_{21}$ and $J_{22}$, it follows that
\begin{eqnarray*}
J_{21} &\lesssim &\int_{\frac{t}{2}}^{t}\int_{\left \vert y\right \vert -%
{\rm c}\tau \leq \frac{\left \vert x\right \vert -{\rm c}t}{2}}\left(
1+t-\tau \right) ^{-2}\Big(1+\frac{(|x|-{\rm c}t)^2}{1+t-\tau}\Big)^{-N/2}\Big(1+\frac{(|x-y|-{\rm c}(t-\tau))^2}{1+t-\tau}\Big)^{-N/2}\\
&& \ \ \ \ \ \ \ \ \cdot\left( 1+t\right) ^{-4}\bigg( 1+\frac{\left( \left \vert
y\right \vert -{\rm c}\tau \right) ^{2}}{1+\tau }\bigg) ^{-2}dyd\tau \\
&\lesssim &\left( 1+t\right) ^{-4}\Big(1+\frac{(|x|-{\rm c}t)^2}{1+t}\Big)^{-N/2}\int_{\frac{t}{2}%
}^{t}\left( 1+t-\tau \right) ^{-2}\left( 1+t-\tau \right) ^{\frac{5}{2}%
}d\tau \\
&\lesssim &\left( 1+t\right) ^{-\frac{5}{2}}\Big(1+\frac{(|x|-{\rm c}t)^2}{1+t}\Big)^{-N/2}\hbox{,}
\end{eqnarray*}%
\begin{eqnarray*}
J_{22} &\lesssim &\int_{\frac{t}{2}}^{t}\int_{\left \vert y\right \vert -%
{\rm c}\tau >\frac{\left \vert x\right \vert -{\rm c}t}{2}}\left(
1+t-\tau \right) ^{-2}\Big(1+\frac{(|x-y|-{\rm c}(t-\tau))^2}{1+t-\tau}\Big)^{-N}\\
&&\cdot\left( 1+t\right) ^{-4}\bigg( 1+\frac{\left( \left \vert y\right \vert -%
{\rm c}\tau \right) ^{2}}{1+\tau }\bigg) ^{-2}dyd\tau \\
&\lesssim &\int_{\frac{t}{2}}^{t}\left( 1+t-\tau \right) ^{-2}\left(
1+t-\tau \right) ^{\frac{5}{2}}\left( 1+t\right) ^{-4}\bigg( 1+\frac{\left(
\left \vert x\right \vert -{\rm c}t\right) ^{2}}{1+t}\bigg) ^{-2}d\tau \\
&\lesssim &\left( 1+t\right) ^{-\frac{5}{2}}\bigg( 1+\frac{\left( \left \vert x\right
\vert -{\rm c}t\right) ^{2}}{1+t}\bigg) ^{-2}\hbox{.}
\end{eqnarray*}

Gathering all the estimates, we can conclude that
\begin{equation*}
K_7\lesssim \left( 1+t\right) ^{-2}\bigg( 1+\frac{\left( \left \vert x\right
\vert -{\rm c}t\right) ^{2}}{1+t}\bigg) ^{-1}\hbox{.}
\end{equation*}
\newline
\textbf{Case 3:} $\left( x,t\right) \in D_{4}$. We split the integral into
four parts%
\begin{eqnarray*}
K_7&=&\bigg( \int_{0}^{\frac{t}{4}-\frac{\left \vert x\right \vert }{4\mathbf{%
c}}}+\int_{\frac{t}{4}-\frac{\left \vert x\right \vert }{4{\rm c}}}^{%
\frac{t}{2}}+\int_{\frac{t}{2}}^{\frac{t}{2}+\frac{1}{4}\left( t+\frac{\left
\vert x\right \vert }{{\rm c}}\right) }+\int_{\frac{t}{2}+\frac{1}{4}%
\left( t+\frac{\left \vert x\right \vert }{{\rm c}}\right) }^{t}\bigg)
\int_{\mathbb{R}^{3}}\left( \cdots \right) dyd\tau \\
&=:&J_{1}+J_{2}+J_{3}+J_{4}\hbox{.}
\end{eqnarray*}

For $J_{1}$, we decompose $\mathbb{R}^{3}$ into two parts%
\begin{equation*}
J_{1}=\int_{0}^{\frac{t}{4}-\frac{\left \vert x\right \vert }{4{\rm c}}%
}\left( \int_{\left \vert y\right \vert \leq \frac{{\rm c}t-\left \vert
x\right \vert }{2}}+\int_{\left \vert y\right \vert >\frac{{\rm c}t-\left
\vert x\right \vert }{2}}\right) \int_{\mathbb{R}^{3}}\left( \cdots \right)
dyd\tau =:J_{11}+J_{12}\hbox{.}
\end{equation*}%
If $0\leq \tau \leq \frac{t}{4}-\frac{\left \vert x\right \vert }{4{\rm c}%
}$ and $\left \vert y\right \vert \leq \frac{{\rm c}t-\left \vert
x\right
\vert }{2}$, then
\begin{equation*}
{\rm c}\left( t-\tau \right) -\left \vert x-y\right \vert \geq {\rm c}%
\left( t-\tau \right) -\left \vert x\right \vert -\left \vert y\right \vert
\geq \frac{{\rm c}t-\left \vert x\right \vert }{4}\geq \frac{{\rm c}t}{%
8}\hbox{.}
\end{equation*}%
Hence, it holds that
\begin{eqnarray*}
J_{11} &\lesssim &\int_{0}^{\frac{t}{4}-\frac{\left \vert x\right \vert }{4%
{\rm c}}}\!\!\int_{\left \vert y\right \vert \leq \frac{{\rm c}t-\left
\vert x\right \vert }{2}}\left( 1\!+\!t\!-\!\tau\right) ^{-2}\Big(1+\frac{(|x\!-\!y|\!-\!{\rm c}(t-\tau))^2}{1+t-\tau}\Big)^{\!-N}\left( 1\!+\!\tau \right) ^{-4}\bigg( 1\!+\!\frac{\left( \left \vert y\right \vert -%
{\rm c}\tau \right) ^{2}}{1+\tau }\bigg) ^{\!-2}\!dyd\tau \\
&\lesssim &\left(
1+t\right) ^{-2}\Big(1\!+\!\frac{(|x|-{\rm c}t)^2}{1+t}\Big)^{-N}\!\int_{0}^{\frac{t}{%
4}-\frac{\left \vert x\right \vert }{4{\rm c}}}\!\left( 1+\tau \right)
^{-4}\left( 1+\tau \right) ^{\frac{5}{2}}d\tau \lesssim \left(
1+t\right) ^{-2}\Big(1\!+\!\frac{(|x|-{\rm c}t)^2}{1+t}\Big)^{-N}\hbox{,}
\end{eqnarray*}%
and%
\begin{eqnarray*}
J_{12} &\lesssim &\int_{0}^{\frac{t}{4}-\frac{\left \vert x\right \vert }{4%
{\rm c}}}\!\!\!\int_{\left \vert y\right \vert >\frac{{\rm c}t-\left \vert
x\right \vert }{2}}\left( 1\!+\!t\!-\!\tau \right) ^{-2}\Big(1\!+\!\frac{(|x-y|-{\rm c}(t-\tau))^2}{1+t-\tau}\Big)^{\!-N}\!\left( 1\!+\!\tau \right) ^{-2}\left( 1\!+\!\left(
{\rm c}t\!-\!\left \vert x\right \vert \right)^{2}\right) ^{\!-2}\!dyd\tau \\
&\lesssim &\left( 1+t\right) ^{-4}\int_{0}^{\frac{t}{4}-\frac{\left \vert
x\right \vert }{4{\rm c}}}\left( 1+t-\tau \right) ^{-2}\left( 1+t-\tau
\right) ^{\frac{5}{2}}\left( 1+\tau \right) ^{-2}d\tau \\
&\lesssim &\left( 1+t\right) ^{-\frac{7}{2}}\lesssim \left( 1+t\right) ^{-\frac{1}{2}}\left(
1+\left \vert x\right \vert \right) ^{-3}\lesssim \left( 1+t\right)
^{-2}\bigg( 1+\frac{\left \vert x\right \vert ^{2}}{1+t}\bigg) ^{-\frac{3}{2}}\hbox{,}
\end{eqnarray*}%
since $\sqrt{1+t}\leq \left \vert x\right \vert \leq \frac{{\rm c}t}{2}$. Thus,
\begin{eqnarray*}
J_{1}\lesssim \left(
1+t\right) ^{-2}\Big(1+\frac{(|x|-{\rm c}t)^2}{1+t}\Big)^{-N}+\left( 1+t\right)
^{-2}\bigg( 1+\frac{\left \vert x\right \vert ^{2}}{1+t}\bigg) ^{-\frac{3%
}{2}}\hbox{.}
\end{eqnarray*}

For $J_{2}$, $\tau\geq\frac{t}{4}-\frac{\left \vert x\right \vert }{4%
{\rm c}}$ implies that
\begin{eqnarray*}
J_{2} &\lesssim &\int_{\frac{t}{4}-\frac{\left \vert x\right \vert }{4%
{\rm c}}}^{\frac{t}{2}}\int_{\mathbb{R}^{3}}\left( 1+t\right) ^{-2}\Big(1+\frac{(|x-y|-{\rm c}(t-\tau))^2}{1+t-\tau}\Big)^{-N}\left( 1+\tau \right)
^{-4}\bigg( 1+\frac{\left( \left \vert y\right \vert -{\rm c}\tau \right)
^{2}}{1+\tau }\bigg) ^{-2}dyd\tau \\
&\lesssim &\left( 1+t\right) ^{-2}\left( 1+\frac{t}{4}-\frac{\left \vert
x\right \vert }{4{\rm c}}\right) ^{-4}\int_{0}^{t}\int_{\mathbb{R}%
^{3}}\Big(1+\frac{(|x-y|-{\rm c}(t-\tau))^2}{1+t-\tau}\Big)^{-N}\bigg( 1+\frac{\left(
\left \vert y\right \vert -{\rm c}\tau \right) ^{2}}{1+\tau }\bigg)
^{-2}dyd\tau \\
&\lesssim &\left( 1+t\right) ^{-2}\left( 1+t\right) ^{-4}\left( 1+t\right)
^{3}\lesssim \left( 1+t\right) ^{-3}\lesssim \left( 1+t\right) ^{-\frac{3}{2}}\bigg(
1+\frac{\left \vert x\right \vert ^{2}}{1+t}\bigg) ^{-\frac{3}{2}}\hbox{,}
\end{eqnarray*}%
and for $J_3$, the relation $t-\tau\geq t-\frac{t}{2}-\frac{1}{4}(t+\frac{|x|}{{\rm c}})=\frac{{\rm c}t-|x|}{4c}>\frac{|x|}{2{\rm c}}$ gives that
\begin{eqnarray*}
J_{3} &\lesssim &\int_{\frac{t}{2}}^{\frac{t}{2}+\frac{1}{4}\left( t+\frac{%
\left \vert x\right \vert }{{\rm c}}\right) }\!\!\!\int_{\mathbb{R}^{3}}\left(
1\!+\!t\!-\!\tau \right) ^{-2}\Big(1\!+\!\frac{(|x-y|-{\rm c}(t-\tau))^2}{1+t-\tau}\Big)^{\!-N}\!\left( 1\!+\!t\right) ^{-4}\bigg( 1\!+\!\frac{\left( \left \vert y\right \vert -%
{\rm c}\tau \right) ^{2}}{1+\tau }\bigg) ^{-2}\!dyd\tau \\
&\lesssim &\Big( 1+\frac{|x|}{2{\rm c}}\Big)^{-2}\left( 1+t\right) ^{-4}\int_{0}^{t}\int_{%
\mathbb{R}^{3}}\Big(1+\frac{(|x-y|-{\rm c}(t-\tau))^2}{1+t-\tau}\Big)^{-N}\bigg( 1+%
\frac{\left( \left \vert y\right \vert -{\rm c}\tau \right) ^{2}}{1+\tau }%
\bigg) ^{-2}dyd\tau \\
&\lesssim &\Big( 1+\frac{|x|}{2c}\Big)^{-2}\left( 1+t\right) ^{-1}\lesssim (1+t)^{-\frac{3}{2}}\Big(1+\frac{|x|^2}{1+t}\Big)^{-\frac{3}{2}}.
\end{eqnarray*}

As for $J_{4}$, we decompose $\mathbb{R}^{3}$ into two parts%
\begin{equation*}
J_{4}=\int_{\frac{t}{2}+\frac{1}{4}\left( t+\frac{\left \vert x\right \vert
}{{\rm c}}\right) }^{t}\left( \int_{\left \vert y\right \vert \leq \frac{%
\left \vert x\right \vert +{\rm c}t}{2}}+\int_{\left \vert y\right \vert >%
\frac{\left \vert x\right \vert +{\rm c}t}{2}}\right) \left( \cdots
\right) dyd\tau =:J_{41}+J_{42}\hbox{.}
\end{equation*}%
If $\frac{t}{2}+\frac{1}{4}\big( t+\frac{\left \vert x\right \vert }{%
{\rm c}}\big) \leq \tau \leq t$ and $\left \vert y\right \vert \leq \frac{%
\left \vert x\right \vert +{\rm c}t}{2}$, then%
\begin{equation*}
{\rm c}\tau -\left \vert y\right \vert \geq \frac{{\rm c}t}{2}+\frac{1%
}{4}\left( {\rm c}t+\left \vert x\right \vert \right) -\frac{\left \vert
x\right \vert +{\rm c}t}{2}\geq \frac{{\rm c}t-\left \vert x\right
\vert }{4}\geq \frac{{\rm c}t}{8}\hbox{.}
\end{equation*}%
If $\frac{t}{2}+\frac{1}{4}\big( t+\frac{\left \vert x\right \vert }{%
{\rm c}}\big) \leq \tau \leq t$ and $\left \vert y\right \vert >\frac{%
\left \vert x\right \vert +{\rm c}t}{2}$, then%
\begin{eqnarray*}
\left \vert x-y\right \vert -{\rm c}\left( t-\tau \right) &\geq &\left
\vert y\right \vert -\left \vert x\right \vert -{\rm c}\left( t-\tau
\right) \geq \frac{\left \vert x\right \vert +{\rm c}t}{2}-\left \vert
x\right \vert -{\rm c}t+\frac{{\rm c}t}{2}+\frac{1}{4}\left( {\rm c}%
t+\left \vert x\right \vert \right) \\
&\geq &\frac{{\rm c}t-\left \vert x\right \vert }{4}\geq \frac{{\rm c}t%
}{8}\hbox{.}
\end{eqnarray*}%
Hence,
\begin{eqnarray*}
J_{41} &\lesssim &\int_{\frac{t}{2}+\frac{1}{4}\left( t+\frac{\left \vert
x\right \vert }{{\rm c}}\right) }^{t}\int_{\left \vert y\right \vert \leq
\frac{\left \vert x\right \vert +{\rm c}t}{2}}\left( 1+t-\tau \right)
^{-2}\Big(1+\frac{(|x-y|-{\rm c}(t-\tau))^2}{1+t-\tau}\Big)^{-N}\\
&&\cdot\left( 1+t\right)
^{-4}\bigg( 1+\frac{( {\rm c}t-\left \vert x\right \vert )^{2}%
}{1+t}\bigg) ^{-2}dyd\tau \\
&\lesssim &\left( 1+t\right) ^{-4}\bigg( 1+\frac{\left( {\rm c}t-\left
\vert x\right \vert \right) ^{2}}{1+t}\bigg) ^{-2}(1+t)^{\frac{3}{2}} \\
&\lesssim &\left( 1+t\right) ^{-\frac{9}{2}}\lesssim \left( 1+t\right) ^{-\frac{3}{2}}\left(
1+\left \vert x\right \vert \right) ^{-3}\lesssim \left( 1+t\right)
^{-3}\bigg( 1+\frac{\left \vert x\right \vert ^{2}}{1+t}\bigg) ^{-\frac{3%
}{2}}\hbox{,}
\end{eqnarray*}%
\begin{eqnarray*}
J_{42} &\lesssim &\int_{\frac{t}{2}+\frac{1}{4}\left( t+\frac{\left \vert
x\right \vert }{{\rm c}}\right) }^{t}\int_{\left \vert y\right \vert >%
\frac{\left \vert x\right \vert +{\rm c}t}{2}}\left( 1+t-\tau \right)
^{-2}\Big(1+\frac{(|x|-{\rm c}t)^2}{1+t}\Big)^{-N/2}\\
&& \cdot \Big(1+\frac{(|x-y|-{\rm c}(t-\tau))^2}{1+t-\tau}\Big)^{-N/2}\left( 1+t\right) ^{-4}dyd\tau \\
&\lesssim &\left( 1+t\right) ^{-4}\Big(1+\frac{(|x|-{\rm c}t)^2}{1+t}\Big)^{-N/2}\int_{\frac{t}{2}+%
\frac{1}{4}\left( t+\frac{\left \vert x\right \vert }{{\rm c}}\right)
}^{t}\left( 1+t-\tau \right) ^{-\frac{5}{2}+\frac{5}{2}}d\tau \\
&\lesssim &\left( 1+t\right) ^{-3}\Big(1+\frac{(|x|-{\rm c}t)^2}{1+t}\Big)^{-N/2}\hbox{.}
\end{eqnarray*}%

As a result,
\begin{eqnarray*}
K_7\lesssim \left( 1+t\right) ^{-\frac{3}{2}}\bigg( 1+\frac{\left \vert x\right \vert
^{2}}{1+t}\bigg) ^{-\frac{3}{2}}+\left( 1+t\right) ^{-2}\Big(1+\frac{(|x|-{\rm c}t)^2}{1+t}\Big)^{-N/2}
\end{eqnarray*}
for $\left( x,t\right) \in D_{4}$. 

\bigskip
\noindent\textbf{Case 4:} $\left( x,t\right) \in D_{5}$. We split the integral $K_7$
into five parts%
\begin{equation*}
K_7=\left( \int_{0}^{\frac{t}{4}-\frac{\left \vert x\right \vert }{4{\rm c}}%
}+\int_{\frac{t}{4}-\frac{\left \vert x\right \vert }{4{\rm c}}}^{\frac{t%
}{2}}+\int_{\frac{t}{2}}^{\frac{|x|}{2{\rm c}}+\frac{1}{4}(t+\frac{|x|}{%
{\rm c}})} +\int_{\frac{|x|}{2{\rm c}}+\frac{1}{4}(t+\frac{|x|}{%
{\rm c}})}^{\frac{t}{2}+\frac{1}{4}(t+\frac{|x|}{{\rm c}})} +\int^{t}_{%
\frac{t}{2}+\frac{1}{4}(t+\frac{|x|}{{\rm c}})}\right) \int_{\mathbb{R}%
^{3}}\left( \cdots \right) dyd\tau =:\sum_{i=1}^{5}J_{i}\hbox{.}
\end{equation*}

For $J_{1}$, we decompose $\mathbb{R}^{3}$ into two parts%
\begin{equation*}
J_{1}=\int_{0}^{\frac{t}{4}-\frac{\left \vert x\right \vert }{4{\rm c}}%
}\left( \int_{\left \vert y\right \vert \leq \frac{{\rm c}t-\left \vert
x\right \vert }{2}}+\int_{\left \vert y\right \vert >\frac{{\rm c}t-\left
\vert x\right \vert }{2}}\right) \left( \cdots \right) dyd\tau
=:J_{11}+J_{12}\hbox{.}
\end{equation*}%
If $0\leq \tau \leq \frac{t}{4}-\frac{\left \vert x\right \vert }{4{\rm c}%
}$ and $\left \vert y\right \vert \leq \frac{{\rm c}t-\left \vert
x\right
\vert }{2}$, then
\begin{equation*}
{\rm c}\left( t-\tau \right) -\left \vert x-y\right \vert \geq {\rm c}%
\left( t-\tau \right) -\left( \left \vert x\right \vert +\left \vert y\right
\vert \right) \geq \frac{{\rm c}t-\left \vert x\right \vert }{4}\hbox{.}
\end{equation*}%
If $0\leq \tau \leq \frac{t}{4}-\frac{\left \vert x\right \vert }{4{\rm c}%
}$ and $\left \vert y\right \vert >\frac{{\rm c}t-\left \vert x\right \vert
}{2}$, then 
$$
\left \vert y\right \vert -{\rm c}\tau \geq \frac{{\rm c}t-\left \vert
x\right \vert }{4}.
$$ 
Hence,%
\begin{eqnarray*}
J_{11} &\lesssim &\int_{0}^{\frac{t}{4}-\frac{\left \vert x\right \vert }{4%
{\rm c}}}\int_{\left \vert y\right \vert \leq \frac{{\rm c}t-\left
\vert x\right \vert }{2}}\left( 1+t\right) ^{-2}\Big(1+\frac{({\rm c}t-|x|)^2}{1+t}\Big)^{-N}\left(
1+\tau \right) ^{-4}\bigg( 1+\frac{\left( \left \vert y\right \vert -{\rm
c}\tau \right) ^{2}}{1+\tau }\bigg) ^{-2}dyd\tau \\
&\lesssim &\left( 1+t\right) ^{-2}\Big(1+\frac{({\rm c}t-|x|)^2}{1+t}\Big)^{-N}\int_{0}^{\frac{{\rm
c}t-\left \vert x\right \vert }{4{\rm c}}}\left( 1+\tau \right) ^{-4+\frac{5}{2}%
}d\tau \lesssim \left( 1+t\right) ^{-2}\Big(1+\frac{({\rm c}t-|x|)^2}{1+t}\Big)^{-N}\hbox{.}
\end{eqnarray*}%
For $J_{12}$, by using Young inequality, we have
\begin{eqnarray*}
J_{12} &\lesssim &\int_{0}^{\frac{t}{4}-\frac{\left \vert x\right \vert }{4%
{\rm c}}}\int_{\left \vert y\right \vert -%
{\rm c}\tau >\frac{\left \vert x\right \vert -{\rm c}t}{2}}\left(
1+t-\tau \right) ^{-2}\Big(1+\frac{(|x-y|-{\rm c}(t-\tau))^2}{1+t-\tau}\Big)^{-N}\\
&&\cdot \left( 1+\tau \right) ^{-4}\bigg( 1+\frac{\left( \left \vert x\right \vert -%
{\rm c}t \right) ^{2}}{1+\tau }\bigg) ^{-1}\bigg( 1+\frac{\left( \left \vert y\right \vert -%
{\rm c}\tau \right) ^{2}}{1+\tau }\bigg) ^{-1}dyd\tau \\
&\lesssim &(1+t)^{-2}\bigg( 1+\frac{\left( \left \vert x\right \vert -%
{\rm c}t \right) ^{2}}{1+t}\bigg) ^{-1}\int_{0}^{\frac{t}{4}-\frac{\left \vert x\right \vert }{4%
{\rm c}}}\int_{\left \vert y\right \vert -%
{\rm c}\tau >\frac{\left \vert x\right \vert -{\rm c}t}{2}}\left(
1+t-\tau \right) ^{-1}\\
&&\cdot \Big(1+\frac{(|x-y|-{\rm c}(t-\tau))^2}{1+t-\tau}\Big)^{-N}\left( 1+\tau \right) ^{-3}\left( 1+\frac{\left( \left \vert y\right \vert -%
{\rm c}\tau \right) ^{2}}{1+\tau }\right) ^{-1}dyd\tau \\
&\lesssim &(1+t)^{-2}\bigg( 1+\frac{\left( \left \vert x\right \vert -%
{\rm c}t \right) ^{2}}{1+t}\bigg) ^{-1}\int_{0}^{\frac{t}{4}-\frac{\left \vert x\right \vert }{4%
{\rm c}}}(1+t-\tau)^{-1}(1+t-\tau)(1+\tau)^{-3}(1+\tau)^{\frac{3}{2}}d\tau \\
&\lesssim &(1+t)^{-2}\bigg( 1+\frac{\left( \left \vert x\right \vert -%
{\rm c}t \right) ^{2}}{1+t}\bigg) ^{-1}.
\end{eqnarray*}%
Therefore, we get%
\begin{equation*}
J_{1}\lesssim (1+t)^{-2}\bigg( 1+\frac{\left( \left \vert x\right \vert -%
{\rm c}t \right) ^{2}}{1+t}\bigg) ^{-1}.
\end{equation*}

For $J_{2}$, we use the spherical coordinates to obtain

\begin{eqnarray*}
J_{2} &\lesssim &\int_{\frac{t}{4}-\frac{\left \vert x\right \vert }{4%
{\rm c}}}^{\frac{t}{2}}\int_{0}^{\infty }\int_{0}^{\pi }\left( 1+t-\tau
\right) ^{-2}\bigg(1+\frac{( \sqrt{\left \vert x\right \vert
^{2}+r^{2}-2r\left \vert x\right \vert \cos \theta }-{\rm c}\left( t-\tau
\right) ) ^{2}}{ 1+t-\tau}\bigg)^{-N}\\
&&\cdot\left( 1+\tau \right)
^{-4}\bigg( 1+\frac{\left( r-{\rm c}\tau \right) ^{2}}{1+\tau }\bigg)
^{-2}r^{2}\sin \theta d\theta drd\tau \\
&\lesssim &\int_{\frac{t}{4}-\frac{\left \vert x\right \vert }{4{\rm c}}%
}^{\frac{t}{2}}\int_{0}^{\infty }\int_{\left \vert \left \vert x\right \vert
-r\right \vert }^{\left \vert x\right \vert +r}\left( 1+t-\tau \right)
^{-2}\bigg(1+\frac{( \sqrt{\left \vert x\right \vert
^{2}+r^{2}-2r\left \vert x\right \vert \cos \theta }-{\rm c}\left( t-\tau
\right) ) ^{2}}{ 1+t-\tau}\bigg)^{-N}\\
&&\cdot\left( 1+\tau \right) ^{-4}\bigg( 1+\frac{%
\left( r-{\rm c}\tau \right) ^{2}}{1+\tau }\bigg) ^{-2}rz\frac{1}{\left
\vert x\right \vert }dzdrd\tau \\
&\lesssim &\int_{\frac{t}{4}-\frac{\left \vert x\right \vert }{4{\rm c}}%
}^{\frac{t}{2}}\int_{0}^{\infty }\int_{0}^{\infty }\left( 1+t-\tau \right)
^{-2}\bigg(1+\frac{( \sqrt{\left \vert x\right \vert
^{2}+r^{2}-2r\left \vert x\right \vert \cos \theta }-{\rm c}\left( t-\tau
\right) ) ^{2}}{ 1+t-\tau}\bigg)^{-N}\\
&&\cdot\left( 1+\tau \right) ^{-4}\bigg( 1+\frac{%
\left( r-{\rm c}\tau \right) ^{2}}{1+\tau }\bigg) ^{-2}rz\frac{1}{\left
\vert x\right \vert }dzdrd\tau \\
&\lesssim &\int_{\frac{t}{4}-\frac{\left \vert x\right \vert }{4{\rm c}}%
}^{\frac{t}{2}}\int_{0}^{\infty }\left( 1+t-\tau \right) ^{-2+%
\frac{3}{2}}\left( 1+\tau \right) ^{-4}\bigg( 1+\frac{\left( r-{\rm c}%
\tau \right) ^{2}}{1+\tau }\bigg) ^{-2}\frac{r}{\left \vert x\right \vert }%
drd\tau \\
&\lesssim &\left( 1+t\right) ^{-2+\frac{3}{2}}\left \vert x\right
\vert ^{-1}\int_{\frac{t}{4}-\frac{\left \vert x\right \vert }{4{\rm c}}%
}^{\frac{t}{2}}\left( 1+\tau \right) ^{-4+\frac{3}{2}}d\tau \\
&\lesssim &\left( 1+t\right) ^{-\frac{3}{2}}\left( 1+{\rm c}t-\left \vert x\right
\vert \right) ^{-\frac{3}{2}}\lesssim \left( 1+t\right) ^{-\frac{3}{2}}\left( 1+%
{\rm c}t-\left \vert x\right \vert \right) ^{\frac{1}{2}}\left( 1+{\rm
c}t-\left \vert x\right \vert \right) ^{-2} \\
&\lesssim &\left( 1+t\right)^{-2}\bigg( 1+\frac{\left( {\rm c}t-\left
\vert x\right \vert \right) ^{2}}{1+t}\bigg) ^{-1}\hbox{,}
\end{eqnarray*}%
where we have used the fact ${\rm c}t-|x|\geq\sqrt{1+t}$ in $D_5$.

For $J_{3}$, we use the spherical coordinates again to obtain%
\begin{eqnarray*}
J_{3} &\lesssim &\int_{\frac{t}{2}}^{\frac{3|x|}{4{\rm c}}+\frac{t}{4}%
}\int_{0}^{\infty }\int_{0}^{\pi }\left( 1+t-\tau \right) ^{-2}\bigg(1+\frac{( \sqrt{\left \vert x\right \vert
^{2}+r^{2}-2r\left \vert x\right \vert \cos \theta }-{\rm c}\left( t-\tau
\right) ) ^{2}}{ 1+t-\tau}\bigg)^{-N}\\
&&\cdot\left( 1+t\right) ^{-4}\bigg( 1+\frac{\left(
r-{\rm c}\tau \right) ^{2}}{1+\tau }\bigg) ^{-2}r^{2}\sin \theta d\theta
drd\tau \\
&\lesssim &\int_{\frac{t}{2}}^{\frac{3|x|}{4{\rm c}}+\frac{t}{4}%
}\int_{0}^{\infty }\int_{\left \vert \left \vert x\right \vert -r\right
\vert }^{\left \vert x\right \vert +r}\left( 1+t-\tau \right) ^{-2}\bigg(1+\frac{( \sqrt{\left \vert x\right \vert
^{2}+r^{2}-2r\left \vert x\right \vert \cos \theta }-{\rm c}\left( t-\tau
\right) ) ^{2}}{ 1+t-\tau}\bigg)^{-N}\\
&&\cdot\left( 1+t\right) ^{-4}\bigg( 1+\frac{\left( r-{\rm c}%
\tau \right) ^{2}}{1+\tau }\bigg) ^{-2}rz\frac{1}{\left \vert x\right \vert
}dzdrd\tau \\
&\lesssim &\left( 1+t\right) ^{-4}\left \vert x\right \vert ^{-1}\int_{\frac{%
t}{2}}^{\frac{3|x|}{4{\rm c}}+\frac{t}{4}} \int_{0}^{\infty }\left(
1+t-\tau \right) ^{-2+\frac{3}{2}}\bigg( 1+\frac{\left( r-{\rm c}\tau \right) ^{2}}{1+\tau }\bigg) ^{-2}rdrd\tau \\
&\lesssim &\left( 1+t\right) ^{-\frac{7}{2}}\int_{\frac{t}{2}}^{\frac{3|x|}{4{\rm
c}}+\frac{t}{4}}\left( 1+t-\tau \right) ^{-\frac{1}{2}}d\tau \\
&\lesssim &\left( 1+t\right) ^{-\frac{7}{2}}\bigg[\bigg(1+\frac{t}{2}\bigg)^{\frac{1}{2}}-
\bigg(1+\frac{{\rm c}t-|x|}{4{\rm c}}\bigg)^{\frac{1}{2}}\bigg] \\
&\lesssim& \left( 1+t\right) ^{-3}\lesssim\left( 1+t\right) ^{-2}\bigg( 1+\frac{\left( {\rm c}t-\left
\vert x\right \vert \right) ^{2}}{1+t}\bigg) ^{-1}\hbox{.}
\end{eqnarray*}
For $J_{4}$, similar to $J_{3}$, we have
\begin{eqnarray*}
J_{4} &\lesssim &\left( 1+t\right) ^{-\frac{7}{2}}\int_{\frac{3|x|}{4{\rm c}}+%
\frac{t}{4}}^{\frac{|x|}{4{\rm c}}+\frac{3t}{4}}\left( 1+t-\tau \right)
^{-\frac{1}{2}}d\tau \\
&\lesssim &\left( 1+t\right) ^{-\frac{7}{2}}\bigg[\left(1+\frac{3({\rm c}t-|x|)}{4%
{\rm c}}\right)^{\frac{1}{2}} -\left(1+\frac{({\rm c}t-|x|)}{2{\rm c}}\right)^{\frac{1}{2}}%
\bigg]\lesssim (1+t)^{-3}\\
&\lesssim& \left( 1+t\right) ^{-2}\bigg( 1+\frac{\left( {\rm c}t-\left
\vert x\right \vert \right) ^{2}}{1+t}\bigg) ^{-1}\hbox{.}
\end{eqnarray*}
For $J_{5}$, we decompose $\mathbb{R}^{3}$ into two parts%
\begin{equation*}
J_{5}=\int_{\frac{t}{2}+\frac{1}{4}\left( t+\frac{\left \vert x\right \vert
}{{\rm c}}\right) }^{t}\left( \int_{\left \vert y\right \vert \leq \frac{%
\left \vert x\right \vert +{\rm c}t}{2}}+\int_{\left \vert y\right \vert >%
\frac{\left \vert x\right \vert +{\rm c}t}{2}}\right) \left( \cdots
\right) dyd\tau =:J_{51}+J_{52}\hbox{.}
\end{equation*}
If $\frac{t}{2}+\frac{1}{4}\left( t+\frac{\left \vert x\right \vert }{%
{\rm c}}\right) \leq \tau \leq t$ and $\left \vert y\right \vert \leq \frac{%
\left \vert x\right \vert +{\rm c}t}{2}$, then%
\begin{equation*}
{\rm c}\tau -\left \vert y\right \vert \geq \frac{{\rm c}t}{2}+\frac{1%
}{4}\left( {\rm c}t+\left \vert x\right \vert \right) -\frac{\left \vert
x\right \vert +{\rm c}t}{2}\geq \frac{{\rm c}t-\left \vert x\right
\vert }{4}\hbox{.}
\end{equation*}%
If $\frac{t}{2}+\frac{1}{4}\left( t+\frac{\left \vert x\right \vert }{%
{\rm c}}\right) \leq \tau \leq t$ and $\left \vert y\right \vert >\frac{%
\left \vert x\right \vert +{\rm c}t}{2}$, then%
\begin{eqnarray*}
\left \vert x-y\right \vert -{\rm c}\left( t-\tau \right) &\geq &\left
\vert y\right \vert -\left \vert x\right \vert -{\rm c}\left( t-\tau
\right) \geq \frac{\left \vert x\right \vert +{\rm c}t}{2}-\left \vert
x\right \vert -{\rm c}t+\frac{{\rm c}t}{2}+\frac{1}{4}\left( {\rm c}%
t+\left \vert x\right \vert \right) \geq \frac{{\rm c}t-\left \vert
x\right \vert }{4}\hbox{.}
\end{eqnarray*}%
Hence,
\begin{eqnarray*}
J_{51} &\lesssim &\int_{\frac{t}{2}+\frac{1}{4}\left( t+\frac{\left \vert
x\right \vert }{{\rm c}}\right) }^{t}\int_{\left \vert y\right \vert \leq
\frac{\left \vert x\right \vert +{\rm c}t}{2}}\left( 1+t-\tau \right)
^{-2}\Big(1+\frac{(|x-y|-{\rm c}(t-\tau))^2}{1+t-\tau}\Big)^{-N}\\
&&\cdot\left( 1+t\right)
^{-4}\bigg( 1+\frac{\left( {\rm c}t-\left \vert x\right \vert \right) ^{2}%
}{1+t}\bigg) ^{-1}\left( 1+\frac{\left( {\rm c}\tau-\left \vert y\right \vert \right) ^{2}%
}{1+t}\right) ^{-1}dyd\tau \\
&\lesssim &\left( 1+t\right) ^{-4}\bigg( 1+\frac{\left( {\rm c}t-\left
\vert x\right \vert \right) ^{2}}{1+t}\bigg) ^{-1}\int_{\frac{t}{2}+\frac{1}{4}\left( t+\frac{\left \vert
x\right \vert }{{\rm c}}\right) }^{t}(1+t-\tau)^{-2}(1+t-\tau)^{\frac{5}{4}}(1+\tau)^{\frac{3}{4}}d\tau\\
&\lesssim &\left( 1+t\right) ^{-3}\bigg( 1+\frac{\left( {\rm c}t-\left
\vert x\right \vert \right) ^{2}}{1+t}\bigg) ^{-1}\hbox{,}
\end{eqnarray*}%
\begin{eqnarray*}
J_{52} &\lesssim &\int_{\frac{t}{2}+\frac{1}{4}\big( t+\frac{\left \vert
x\right \vert }{{\rm c}}\big) }^{t}\int_{\left \vert y\right \vert >%
\frac{\left \vert x\right \vert +{\rm c}t}{2}}\left( 1+t-\tau \right)
^{-2}\Big(1+\frac{({\rm c}t-|x|)^2}{1+t}\Big)^{-N/2}\\
&&\cdot\Big(1+\frac{(|x-y|-{\rm c}(t-\tau))^2}{1+t-\tau}\Big)^{-N/2}\left( 1+t\right) ^{-4}dyd\tau \\
&\lesssim &\left( 1+t\right) ^{-4}\Big(1+\frac{({\rm c}t-|x|)^2}{1+t}\Big)^{-N/2}\int_{\frac{t}{2}+\frac{1}{4}\left( t+\frac{%
\left \vert x\right \vert }{{\rm c}}\right) }^{t}\left( 1+t-\tau \right)
^{-2+\frac{5}{2}}d\tau \\
&\lesssim &\left( 1+t\right) ^{-\frac{5}{2}}\Big(1+\frac{({\rm c}t-|x|)^2}{1+t}\Big)^{-N/2}\hbox{.}
\end{eqnarray*}%
Gathering all the estimates, we have
\begin{equation*}
K_7\lesssim \left( 1+t\right) ^{-\frac{3}{2}}\bigg( 1+\frac{\left \vert x\right \vert
^{2}}{1+t}\bigg) ^{-\frac{3}{2}}+\left( 1+t\right) ^{-2}\bigg( 1+\frac{%
\left( {\rm c}t-\left \vert x\right \vert \right) ^{2}}{1+t}\bigg) ^{-1}%
\hbox{,}
\end{equation*}  
and have completed the proof of (\ref{6.5}). \ \ \ \ \ \ \ \ \ \ \ \ \ \ \ \ \ \ \ \ \ \ \ \ \ \ \ \ \ \ \ \ \ \ \ \ \ \ \ \ \ \ \ \ \ \ \ \ \ \ \ \ \ \ \ \ \ \ \ \ \ \ \ \ \ \ \ \ \ \ \ \ \ \ \ \ \textsquare

The following two lemmas are usually employed to derive space-time estimates of the fluid models. In fact, Lemma \ref{A.6} is used to deal with the term with -1-order Riesz operator in low frequency of Green funcion, Lemma \ref{A.7} is used to deal with the terms containing double Riesz operator with the symbol $\frac{\xi\xi^T}{|\xi|^2}$ in low frequency part of Green's function. For readers' convenience, we write down the proof since it was just stated in the previous works without a detailed proof.
\begin{lemma}\label{A.6}\cite{Wu4}
If $t\in\mathbb{R}^+$, $x\in\mathbb{R}^n$ with $n\geq2$ and $|f(x,t)|\lesssim \big(1+\frac{|x|^2}{1+t}\big)^{-r}$ with $r>\frac{n}{2}$, then
\begin{equation}\label{8.5}
\Big|\nabla (-\Delta)^{-1}f(x,t)\Big|\lesssim (1+t)^{\frac{1}{2}}\Big(1+\frac{|x|^2}{1+t}\Big)^{-\frac{n-1}{2}}.
\end{equation}
\end{lemma}

\begin{lemma}\label{A.7}\cite{Wu4}
Let $t>0$ and $x\in\mathbb{R}^n$ with $n\geq2$. Suppose that $f(x,t)$ satisfies
 \begin{equation}\label{8.10}
 \begin{array}{rl}
&\ \ \ \ \ |f(x,t)|\lesssim \big(1+\frac{|x|^2}{1+t}\big)^{-r_1},\ \ \ \ \ \ \ \ \ r_1>\frac{n}{2},\\
&|\nabla f(x,t)|\lesssim (1+t)^{-\frac{1}{2}}\big(1+\frac{|x|^2}{1+t}\big)^{-r_2},\ \ \ r_2>\frac{n+1}{2}.
\end{array}
\end{equation}
Then it holds that
 \begin{equation}\label{8.11}
\Big|\nabla{\rm div}(-\Delta)^{-1}f(x,t)\Big|\lesssim  \Big(1+\frac{|x|^2}{1+t}\Big)^{-\frac{n}{2}}.
\end{equation}
\end{lemma}

\begin{corollary}\label{A.8}\cite{Wu4}
Suppose that $\xi\in\mathbb{R}^n$ with $n\geq2$ and $\hat{f}(\xi,t)=\frac{\xi\xi^T}{|\xi|^2}\chi_1(\xi)e^{-C|\xi|^2t+\mathcal{O}(|\xi|^3)t}$ in the Fourier space, where $\chi_1(\xi)$ is the cutoff function for the low frequency.   Then, it holds that
\begin{equation}\label{8.18}
|\mathcal{F}^{-1}(\xi^\alpha\hat{f}(\xi,t))|\lesssim (1+t)^{-\frac{n+|\alpha|}{2}}\big(1+\frac{|x|^2}{1+t}\big)^{-\frac{n+|\alpha|}{2}}.
\end{equation}
\end{corollary}

\section*{Acknowledgments}
Z.G. Wu was supported by Guangxi Natural Science Foundation $\#$2025JJA110146. W.K. Wang was supported by 
National Natural Science Foundation of China $\#$12271357 and Shanghai Science and Technology 
Innovation Action Plan $\#$21JC1403600. Y.H. Zhang was supported
by Guangxi Natural Science Foundation $\#$2025JJF110001, $\#$2024GXNSFDA010071, National Natural Science Foundation of China $\#$12271114, Center for Applied Mathematics of Guangxi (Guangxi Normal University), and the Key Laboratory of Mathematical Model and Application (Guangxi
Normal University), Education Department of Guangxi Zhuang Autonomous Region.

\end{document}